\documentclass[12pt,a4paper,twoside,reqno]{amsart}
\usepackage[utf8]{inputenc}
\usepackage[margin=1.1in]{geometry} 
\usepackage{amssymb}
\usepackage{setspace}
\usepackage{chngcntr}
\usepackage{apptools}
\usepackage{amsthm}
\usepackage{amsmath}
\usepackage{amsfonts}
\usepackage{amssymb}
\usepackage{mathrsfs} 
\usepackage{bbm}
\usepackage{url}
\usepackage{graphicx,tikz}
\usepackage[makeroom]{cancel}
\usepackage{enumerate}
\usepackage{scalerel}
\usepackage{hyperref}
\usepackage{longtable}
\allowdisplaybreaks
\newtheorem{Theorem}{Theorem}

\newtheorem{Proposition}[Theorem]{Proposition}
\newtheorem{Corollary}[Theorem]{Corollary}
\newtheorem{Lemma}[Theorem]{Lemma}
\newtheorem{Claim}[Theorem]{Claim}
\newtheorem{Example}[Theorem]{Example}

\newtheorem{Definition}[Theorem]{Definition}
\newtheorem{Remark}[Theorem]{Remark}

\newcommand{\ba}{\boldsymbol{\alpha}}
\newcommand{\bo}{\boldsymbol{\omega}}
\newcommand{\bc}{\boldsymbol{\chi}}
\newcommand{\degen}{\operatorname{degen}}

\newcommand{\rank}{\operatorname{rank}}

\newcommand{\dist}{\operatorname{dist}}

\newcommand{\unif}{\operatorname{unif}}
\newcommand{\spn}{\operatorname{span}}
\newcommand{\im}{\operatorname{im}}

\counterwithin{Theorem}{section}
\numberwithin{equation}{section}

\AtEndDocument{\bigskip{\footnotesize%
  \textsc{Aled Walker, Trinity College, Cambridge, United Kingdom, CB2 1TQ} \\ 
\texttt{aw530@cam.ac.uk}}}

\title{Gowers norms control diophantine inequalities}
\author{Aled Walker}
\begin{document}

\begin{abstract}
A central tool in the study of systems of linear equations with integer coefficients is the Generalised von Neumann Theorem of Green and Tao. This theorem reduces the task of counting the weighted solutions of these equations to that of counting the weighted solutions for a particular family of forms, the Gowers norms $\Vert f \Vert_{U^{s+1}[N]}$ of the weight $f$. In this paper we consider systems of linear inequalities with real coefficients, and show that the number of solutions to such weighted diophantine inequalities may also be bounded by Gowers norms. Furthermore, we provide a necessary and sufficient condition for a system of real linear forms to be governed by Gowers norms in this way. We present applications to cancellation of the M\"{o}bius function over certain sequences. 

The machinery developed in this paper can be adapted to the case in which the weights are unbounded but suitably pseudorandom, with applications to counting the number of solutions to diophantine inequalities over the primes. Substantial extra difficulties occur in this setting, however, and we have prepared a separate paper on these issues. 
\end{abstract}
\maketitle
\tableofcontents
\section{Introduction}
\label{Sec introduction}

The field of \emph{diophantine inequalities} is a large, and somewhat loosely-defined, collection of problems which lie at the intersection of traditional number theory and diophantine approximation. As far as this paper is concerned, we will restrict our attention to the following class of questions. Let $A$ be a set of integers, let $\varepsilon$ be a positive parameter, and let $L$ be an $m$-by-$d$ real matrix. One might then ask whether there are infinitely many solutions to 
\begin{equation}
\label{the general inequality}
\Vert L\mathbf{a}\Vert_{\infty}\leqslant \varepsilon
\end{equation}
\noindent in which all of the coordinates of $\mathbf{a}$ lie in $A$. Further, letting $N$ be an integer parameter tending to infinity, one might seek an asymptotic formula for the number of such solutions $\mathbf{a}$ which lie in the box given by the condition $\Vert \mathbf{a} \Vert_\infty \leqslant N$. One might even try to count solutions in certain cases in which $L$ depends on $N$.  

Much is known about these problems for certain special sets $A$ (see \cite{Ba67, Ba86, DaHe46, Ma87, Mu05, Pa02, Parsell02}), in particular for the image sets of polynomials. This work is discussed in Section \ref{Section Classical results} below. However, as far we are aware, the inequality (\ref{the general inequality}) has not been considered before in such generality. Naturally there are some advantages and some disadvantages in pursuing such a general formulation, the main disadvantage being that the statements of our main results must perforce include some complicated technical hypotheses on the matrix $L$.

 It will take us the rest of Sections 1 and 2 to properly motivate these hypotheses, culminating in the statement of Theorem \ref{Main Theorem chapter 3} (our main theorem). Section \ref{Sec introduction} will focus on qualitative results and applications, whereas Section \ref{section history} goes on to explore the issues of diophantine approximation and non-degeneracy which are required for a quantitative treatment when $L$ depends on $N$. At the end of Section \ref{section history} we will give a detailed sketch of our entire proof strategy. For now, we present the reader with a certain corollary of our main theorem, which we hope will encourage further reading through this long introduction.

\begin{Corollary}[Example of M\"{o}bius orthogonality]
\label{Corollary irrational szem mobius}
Let $\theta_1,\dots,\theta _s \in \mathbb{R}$ be distinct irrational numbers, let $N$ be an integer parameter, and let $f_1,f_2,\dots,f_{s+1}: \{ 1,\dots,N\} \longrightarrow [-1,1]$ be arbitrary $1$-bounded functions. Then 
\begin{equation}
\frac{1}{N^2} \sum\limits_{\substack{x,d \in \mathbb{Z} \\ 1 \leqslant x \leqslant N}} \mu(x) f_1(x+d)\Big (\prod\limits_{i=2}^{s+1} f_i([x + \theta_{i-1} d]) \Big ) = o(1)
\end{equation}
\noindent as $N \rightarrow \infty$, where $\mu$ denotes the M\"{o}bius function and $[x] : = \lfloor x + \frac{1}{2}\rfloor$ is the nearest integer to $x$. The $o(1)$ error term may depend on the numbers $\theta_1,\dots,\theta_s$ but is independent of the choice of functions $f_1,\dots,f_{s+1}$.\\  
\end{Corollary}

\subsection{Classical results}
\label{Section Classical results}
As we said above, much is known about the inequality (\ref{the general inequality}) for certain special sets $A$, particularly when $m=1$. For example, if $A$ is the set of squares, it was shown by Davenport and Heilbronn in \cite{DaHe46} that there are infinitely many solutions to (\ref{the general inequality}) for $m=1$ and $d=5$, i.e. infinitely many solutions to 
\begin{equation}
\label{eq squares}
\vert \lambda_1 n_1^2 + \lambda_2 n_2^2 + \lambda_3 n_3^2 + \lambda_4 n_4^2 + \lambda_5 n_5^2 \vert \leqslant \varepsilon,
\end{equation} provided the coefficients $\lambda_i$ are non-zero, not all of the same sign, and not all in pairwise rational ratio. Their work also proves the same result for $k^{th}$ powers, provided that the number of variables is at least $2^k + 1$. Some 55 years after Davenport and Heilbronn, Freeman \cite{Fr02} refined the analysis from \cite{DaHe46} to obtain asymptotic formulas for the number of solutions to (\ref{eq squares}) in which $n_i\leqslant N$ for every $i$, and he also reduced the number of variables required in the case of $k^{th}$ powers, to $k(\log k + \log\log k + O(1))$. In \cite{Wo03} Wooley further reduced this number, particularly for small $k$. 

The Davenport-Heilbronn method is Fourier-analytic. One begins by replacing the interval $[-\varepsilon,\varepsilon]$ with a Lipschitz cut-off function, and then one expresses the solution count via the Fourier inversion formula (see \cite[Chapter 20]{Da05} or \cite[Chapter 11]{Va97}). The device of replacing $[-\varepsilon,\varepsilon]$ with a friendlier cut-off plays an important role in our argument too, and we discuss it at length in Section \ref{Section proof strategy}.

There are also some results on the inequality (\ref{the general inequality}) when $m\geqslant 2$, although this setting has been studied less intensively. For example, Parsell \cite{Parsell02} considered the setting of $k^{th}$ powers, with  M\"{u}ller \cite{Mu05} developing a refined result in the case of inequalities for general real quadratics. Parsell's result is rather technical to state, and we defer the interested reader to the original paper. Later on in Section \ref{section history}, however, we will state M\"{u}ller's result precisely, as one of his hypotheses is closely related to a hypothesis in our main theorem. \\

One of our main goals, for this paper and for our follow-up \cite{Wa19}, is to find a method of proving asymptotic formulae for the number of solutions to diophantine inequalities which goes beyond what can be done using the Davenport-Heilbronn method. Of particular interest to us is the case of inequality (\ref{the general inequality}) when $A$ is the set of prime numbers. A result first claimed in \cite{Ba67} by A. Baker\footnote{In fact Baker proved a slightly different result, writing in \cite{Ba67} that the result we quote here followed easily from the then-existing methods.} states that for any fixed positive $\varepsilon$ there exist infinitely many triples of primes $(p_1,p_2,p_3)$ satisfying 
\begin{equation}
\label{the type of inequality we're talking about}
\vert \lambda_1p_1 + \lambda_2p_2 + \lambda_3 p_3\vert \leqslant \varepsilon,
\end{equation}
\noindent assuming again that the coefficients $\lambda_i$ are non-zero, not all of the same sign, and not all in pairwise rational ratio. Parsell \cite{Pa02} then used a similar refinement to that of Freeman to prove a lower bound on the number of solutions to (\ref{the type of inequality we're talking about}) in the box $p_1,p_2,p_3\leqslant N$. For $m$ simultaneous inequalities, and for a generic matrix $L$, Parsell's method is powerful enough\footnote{although this doesn't seem to be present in the literature except in an appendix of our paper \cite{Wa19}.}  to prove an asymptotic formula for the number of solutions to (\ref{the general inequality}) in prime variables $p_1,p_2,\dots,p_d \leqslant N$, provided that $d\geqslant 2m+1$. In \cite{Wa19}, building on the work of the present paper, we manage to reach the same conclusion under the weaker hypothesis that $d\geqslant m+2$, provided that $L$ has algebraic coefficients. \\

A discussion of the literature on diophantine inequalities would not be complete without at least making reference to Margulis's famous resolution \cite{Ma87} of the Oppenheim Conjecture. With this work Margulis reduced the number of variables required to show the existence of infinitely many solutions to the inequality (\ref{eq squares}) from $5$ to $3$. Margulis's approach used dynamical methods, and is rather different in flavour to anything in this paper. In particular this method does not provide an asymptotic formula for the number of solutions in which the variables are bounded in a box. \\

\subsection{Notation}
\label{Section Notation}
Before continuing with the rest of our introduction, we feel that, given the technical nature of some of the statements to follow, it is prudent to fix all our notation at the outset. 

We will use standard asymptotic notation $O$, $o$, and $\Omega$. We do not, as is sometimes the convention, for a function $f$ and a positive function $g$ choose to write $f = O(g)$ if there exists a constant $C$ such that $\vert f(N)\vert \leqslant C g(N) $ \emph{for $N$ sufficiently large}. Rather we require the inequality to hold for all $N$ in some pre-specified range. If $N$ is a natural number, the range is always assumed to be $\mathbb{N}$ unless otherwise specified. For us, $0\notin \mathbb{N}$. 

It will be a convenient shorthand to use these symbols in conjunction with minus signs. So, by convention, we determine that expressions such as $-O(1), -o(1), -\Omega(1)$ are negative, e.g. $N^{-\Omega(1)}$ refers to a term $N^{-c}$, where $c$ is some positive quantity bounded away from $0$ as the asymptotic parameter tends to infinity. It will also be convenient to use the Vinogradov symbol $\ll$, where for a function $f$ and a positive function $g$ we write $f\ll g$ if and only if $f = O(g)$. We write $f\asymp g$ if $f\ll g$ and $g\ll f$. We also adopt the $\kappa$ notation from \cite{GT10}: $\kappa(x)$ denotes any quantity that tends to zero as $x$ tends to zero, with the exact value being permitted to change from line to line.

All the implied constants may depend on the dimensions of the underlying spaces. These will be obvious in context, and will always be denoted by $m$, $d$, $h$, or $s$ (or, in the case of Proposition \ref{normal form algorithm}, by $n$). If an implied constant depends on other parameters, we will denote these by subscripts, e.g. $O_{c,C,\varepsilon}(1)$, or $f\asymp_{\varepsilon} g$. By notation such as $o_\rho(1)$ we mean a term which tends to zero as the asymptotic parameter tends to infinity with $\rho$ fixed. 

If $N$ is a natural number, we use $[N]$ to denote $\{n\in \mathbb{N}:n\leqslant N\}$, whereas $[1,N]$ will be reserved for the closed real interval. For $x \in \mathbb{R}$, we write $[x]: = \lfloor x+\frac{1}{2}\rfloor$ for the nearest integer to $x$, and $\Vert x \Vert$ for $\vert x - [x] \vert$. This means that there is slight overloading of the notation $[N]$, but the sense will always be obvious in context. When other norms are present, we may write $\Vert x\Vert_{\mathbb{R}/\mathbb{Z}}$ for $\Vert x \Vert$ to avoid confusion. For $\mathbf{x} \in \mathbb{R}^m$, we let $\Vert \mathbf{x}\Vert_{\mathbb{R}^m/\mathbb{Z}^m}$ denote  $\sup_i \vert x_i - [x_i]\vert$.

If $X,Y\subset\mathbb{R}^d$ for some $d$, we define \[\operatorname{dist}(X,Y): = \inf\limits_{x\in X, y\in Y} \Vert x - y\Vert_{\infty}.\] If $X$ is the singleton $\{x\}$, we write $\operatorname{dist}(x,Y)$ for $\operatorname{dist}(\{x\},Y)$. By identifying sets of $m$-by-$d$ matrices with subsets of $\mathbb{R}^{md}$ (by identifying the coefficients of the matrices with coordinates in $\mathbb{R}^{md}$), we may also define $\dist(X,Y)$ when $X$ and $Y$ are sets of matrices of the same dimensions. We will consider a linear map $L:\mathbb{R}^d \longrightarrow \mathbb{R}^m$ to be synonymous with the $m$-by-$d$ matrix that represents $L$ with respect to the standard bases. The norm $\Vert L\Vert_\infty$ will refer to the maximum absolute value of the coefficients of this matrix. We use the notation $L^*: (\mathbb{R}^m)^* \longrightarrow (\mathbb{R}^d)^*$ for the dual linear map between the dual spaces. For a set $U \subset \mathbb{R}^d$ we use $U^{0}$ to denote the annihilator of $U$, i.e. the set of all $f$ in the dual space $(\mathbb{R}^d)^*$ for which $f|_U \equiv 0$.

We let $\partial(X)$ denote the topological boundary of a set $X \subset \mathbb{R}^d$. Given $S \subset \mathbb{R}$ and $\lambda \in \mathbb{R}$, we let $\lambda S: = \{ x \in \mathbb{R}: \exists s \in S \, \text{for which} \, \lambda s = x\}.$ If $A$ and $B$ are two sets with $A\subseteq B$, we let $1_A:B\longrightarrow \{0,1\}$ denote the indicator function of $A$. The relevant set $B$ will usually be obvious from context. The notation for logarithms, $\log$, will always denote the natural log. For $\theta \in \mathbb{R}$ we also adopt the standard shorthand $e(\theta)$ to mean $e^{2\pi i \theta}$.

In Section \ref{section General proof of the real variable von Neumann Theorem}, if $\mathbf{x} \in \mathbb{R}^d$ and if $a$ and $b$ are two subscripts with $1\leqslant a\leqslant b\leqslant d$, we use the notation $\mathbf{x_a^b}$ to denote the vector $(x_a,x_{a+1},\dots,x_{b})^T \in \mathbb{R}^{b-a+1}$.

\subsection{The main corollary}
We will now begin the process of developing our first main result, namely Corollary \ref{much easier to state}. This result is the first to link diophantine inequalities, such as (\ref{the general inequality}), to \emph{Gowers norms}. 

Given natural numbers $N$ and $d$, and a function $f:[N] \longrightarrow \mathbb{C}$, the \emph{Gowers $U^d$ norm} $\Vert f\Vert_{U^{d}[N]}$ was introduced into the literature around twenty years ago, as part\footnote{Gowers worked over the cyclic group $\mathbb{Z}/N\mathbb{Z}$ rather than $[N]$, but this is a very minor difference.} of Gowers' proof of Szemerédi's Theorem \cite{Go01}. The $U^d$ norms are genuine norms for $d \geqslant 2$, with $\Vert f \Vert_{U^d[N]}$ measuring the density of certain linear patterns weighted by $f$. Their presence in analytic number theory is by now well established (see for instance \cite{GT08, GT10, TT18, TT19}), but, to help any readers who are unfamiliar with these norms, in Appendix \ref{sec.Gowers norms} we have given a summary of the necessary definitions and basic notions. \\

Our present endeavour is motivated by one particular application of Gowers norms, namely the so-called `Generalised von Neumann Theorem' developed by Green and Tao in \cite{GT08, GT10} to study linear equations with rational coefficients. 

\begin{Theorem}[Generalised von Neumann Theorem for rational forms (non-quantitative)]
\label{Linear equations in bounded functions non quantitative}
Let $m,d$ be natural numbers, satisfying $d\geqslant m+2$. Let $L$ be an $m$-by-$d$ real matrix with integer coefficients, with rank $m$. Suppose that there does not exist any non-zero row-vector in the row-space of $L$ that has two or fewer non-zero coordinates. Then there is some natural number $s$ at most $d-2$ that satisfies the following. Let $N$ be an integer parameter, let $f_1,\dots,f_d:[N]\longrightarrow [-1,1]$ be arbitrary functions, and suppose that \[ \min_j \Vert f_j\Vert_{U^{s+1}[N]} \leqslant \rho \] for some parameter $\rho$ in the range $0<\rho \leqslant 1$. Then $$\frac{1}{N^{d-m}}\sum\limits_{\substack{\mathbf{n}\in [N]^d\\ L\mathbf{n} = \mathbf{0}}} \prod\limits_{j=1}^d f_j(n_j) \ll_L  \rho^{\Omega(1)} + o_\rho(1)$$ as $N\rightarrow \infty$.
\end{Theorem}

Results similar to Theorem \ref{Linear equations in bounded functions non quantitative} are central to Green-Tao's approach to counting solutions to linear equations in primes \cite{GT10}. It seems reasonable to hope that, if one could combine Gowers norms and diophantine inequalities in a suitable way, then one might be able to develop a strong understanding of linear inequalities in primes. As we have already intimated in Section \ref{Section Classical results}, when describing our improvements to Parsell's results, this can be done. However, many additional technical difficulties occur for the primes, as the von Mangoldt function is unbounded; we have chosen to present a separate work on these issues \cite{Wa19}.

We should briefly discuss the non-degeneracy condition on $L$ in the statement of Theorem \ref{Linear equations in bounded functions non quantitative}, namely that `there does not exist any non-zero row-vector in the row-space of $L$ that has two or fewer non-zero coordinates', as it may seem a little unnatural at first sight.\footnote{For readers who are already familiar with the notion of Cauchy-Schwarz complexity, imposing this non-degeneracy condition on $L$ is equivalent to insisting that $\ker L$ may be parametrised by a system of linear forms with finite Cauchy-Schwarz complexity.} Suppose that such a row-vector existed. Suppose also that it is the coordinates at index $i$ and index $j$ which are non-zero. Then, by short linear algebra argument (see Proposition \ref{Proposition easy degeneracy relation}), for any linear parametrisation $(\psi_1,\dots, \psi_d) = \Psi: \mathbb{R}^{d-m} \longrightarrow \ker L$,  $\psi_i$ is a multiple of $\psi_j$. Such a coupling between the coordinates has dire consequences for any averaging approach built upon the independence of the different coordinates, such as the averaging in Gowers norms, and so this coupling must be precluded. We will present a rigorous version of this principle in the context of linear inequalities, in Theorem \ref{Converse to main theorem} below.

Regarding the condition $d \geqslant m+2$, if $L$ has rank $m$ and $d \leqslant m+1$ then in fact, as follows Gaussian elimination, there must always exist a non-zero vector in the row space of $L$ with two or fewer non-zero coordinates. Thus, the condition $d \geqslant m+2$ is a necessary one if the coordinates $\ker L$ are to be suitably independent.

\begin{Remark}
\emph{Theorem \ref{Linear equations in bounded functions non quantitative} is implicit in \cite{GT10}, but there is no explicit such statement presented there, as those authors were focused on results over the primes. We will describe how to extract Theorem \ref{Linear equations in bounded functions non quantitative} from the arguments of \cite{GT10}, but we postpone this task until Section \ref{section Reductions}, as at that point we will also introduce a quantitative version (this will be Theorem \ref{Linear equations in bounded functions}).}
\end{Remark}

Our first main result is a version of Theorem \ref{Linear equations in bounded functions non quantitative} for diophantine inequalities. 

\begin{Corollary}[Gowers norms control diophantine inequalities (non-quantitative)]
\label{much easier to state}
Let $m,d$ be natural numbers, satisfying $d\geqslant m+2$, and let $\varepsilon$ be a positive parameter. Let $L:\mathbb{R}^d \longrightarrow \mathbb{R}^m$ be an $m$-by-$d$ real matrix, with rank $m$. Suppose that there does not exist any non-zero row-vector in the row-space of $L$ that has two or fewer non-zero coordinates. Then there is some natural number $s$ at most $d-2$, independent of $\varepsilon$, such that the following is true. Let $N$ be an integer parameter and let $f_1,\dots,f_d:[N]\longrightarrow [-1,1]$ be arbitrary functions, and suppose that \[ \min_j \Vert f_j \Vert_{U^{s+1}[N]} \leqslant \rho,\] for some parameter $\rho$ in the range $0<\rho \leqslant 1$. Then \[ \Big\vert \frac{1}{N^{d-m}}\sum\limits_{\substack{\mathbf{n} \in [N]^d \\ \Vert L \mathbf{n} \Vert_\infty \leqslant \varepsilon}} \Big( \prod\limits_{j=1}^d f_j(n_j) \Big)\Big\vert \ll_{L,\varepsilon} \rho^{\Omega(1)} + o_{\rho,L}(1)\] as $N\rightarrow \infty$.
\end{Corollary}

\noindent  We can provide detailed information about how the implied constant and the $o_{\rho,L}(1)$ term depend on $L$, but we leave that to the next section and to Theorem \ref{Main Theorem chapter 3}. \\

Before giving some examples, let us make a few remarks about Corollary \ref{much easier to state}. Firstly, note that if $L$ has integer coefficients then, by picking $\varepsilon$ small enough, Corollary \ref{much easier to state} immediately implies Theorem \ref{Linear equations in bounded functions non quantitative}, sine the inequality $\Vert L\mathbf{n}\Vert_\infty \leqslant \varepsilon$ is only satisfied if $L\mathbf{n} = \mathbf{0}$. 

Next, due to the nested property of Gowers norms (see Appendix \ref{sec.Gowers norms}), one sees that Corollary \ref{much easier to state} may be fruitfully applied under the hypothesis $\min_j\Vert f_j\Vert_{U^{d-1}[N]} \leqslant \rho$.

Finally, we note that it might be tempting to think that Corollary \ref{much easier to state} would follow easily from taking rational approximations of the coefficients of $L$ and then using Theorem \ref{Linear equations in bounded functions non quantitative} as a black box. Though of course we cannot completely rule out an alternative approach to that of this paper, when one investigates the quantitative details it seems that such an argument will only quickly succeed if the coefficients of $L$ are all extremely well-approximable by rationals, else the height of the rational approximations becomes too great to apply results like Theorem \ref{Linear equations in bounded functions non quantitative}. We will need to employ a different strategy, and we discuss this at length in Section \ref{Section proof strategy}. \\

\subsection{Fourier uniform sets and other examples}
Let us illustrate the applications of Corollary \ref{much easier to state} with certain explicit examples.

\begin{Example}[Three-term irrational AP]
\label{example three term irrational AP}
\emph{The first example could have been analysed by Davenport and Heilbronn using the methods of \cite{DaHe46}, but we include it here to demonstrate the simplest case in which Corollary \ref{much easier to state} applies. }

\emph{Let $$L := \Big(\begin{matrix} 1 & -\sqrt{2}  & -1 + \sqrt {2}\end{matrix}\Big).$$ Then $m=1$ and $d=3$, and manifestly there does not exist any non-zero row-vector in the row-space of $L$ that has two or fewer non-zero coordinates. Therefore Corollary \ref{much easier to state} applies, and so, if for each $N$ we have three functions $f_1,f_2,f_3:[N] \longrightarrow [-1,1]$ satisfying $\min_j \Vert f_j\Vert_{U^2[N]} \leqslant \rho$ for some $\rho$ in the range $0 < \rho \leqslant 1$, then we have
\begin{equation}
\label{three term Ap eq explicit}
\Big\vert\frac{1}{N^2}\sum\limits_{\substack{n_1,n_2,n_3\leqslant N\\ \vert n_1 - \sqrt{2}n_2 + (-1+\sqrt{2})n_3\vert\leqslant \varepsilon} }f_1(n_1)f_2(n_2)f_3(n_3)\Big\vert \ll_\varepsilon \rho^{\Omega(1)} +o_{\rho}(1)
\end{equation}}
\noindent \emph{as $N\rightarrow \infty$.} 

\emph{The statement (\ref{three term Ap eq explicit}) admits a different interpretation, which some readers may find more natural, that of counting the number of occurrences of a certain irrational pattern: a `three-term irrational arithmetic progression'. Indeed, recall that for $\theta \in \mathbb{R}$ we let $[\theta]$ denote $\lfloor \theta +\frac{1}{2}\rfloor$, i.e. the nearest integer to $\theta$. Then for any three functions $f_1,f_2,f_3:[N]\longrightarrow [-1,1]$, we make the definition
\begin{equation}
\label{reparametrised three term AP}
T(f_1,f_2,f_3):=\frac{1}{N^2}\sum\limits_{x,d\in \mathbb{Z}} f_3(x)f_2(x+d)f_1([x+\sqrt{2} d]).
\end{equation} Informally speaking, $T$ counts the number of near-occurrences of the pattern $(x,x+d,x+\sqrt{2}d)$, weighted by the functions $f_j$. By performing the change of variables $n_1 = [ x+\sqrt 2 d]$, $n_2 = x+d$, $n_3 = x$, and noting that $x+\sqrt{2}d\notin \frac{1}{2}\mathbb{Z}$, we see that}
\begin{equation}
\label{relating T}
T(f_1,f_2,f_3)=\frac{1}{N^2}\sum\limits_{\substack{n_1,n_2,n_3\leqslant N\\ \vert n_1 - \sqrt{2}n_2 + (-1+\sqrt{2})n_3\vert\leqslant \frac{1}{2}} }f_1(n_1)f_2(n_2)f_3(n_3).
\end{equation} \emph{By (\ref{three term Ap eq explicit}), this means that if $\min_j\Vert f_j\Vert_{U^2[N]}\leqslant \rho$ then
\begin{equation}
\label{bound on T}
\vert T(f_1,f_2,f_3)\vert \ll \rho^{\Omega(1)} + o_{\rho}(1)
\end{equation} as $N\rightarrow \infty$.} \\

\emph{Now suppose that $A_N$ is a subset of $[N]$ with $\vert A_N \vert = \alpha_N N$. Let }
\begin{equation*}
f_{A_N}:=1_{A_N} - \alpha_N 1_{[N]}
\end{equation*}
\noindent \emph{be its so-called `balanced function'. By the usual telescoping trick, $T(1_{A_N},1_{A_N},1_{A_N})$ is equal to
\begin{align*}
T(\alpha_N 1_{[N]},\alpha_N 1_{[N]},\alpha_N 1_{[N]})+ 
T(f_{A_N},\alpha_N 1_{[N]},\alpha_N 1_{[N]})+T(1_{A_N},f_{A_N},\alpha_N 1_{[N]})\nonumber \\ +T(1_{A_N},1_{A_N},f_{A_N}).
\end{align*}
\noindent One may then bound the final three terms using $\Vert f_{A_N}\Vert_{U^2[N]}$ and, from the relation (\ref{relating T}), one has then established that, provided $\Vert f_{A_N}\Vert_{U^2[N]}\leqslant \rho$,
\begin{equation*}
\frac{1}{N^2}\sum\limits_{x,d\in \mathbb{Z}} 1_{A_N}(x)1_{A_N}(x+d)1_{A_N}([ x+\sqrt{2} d])
\end{equation*}
\noindent is equal to
\begin{equation}
\label{counting three term APS} \frac{\alpha_N^3}{N^2}\sum\limits_{x,d\in \mathbb{Z}} 1_{[N]}(x)1_{[N]}(x+d)1_{[N]}([ x+\sqrt{2} d]) +O(\rho^{\Omega(1)}) + o_{\rho}(1)
\end{equation}
as $N\rightarrow \infty$. If $\Vert f_{A_N}\Vert_{U^2[N]} = o(1)$ as $N\rightarrow \infty$ then, by picking $\rho = \rho(N)$ to be a quantity that tends to zero suitably slowly, one can ensure that the error term in (\ref{counting three term APS}) is $o(1)$ as $N\rightarrow \infty$.}

\emph{As is familiar from \cite{Go01}, for bounded functions the $U^2$-norm is closely related to the Fourier transform. Indeed, we say that the family of sets $A_N$ is Fourier-uniform if the balanced functions $f_{A_N}$ satisfy \[ \sup\limits_{\theta\in[0,1]}\Big\vert \frac{1}{N}\sum\limits_{n\leqslant N} f_{A_N}(n) e(n\theta)\Big\vert = o(1)\] as $N\rightarrow \infty$, and it is a standard result (see \cite[Exercise 1.3.18]{Ta12}) that $A_N$ is Fourier uniform if and only if $\Vert f_{A_N}\Vert_{U^2[N]} = o(1)$ as $N\rightarrow \infty$. Therefore expression (\ref{counting three term APS}), and the remarks following it, imply the following corollary.}
\begin{Corollary}[Fourier-uniform sets]
\label{Corollary for Fourier uniform sets}
Let $\beta \in \mathbb{R} \setminus \mathbb{Q}$, and for each natural number $N$ let $A_N$ be a subset of $[N]$ with $\vert A_N\vert = \alpha_N N$. Suppose that $A_N$ is a Fourier-uniform family of sets.  Then \[\frac{1}{N^2}\sum\limits_{x,d\in \mathbb{Z}} 1_{A_N}(x)1_{A_N}(x+d)1_{A_N}([ x+\beta d])\] is equal to \[\frac{\alpha_N^3}{N^2}\sum\limits_{x,d\in \mathbb{Z}} 1_{[N]}(x)1_{[N]}(x+d)1_{[N]}([ x+\beta d])+o_\beta(1)\] as $N\rightarrow \infty$, where the $o_{\beta}(1)$ term also depends on the $o(1)$ term in the Fourier-uniformity expression for the family $A_N$. 
\end{Corollary}
\end{Example}

\begin{Example}
\label{four term irrational AP}
\emph{Let 
\begin{equation}
\label{the equation giving an example of L}
L := \left(\begin{matrix}
1 & 0 & -\sqrt{2} & -1 + \sqrt{2} \\
0 & 1 & -\sqrt{3} & -1 + \sqrt{3}
\end{matrix} \right).
\end{equation} Since $\sqrt{2}$ and $\sqrt{3}$ are distinct irrationals it is not hard to see that all elements of the row-space of $L$ must have three or four non-zero coordinates, and so Corollary \ref{much easier to state} applies. Letting $f_1,f_2,f_3,f_4:[N] \longrightarrow [-1,1]$ be arbitrary functions, the reparametrisation $n_1 = [x + \sqrt{2} d]$, $n_2 = [x + \sqrt{3} d]$, $n_3 = x + d$, $n_4 = x$, shows that} \[\frac{1}{N^2}\sum\limits_{\substack{\mathbf{n} \in [N]^4 \\ \Vert L \mathbf{n} \Vert_\infty \leqslant \frac{1}{2}}} \Big( \prod\limits_{j=1}^4 f_j(n_j) \Big) = \frac{1}{N^2}\sum\limits_{x,d\in \mathbb{Z}} f_4(x) f_3(x+d) f_1([x+\sqrt{2}d]) f_2([ x+\sqrt{3}d]).\]
\emph{Corollary \ref{much easier to state} controls the left-hand side of this expression in terms of the Gowers norms of the functions $f_j$, and so the right-hand side is controlled as well.}
\end{Example}

We can generalise the previous two examples as follows. 

\begin{Corollary}
\label{Corollary irrational szem}
Let $\theta_1,\dots,\theta _s \in \mathbb{R}$ be distinct irrational numbers. For each natural number $N$ let $A_N$ be a subset of $[N]$, with $\vert A_N\vert = \alpha_N N$ and with balanced function $f_{A_N}$. Suppose that $\Vert f_{A_N}\Vert_{U^{s+1}[N]} = o(1)$ as $N\rightarrow \infty$. Then 
\begin{equation}
\label{particular patterns}
\frac{1}{N^2} \sum\limits_{x,d \in \mathbb{Z}} 1_{A_N}(x) 1_{A_N}(x+d)\Big (\prod\limits_{i=1}^s 1_{A_N}([x + \theta_i d]) \Big )\end{equation}\noindent is equal to
\begin{equation*}
\frac{\alpha_N^{s+2}}{N^2} \sum\limits_{x,d \in \mathbb{Z}} 1_{[N]}(x) 1_{[N]}(x+d)\Big (\prod\limits_{i=1}^s 1_{[N]}([x + \theta_i d]) \Big ) + o(1)
\end{equation*}
\noindent as $N\rightarrow \infty$, where the $o(1)$ error term may depend on $\theta_1,\dots,\theta_s$ and on the rate of decay of $\Vert f_{A_N}\Vert_{U^{s+1}[N]}$. 
\end{Corollary}
\begin{proof}
Apply Corollary \ref{much easier to state} to the $s$-by-$s+2$ matrix 
\begin{equation}
\label{general irrational L}
L = \left( \begin{matrix} I & -\boldsymbol{\theta} & -1 + \boldsymbol{\theta} \end{matrix} \right),
\end{equation} where $I$ denotes the identity matrix and $\boldsymbol{\theta}$ denotes the vector $(\theta_1,\dots,\theta_s)^T \in \mathbb{R}^{s}$. 
\end{proof}

The question remains as to whether one can use Corollary \ref{Corollary irrational szem}, perhaps in conjunction with a density increment argument such as is used in \cite{GT10a}, to deduce that there are infinitely many $s+2$-tuples of the form $(x,x+d,[x + \theta_1 d],\dots,[ x + \theta_s d ])$ inside any set of natural numbers with positive upper Banach density. More generally, one might wish to find tuples in which all the coordinates are of the form $x + p(d)$ where $p$ is a generalised polynomial of degree $1$ without a constant term. This result is already known, in fact, but as it stands the only proof uses ergodic theory methods (see \cite[Theorem B]{McC05}). We view Corollary \ref{Corollary irrational szem} as a promising first step towards a purely combinatorial proof of this result, with a chance to prove explicit bounds.\\

Corollary \ref{much easier to state} has immediate consequences for counting solutions to diophantine inequalities weighted by explicit bounded pseudorandom functions. In particular there is the following natural analogue of \cite[Proposition 9.1]{GT10} concerning the cancellation of the M\"{o}bius function, which we mentioned earlier in Corollary \ref{Corollary irrational szem mobius}.

\begin{Corollary}[M\"{o}bius orthogonality]
\label{Corollary Mobius}
Let $m,d$ be natural numbers satisfying $d\geqslant m+2$, and let $\varepsilon$ be a positive parameter. Let $L:\mathbb{R}^d \longrightarrow \mathbb{R}^m$ be an $m$-by-$d$ real matrix, with rank $m$. Suppose that there does not exist any non-zero row-vector in the row-space of $L$ that has two or fewer non-zero coordinates. Let $\mu$ denote the M\"{o}bius function and let $N$ be an integer parameter. Then, for any bounded functions $f_2,\dots,f_d:[N]\longrightarrow [-1,1]$, \[ \sum\limits_{ \substack{\mathbf{n} \in [N]^d\\ \Vert L\mathbf{n}\Vert_\infty \leqslant \varepsilon}} \mu(n_1) \Big(\prod\limits_{j=2}^d f_j(n_j)\Big)= o_{L,\varepsilon}(N^{d-m})\] as $N\rightarrow \infty$. The same is true with $\mu$ replaced by the Liouville function $\lambda$. 
\end{Corollary}
\begin{proof}
This follows immediately from Corollary \ref{much easier to state} and the deep facts (stated in \cite{GT10}, proved in \cite{GT12} and \cite{GTZ12}) that $\Vert \mu \Vert_{U^{s+1}[N]} = o_s(1)$ and $\Vert \lambda \Vert_{U^{s+1}[N]} = o_s(1)$ as $N\rightarrow \infty$.
\end{proof}

Corollary \ref{Corollary Mobius}, when applied to the matrix (\ref{general irrational L}), yields Corollary \ref{Corollary irrational szem mobius} from earlier in this introduction. It also yields cancellation in expressions such as 
\begin{equation}
\label{mobius example}\sum\limits_{\substack{\mathbf{n}\in [N]^4\\ n_1 - n_2 = n_2 - n_3\\ \vert (n_2 -n_3) - \sqrt {2}(n_3 - n_4)\vert \leqslant \frac{1}{2}}} \mu(n_1)\mu(n_2)\mu(n_3)\mu(n_4)  = o(N^2)
\end{equation} as $N\rightarrow \infty$. There are of course many such examples; we chose (\ref{mobius example}) to emphasise that one can choose configurations that combine rational and irrational relations.\\

\section{The main theorem}
\label{section history}

In our results from the previous section, the quantitative nature of the dependence of the error terms on the matrix $L$ was hidden. Our main theorem (Theorem \ref{Main Theorem chapter 3} below) addresses this point, which turns out to be surprisingly subtle. \\

To start off, let us introduce a multilinear form that will count solutions to a general version of (\ref{the general inequality}).
\begin{Definition}
\label{Definition solution count}
Let $N,m,d$ be natural numbers, and let $L:\mathbb{R}^d \longrightarrow \mathbb{R}^m$ be a linear map. Let $F:\mathbb{R}^d\longrightarrow [0,1]$ and $G:\mathbb{R}^m\longrightarrow [0,1]$ be two functions, with $F$ supported on $[-N,N]^d$ and $G$ compactly supported. Let $f_1,\dots,f_d:[N]\longrightarrow [-1,1]$ be arbitrary functions. We define 
\begin{equation}
\label{original inequality}
T_{F,G,N}^L(f_1,\dots,f_d) :=  \frac{1}{N^{d-m}}\sum\limits_{\mathbf{n}\in \mathbb{Z}^d}\Big(\prod\limits_{j=1}^d f_j(n_j)\Big)F(\mathbf{n})G(L\mathbf{n}).
\end{equation}
\end{Definition}
\noindent The normalisation factor of $N^{d-m}$ is appropriate; we will show  in Lemma \ref{Lemma crude bound on number of solutions} that $T_{F,G,N}^L(f_1,\dots,f_d) \ll_G 1$. \\

In Theorem \ref{Main Theorem chapter 3} we will bound $T_{F,G,N}^L(f_1,\dots,f_d)$ above by Gowers norms. The error term will depend on three further notions: the rational relations present in $L$; the `approximation function' $A_L$, which will measure the approximate rational relations present in $L$; and the non-degeneracy of $L$, which is related to the non-degeneracy conditions in Theorem \ref{Linear equations in bounded functions non quantitative}. These three notions will be introduced in the next three subsections, before we (finally!) state Theorem \ref{Main Theorem chapter 3} in Section \ref{The main theorem and partial converse}.\\

\subsection{Rational relations}
\label{Section Rational dimension}
Let us consider some properties of the image $L(\mathbb{Z}^d)$. It is certainly true that if $L:\mathbb{R}^d \longrightarrow \mathbb{R}^m$ is a surjective linear map then $\spn(L(\mathbb{Z}^d)) = \mathbb{R}^m$. However, $L(\mathbb{Z}^d)$ needn't be dense in $\mathbb{R}^m$, as the matrix $L$ may satisfy some rational relations. These in turn restrict $L(\mathbb{Z}^d)$ to various affine subspaces of $\mathbb{R}^m$.

This observation motivates the following definitions.

\begin{Definition}[Rational dimension, rational map, purely irrational]
\label{Definition rational space}
Let $m$ and $d$ be natural numbers, with $d\geqslant m+1$. Let $L:\mathbb{R}^d\longrightarrow \mathbb{R}^m$ be a surjective linear map. Let $u$ denote the largest integer for which there exists a surjective linear map $\Theta:\mathbb{R}^m \longrightarrow \mathbb{R}^u$ for which $\Theta L (\mathbb{Z}^d) \subseteq \mathbb{Z}^u$. We call $u$ the \emph{rational dimension} of $L$, and we call any map $\Theta$ with the above property a \emph{rational map} for $L$. We say that $L$ is \emph{purely irrational} if $u=0$. 
\end{Definition}
\noindent For example, suppose that $L:\mathbb{R}^4 \longrightarrow \mathbb{R}^2$ is the linear map represented by the matrix \[ L: = \left(\begin{matrix} 1 & 0 & -\sqrt{2} & -\sqrt{3} + 1 \\
0 & 1 & 5\sqrt{2} & 5\sqrt{3} \end{matrix} \right).\] If $\Theta:\mathbb{R}^2 \longrightarrow \mathbb{R}$ is given by the matrix \[ \Theta:= \Big(\begin{matrix} 5 & 1 \end{matrix} \Big),\] then $\Theta L (\mathbb{Z}^4) \subseteq \mathbb{Z}$, and in fact $\Theta L (\mathbb{Z}^4) = \mathbb{Z}$. So the rational dimension of $L$ is at least $1$. But the rational dimension of $L$ cannot be $2$, as if there were a surjective map $\Theta:\mathbb{R}^2 \longrightarrow \mathbb{R}^2$ such that $\Theta L (\mathbb{Z}^4) \subseteq \mathbb{Z}^2$ then $L(\mathbb{Z}^4)$ would be a subset of a $2$-dimensional lattice, which it is not. So the rational dimension of $L$ is equal to $1$.\\


Ours is certainly not the first paper on the topic of diophantine inequalities to have considered issues such as this. For example, in the previous section we remarked that M\"{u}ller came across a similar phenomenon in the work \cite{Mu05}. Given quadratic forms $Q_1, \dots, Q_r$ he found infinitely many solutions $\mathbf{x}$ to the inequalities \[\vert Q_1(\mathbf{x})\vert < \varepsilon, \dots, \vert Q_{r}(\mathbf{x})\vert  < \varepsilon,\] under the hypotheses that every quadratic form in the set \[ \{ \sum\limits_{i=1}^r \alpha_i Q_i : \alpha_1,\dots, \alpha_r\in \mathbb{R}, \, \ba \neq \mathbf{0}\}\] was irrational and had rank greater than $8r$. One can use the coefficients of the quadratic forms to translate this problem into one of trying to understand $T_{F,G,N}^L(f_1,\dots,f_d)$ for a certain linear map $L$ and for functions $f_1,\dots,f_d$ supported on the image of quadratic monomials. Then, M\"{u}ller's hypothesis that all the linear combinations of the $Q_i$ are irrational is transformed into the hypothesis that $L$ is purely irrational. In this paper we consider all $L$, not just those which are purely irrational, and this causes some added complications. \\

In our definition of rational dimension, there is some flexibility over the exact choice of map $\Theta$. The next lemma identifies an invariant.
\begin{Lemma}
\label{Lemma invariant space}
Let $m$ and $d$ be natural numbers, with $d\geqslant m+1$. Let $L:\mathbb{R}^d\longrightarrow \mathbb{R}^m$ be a surjective linear map, and let $u$ be the rational dimension of $L$. Then, if $\Theta_1,\Theta_2: \mathbb{R}^m \longrightarrow \mathbb{R}^u$ are two rational maps for $L$, $\ker \Theta_1 = \ker \Theta_2$.
\end{Lemma}
\begin{proof}
Suppose that $\Theta_1,\Theta_2:\mathbb{R}^m \longrightarrow \mathbb{R}^u$ are two rational maps for $L$ for which $\ker \Theta_1 \neq \ker \Theta_2$. Then consider the map $(\Theta_1,\Theta_2):\mathbb{R}^m\longrightarrow \mathbb{R}^{2u}.$ The kernel of this map has dimension at most $m-u-1$, as it is the intersection of two different subspaces of dimension $m-u$. Therefore the image has dimension at least $u+1$. 

Also, $((\Theta_1,\Theta_2) \circ L ) (\mathbb{Z}^d) \subseteq \mathbb{Z}^{2u}. $ Let $\Phi$ be any surjective map from $\im ((\Theta_1,\Theta_2))$ to $\mathbb{R}^{u+1}$  for which $\Phi(\mathbb{Z}^{2u}\cap\im ((\Theta_1,\Theta_2))) \subseteq\mathbb{Z}^{u+1}$. Then $\Phi \circ (\Theta_1,\Theta_2):\mathbb{R}^m \longrightarrow \mathbb{R}^{u+1}$ is surjective and $(\Phi \circ (\Theta_1,\Theta_2) \circ L)(\mathbb{Z}^d) \subseteq \mathbb{Z}^{u+1}$. This contradicts the definition of $u$ as the rational dimension. 
\end{proof}

We will also need to identify the quantitative aspects of these rational relations, in order to properly state the main theorem. 

\begin{Definition}[Rational complexity]
Let $m$ and $d$ be natural numbers, with $d\geqslant m+1$. Let $L:\mathbb{R}^d\longrightarrow \mathbb{R}^m$ be a surjective linear map, and let $u$ denote the rational dimension of $L$. We say that $L$ has \emph{rational complexity at most $C$} if there exists a map $\Theta$ that is a rational map for $L$ and for which $\Vert \Theta\Vert_\infty \leqslant C$. If $L$ is purely irrational, then $L$ has rational complexity $0$. 
\end{Definition}
\noindent In this definition, recall that for a linear map $\Theta:\mathbb{R}^m\longrightarrow \mathbb{R}^u$ we use $\Vert \Theta\Vert_\infty$ to denote the maximum absolute value of the coefficients of its matrix with respect to the standard bases.\\

We observe that a linear map with maximal rational dimension, i.e. with rational dimension $m$, is equivalent to a linear map with integer coefficients, in the following sense. 
\begin{Lemma}[Maximal rational dimension]
\label{Lemma full rational dimension}
Let $m$ and $d$ be natural numbers, with $d\geqslant m+1$. Let $L:\mathbb{R}^d\longrightarrow \mathbb{R}^m$ be a surjective linear map, and suppose that $L$ has rational dimension $m$ and rational complexity at most $C$. Then there exists an invertible $m$-by-$m$ matrix $\Theta$ and an $m$-by-$d$ matrix $S$ with integer coefficients such that, as matrices, $\Theta L = S$. Furthermore, $\Vert \Theta\Vert_\infty \leqslant C$. 
\end{Lemma}
\begin{proof}
Let $\Theta:\mathbb{R}^m \longrightarrow \mathbb{R}^m$ be a rational map for $L$ for which $\Vert \Theta \Vert_\infty \leqslant C$.
\end{proof}
\noindent We will use this lemma in Section \ref{section Reductions}, to reduce the study of maps $L$ with maximal rational dimension to the study of maps $L$ with integer coefficients.\\

\subsection{Approximation function}
\label{section approximation function}
We must also quantify the rational relations in a second way. Indeed, $L$ might have rational dimension $u$ but be extremely close to having rational dimension at least $u+1$, in the sense that there might exist some surjective linear map $\Theta:\mathbb{R}^m\longrightarrow \mathbb{R}^{u+1}$ such that the matrix of $\Theta L$ is very close to having integer coefficients. This phenomenon, essentially a notion of diophantine approximation, will also have a quantitative effect on our final bounds. The critical place where it enters the argument is Lemma \ref{Lemma upper bound involving integral}, whose content we briefly sketch here, so as to further motivate our introduction of the `approximation function' below. 

This will be the first, of many, places in the paper in which we have to manipulate Lipschitz functions. For the reader's benefit, in Appendix \ref{Lipschitz functions} we have collected together the definitions and results that we will use. \\

Let $L$ be an $m$-by-$d$ matrix, which may depend on the asymptotic parameter $N$. Suppose that one is seeking an upper bound on $T_{F,G,N}^L(1,\dots,1)$, where $G$ is a Lipschitz function supported on $[-1,1]^m$ and $F$ is a function supported on $[-N,N]^d$. We note that this task is a special case of bounding $T_{F,G,N}^L(f_1,\dots,f_d)$ above by the Gowers norms of the functions $f_i$. In our main proof, bounds on $T_{F,G,N}^L(1,\dots,1)$ will be useful when controlling some error terms which occur when the inequality is perturbed (see Section \ref{Section proof strategy} for a full discussion of this point).

 Also suppose, for simplicity, that the first $m$ columns of $L$ form the identity matrix, and let $\mathbf{v_j} \in \mathbb{R}^m$ denote the $j^{th}$ column of $L$. Then, by summing over the variables $n_1,\dots,n_m \in \mathbb{Z}$, one quickly derives that \[ T_{F,G,N}^L(1,\dots,1) \ll \frac{1}{N^{d-m}} \sum\limits_{\substack{n_{m+1},\dots,n_d \in \mathbb{Z} \\ \vert n_{m+1}\vert, \dots, \vert n_d\vert \leqslant N}} \widetilde{G}\Big(\sum\limits_{j=m+1}^d \mathbf{v_j} n_j\Big),\] where $\widetilde{G}$ is a $\mathbb{Z}^m$-periodic Lipschitz function formed by taking translates of $G$. 

We consider $\widetilde{G}$ as a function on $\mathbb{R}^m/\mathbb{Z}^m$, and approximate it by a short exponential sum (as one may do for all such Lipschitz functions).\footnote{See Lemma \ref{Fourier transforms of Lipschitz functions on tori}.} As is familiar in these kind of problems, one is left with having to bound the expression that arises from the non-zero Fourier modes. Here, one ends up with terms\[ \frac{1}{N^{d-m}}\sum\limits_{\substack{n_{m+1},\dots,n_d \in \mathbb{Z} \\ \vert n_{m+1}\vert, \dots, \vert n_d\vert \leqslant N}} e\Big( \mathbf{k} \cdot \sum\limits_{j=m+1}^d \mathbf{v_j} n_j \Big)\] with $\mathbf{k} \in \mathbb{Z}^m \setminus \{\mathbf{0}\}$, which one may sum as geometric progressions over $n_{m+1}$ to $n_d$. This means we have to control \[ \max \limits_{\substack{ \mathbf{k} \in \mathbb{Z}^m \\ 0 < \Vert \mathbf{k} \Vert_\infty \leqslant X}} \Big(\prod\limits_{j=m+1}^d \min(1,N^{-1} \Vert \mathbf{k} \cdot \mathbf{v_j}\Vert_{\mathbb{R}/\mathbb{Z}}^{-1})\Big),\] where $X$ is some threshold, and the above expression is certainly bounded by 
\begin{equation}
\label{approx function appears}
 N^{-1}\max \limits_{ \substack{\mathbf{k} \in \mathbb{Z}^m \\ 0 < \Vert \mathbf{k}\Vert_\infty \leqslant X}} \Vert L^T \mathbf{k}\Vert_{\mathbb{R}^d/\mathbb{Z}^d}^{-1},
 \end{equation} as the first $m$ columns of $L$ have integer coordinates. One hopes to bound expression (\ref{approx function appears}) by $o(1)$ as $N\rightarrow \infty$.

We observe two facts about (\ref{approx function appears}). Firstly, if $L$ isn't purely irrational and if $X$ is larger than the rational complexity of $L$, then the expression (\ref{approx function appears}) is infinite! Secondly, even if $L$ is purely irrational then it could still be the case that (\ref{approx function appears}) tends to infinity with $N$, as $L$ may depend on $N$. We conclude that, with the state of our current argument, the size of expression (\ref{approx function appears}) -- or an expression like it -- must be included in our error terms. \\

Motivated by the above discussion, we introduce the `approximation function'. The definition is phrased in terms of dual functions -- this will make linear algebraic manipulations more straightforward later on -- and for real vectors $\varphi$ rather than for integer vectors $\mathbf{k}$, which reflects the general situation in which the first $m$ columns of $L$ are not the identity matrix. We also generalise to the case of arbitrary rational dimension $u$, rather than just $u=0$. 

 Following this definition and some remarks, we will show how to calculate the approximation function in an explicit example. This should hopefully serve to clarify the properties of this somewhat technical object. 

\begin{Definition}[Approximation function]
\label{Definition approximation function}
Let $m$ and $d$ be natural numbers, with $d\geqslant m+1$. Let $L:\mathbb{R}^d\longrightarrow \mathbb{R}^m$ be a surjective linear map, and let $u$ denote the rational dimension of $L$. Let $\Theta:\mathbb{R}^m\longrightarrow \mathbb{R}^u$ be any rational map for $L$. Suppose that $u\leqslant m-1$. Then we define the \emph{approximation function of $L$}, denoted $A_L:(0,1] \times (0,1]\longrightarrow (0,\infty)$, by \[ A_L(\tau_1,\tau_2) : = \inf\limits_{\substack{\varphi \in (\mathbb{R}^m)^*\\ \dist( \varphi, \Theta^*((\mathbb{R}^u)^*)) \geqslant \tau_1 \\  \Vert \varphi \Vert_\infty \leqslant \tau_2^{-1}} } \dist (L^* \varphi, (\mathbb{Z}^d)^T ),\] where $(\mathbb{Z}^d)^T$ denotes the set of those $\varphi \in (\mathbb{R}^d)^*$ that have integer coordinates with respect to the standard dual basis.

If $u = m$, we define\footnote{When $u=m$ we've already seen (Lemma \ref{Lemma full rational dimension}) that $L$ may be transformed into a matrix $S$ with integer coefficients, and thus $L$ is somewhat degenerate from the point of view of diophantine approximation. We define $A_L(\tau_1,\tau_2)$ for such matrices only to avoid having to discuss this special case in the statement of Theorem \ref{Main Theorem chapter 3} later on.} $A_L(\tau_1,\tau_2)$ to be identically equal to $\tau_1$. 
\end{Definition}
\noindent From our discussion of (\ref{approx function appears}) above, one sees that upper bounds on $A_L(\tau_1,\tau_2)^{-1}$ will be the main focus. The threshold $\tau_2^{-1}$ plays the role of the threshold $X$ in (\ref{approx function appears}), and the condition $\dist(\varphi, \Theta^*((\mathbb{R}^u)^*))\geqslant \tau_1$ corresponds to the condition $\Vert \mathbf{k}\Vert_\infty \geqslant 1$ which is implicit in (\ref{approx function appears}).\\

There is some notation to unpack in Definition \ref{Definition approximation function}. Regarding the notion `$\dist$', we remind the reader of some material from Section \ref{Section Notation}, namely that we consider $a$-by-$b$ matrices $M$ as elements of $\mathbb{R}^{ab}$, simply by identifying the coefficients of $M$ with coordinates in $\mathbb{R}^{ab}$. The $\ell^\infty$ norm and the $\operatorname{dist}$ operator may then be defined on matrices, i.e. if $V$ is a set of $a$-by-$b$ matrices, and $L$ is an $a$-by-$b$ matrix, then $$\operatorname{dist}(L,V): = \operatorname{inf}\limits_{L^\prime \in V} \Vert L - L^ \prime\Vert_\infty.$$ In this instance we are working with $1$-by-$d$ matrices, i.e. elements of $(\mathbb{R}^d)^*$. \\

Note that Definition \ref{Definition approximation function} is independent of the choice of $\Theta$. Indeed, by basic linear algebra $\Theta^*((\mathbb{R}^u)^*) = (\ker \Theta)^{0} $, where $(\ker \Theta)^{0}$ is the annihilator of $\ker \Theta$ (see Section \ref{Section Notation}). By Lemma \ref{Lemma invariant space}, $\ker \Theta$ is independent of the choice of $\Theta$, and therefore so is $(\ker \Theta)^0$.\\

\begin{Example}
\label{Example approx function}
\emph{Suppose that, as a matrix, 
\begin{equation*}
L := \left(\begin{matrix} 1 & -\sqrt{2}  & -1 + \sqrt {2}\end{matrix}\right)
\end{equation*} as in Example \ref{example three term irrational AP}. Then $L$ is purely irrational, i.e. $u=0$, since its coefficients are not all in rational ratio. Therefore $A_L(\tau_1,\tau_2)$ is equal to \[ \inf\limits_{\substack{k\in \mathbb{R}: \tau_1 \leqslant \vert k\vert \leqslant \tau_2^{-1} }} \max(\Vert k\Vert_{\mathbb{R}/\mathbb{Z}},\Vert -k \sqrt{2}\Vert_{\mathbb{R}/\mathbb{Z}}, \Vert -k+k\sqrt{2}\Vert_{\mathbb{R}/\mathbb{Z}}). \] As we said above, we seek an upper bound on $A_L(\tau_1,\tau_2)^{-1}$. To this end, we claim that, for this particular $L$ and for all $\tau_1,\tau_2 \in (0,1]$, one has \[ A_L(\tau_1,\tau_2) \gg \min (\tau_1,\tau_2).\]}
\emph{Indeed, we know that, for all $q \in \mathbb{N}$, $\Vert q \sqrt{2} \Vert_{\mathbb{R}/\mathbb{Z}} \geqslant 1/(10q)$. This is the statement that $\sqrt{2}$ is a badly approximable irrational. The proof is straightforward: if there were some natural number $p$ for which $\vert q\sqrt{2} - p\vert < 1/(10q)$, then \[ 1 \leqslant \vert 2q^{2} - p^2\vert < \frac{\sqrt{2}}{10} + \frac{p}{10q} < \frac{\sqrt{2}}{5} + \frac{1}{10},\] which is a contradiction.}

\emph{Suppose first that $\Vert k\Vert_{\mathbb{R}/\mathbb{Z}}\leqslant \tau_2/100$                                                                                  and $1/2 \leqslant \vert k\vert \leqslant \tau_2^{-1}$. Then, replacing $k$ by $[k]$ (the nearest integer to $k$), we can conclude that}
\begin{align*}
\max(\Vert -k \sqrt{2}\Vert_{\mathbb{R}/\mathbb{Z}}, \Vert -k+k\sqrt{2}\Vert_{\mathbb{R}/\mathbb{Z}}) &\geqslant \Vert [k] \sqrt{2}\Vert_{\mathbb{R}/\mathbb{Z}} - \frac{\tau_2}{50} \\
& \geqslant \frac{1}{10 [k]} - \frac{\tau_2}{50}  \\
& \geqslant\frac{1}{10\tau_2^{-1} + 10} - \frac{\tau_2}{50}  \\
&\gg \tau_2.
\end{align*}

\emph{Otherwise, one has \[ \Vert k\Vert_{\mathbb{R}/\mathbb{Z}} \gg \min(\tau_1,\tau_2).\] Therefore, \[A_L(\tau_1,\tau_2) \gg \min(\tau_1, \tau_2)\] as claimed.}
\end{Example}

It is not too difficult to show that if $L$ is an $m$-by-$d$ matrix with rank $m$ and with algebraic coefficients, then \begin{equation}
\label{approximation function in the algebraic case}
A_L(\tau_1,\tau_2) \gg_L \min(\tau_1,\tau_2^{O_L(1)}),
\end{equation} where the $O_L(1)$ term in the exponent depends on the algebraic degree of the coefficients\footnote{One could perhaps remove this dependence by using the Schmidt subspace theorem, though as there are power losses throughout the rest of the argument there does not seem to be a great advantage in doing so.} of $L$. We shall give a proof of this statement in Appendix \ref{section algebraic approximation}. In general, however, $A_L(\tau_1,\tau_2)$ could tend to zero arbitrarily quickly as $\tau_2$ tends to zero, for example in the case when $L = \left(\begin{matrix} 1 & -\lambda  & -1 + \lambda\end{matrix}\right)$ and $\lambda$ is a Liouville number (an irrational number that may be very well approximated by rationals).

Yet, however fast $A_L(\tau_1,\tau_2)$ decays, we have the following critical claim. 
\begin{Claim}
\label{claim approximation function is positive}
For all permissible choices of $L$, $\tau_1$ and $\tau_2$ in Definition \ref{Definition approximation function}, $A_L(\tau_1,\tau_2)$ is positive. 
\end{Claim}
\begin{proof}
Let $u$ be the rational dimension of $L$. Without loss of generality we may assume that $u\leqslant m-1$. Then, for all $\varphi \in (\mathbb{R}^m)^* \setminus \Theta^*((\mathbb{R}^u)^*)$ we have that $\dist (L^* \varphi, (\mathbb{Z}^d)^T )> 0$. (If this were not the case then the map $(\Theta,\varphi):\mathbb{R}^m \longrightarrow \mathbb{R}^{u+1}$ would contradict the definition of $u$.) Therefore, as the definition of $A_L(\tau_1,\tau_2)$ involves taking the infimum of a positive continuous function over a compact set, $A_L(\tau_1,\tau_2)$ is positive.
\end{proof}
\noindent The expression $A_L(\tau_1,\tau_2)^{-1}$ will appear in the error term of our main theorem; Claim \ref{claim approximation function is positive} shows that such an error term still has content. \\

\subsection{Non-degeneracy}
\label{Section approximate degeneracy}
In the statement of Theorem \ref{Linear equations in bounded functions non quantitative}, which we remind the reader was the result of Green and Tao that used Gowers norms to control the number of solutions to linear equations with integer coefficients, one recalls that there were certain linear-algebraic notions of non-degeneracy for the matrix $L$. These concerned the rank of $L$ and the properties of its row space. In the setting of diophantine inequalities it will transpire that the same notions of non-degeneracy are important -- this much was obvious from the statement of Corollary \ref{much easier to state} -- except that, in order to control the error terms when $L$ depends on $N$, one must assume that $L$ is not even `approximately' degenerate. \\

In order to make these notions precise, we will first give some names to the sets of degenerate maps that we wish to avoid. 

\begin{Definition}[Low-rank variety]
\label{Definition rank variety}
Let $m,d$ be natural numbers satisfying $d\geqslant m+1$. Let $V_{\rank}(m,d)$ denote the set of all linear maps $L:\mathbb{R}^d \longrightarrow \mathbb{R}^m$ whose rank is less than $m$. We call $V_{\rank}(m,d)$ the \emph{low-rank variety}.

 Let $V_{ \rank}^{\unif}(m,d)$ denote the set of all linear maps $L:\mathbb{R}^d \longrightarrow \mathbb{R}^m$ for which there exists a standard basis vector of $\mathbb{R}^d$, say $\mathbf{e_i}$, for which $L|_{\spn(\mathbf{e_j}:j\neq i)}$ has rank less than $m$. We call $V_{ \rank}^{\unif}(m,d)$ the \emph{uniform low-rank variety}. 
\end{Definition}

\noindent We make the trivial observation that $V_{\rank}^{\unif}(m,d)$ contains $V_{\rank}(m,d)$. For certain technical reasons it will be much more convenient to work with matrices $L \notin V_{\rank}^{\unif}(m,d)$, as opposed to merely working with matrices $L \notin V_{\rank}(m,d)$, as we will be able to fix an arbitrary coordinate and still be left with a full rank linear map.

\begin{Definition}[Dual degeneracy variety]
\label{Definition dual degeneracy variety}
Let $m,d$ be natural numbers satisfying $d\geqslant m+2$. Let $\mathbf{e_1}, \dots,\mathbf{e_d}$ denote the standard basis vectors of $\mathbb{R}^d$, and let  $\mathbf{e_1^\ast}, \dots,\mathbf{e_d^\ast}$ denote the dual basis of $(\mathbb{R}^d)^\ast$. Then let $V^*_{\degen}(m,d)$ denote the set of all linear maps $L:\mathbb{R}^d \longrightarrow \mathbb{R}^m$  for which there exist two indices $i,j \leqslant d$, and some real number $\lambda$, such that $\mathbf{e_i}^* - \lambda \mathbf{e_j}^*$ is non-zero and $(\mathbf{e_i}^* - \lambda \mathbf{e_j}^*) \in L^*((\mathbb{R}^m)^*)$. We call $V^*_{\degen}(m,d)$ the \emph{dual degeneracy variety}. 
\end{Definition}

\noindent It may be easily verified that this definition does nothing more than rephrase the condition that appeared in the statements of Corollary \ref{much easier to state} and Theorem \ref{Linear equations in bounded functions non quantitative} concerning the row-space of a degenerate map $L$, namely that there exists a non-zero row-vector in the row-space of $L$ that has two or fewer non-zero coordinates. The formulation in terms of dual spaces will be particularly convenient, however, for some of the algebraic manipulations in Section \ref{section Reductions}. This is the reason why we use the term `dual' in the name\footnote{Later on (in Definition \ref{Definition degeneracy varieties}) we will have a set  of degenerate maps $V_{\degen}(d-m,d)$ which will parametrise the kernel of maps in $V_{\degen}^*(m,d)$. Since these maps feel somewhat dual to those maps in $V_{\degen}^*(m,d)$, we will come to call $V_{\degen}(d-m,d)$ the `degeneracy variety'.} of $V_{\degen}^*(m,d)$.

Having introduced $V_{\rank}(m,d)$, $V_{\rank}^{\unif}(m,d)$ and $V_{\degen}^*(m,d)$, we can articulate the relationship between the non-degeneracy conditions in Theorem \ref{Linear equations in bounded functions non quantitative} (for linear equations given by $L$) and our non-degeneracy conditions in Theorem \ref{Main Theorem chapter 3} below (for linear inequalities given by $L$). Indeed, for equations, $L$ is non-degenerate if \begin{equation}
\label{conditions for equations}
L \notin V_{\rank}(m,d) \qquad \text{or} \qquad L \notin V_{\degen}^*(m,d).
\end{equation} For inequalities,  $L$ is non-degenerate if
\begin{equation}
\label{conditions for inequalities}
\dist(L, V_{\rank}^{\unif}(m,d)) \geqslant c \qquad \text{or} \qquad \dist(L,V_{\degen}^*(m,d)) \geqslant c^\prime,
\end{equation} for some fixed parameters $c$ and $c^\prime$. One can see immediately how the conditions for inequalities are `approximate' versions of the conditions for equations. \\

\begin{Example}
\label{Example bad behaviour}
\emph{It may be instructive to consider a matrix such as $$L = \left(\begin{matrix}
1 +N^{-1}& \sqrt{3}+N^{-\frac{1}{2}} & \pi & -\pi+\sqrt{2} \\
2 & 2\sqrt{3}+N^{-\frac{1}{2}} & -\sqrt{5} & e
\end{matrix} \right).$$ We observe that $L$ has rank 2 and $L\notin V_{\degen}^*(2,4)$. If one knew Theorem \ref{Linear equations in bounded functions non quantitative} and the conditions (\ref{conditions for equations}), then one might perhaps have hoped to apply the theory of Gowers norms to bound the number of solutions to inequalities given by $L$. However, by considering perturbations of the first two columns, we see that $\operatorname{dist}(L,V_{\degen}^*(2,4))=o(1)$ as $N\rightarrow \infty$. Indeed, one may perturb $L$ by $O(N^{-1/2})$ such that there is a vector $(0,0,x_3,x_4)$ in the row space. So, despite the fact that $L$ is non-degenerate from the point of view of equations, $L$ \emph{is} degenerate from the point of view of inequalities and the conditions (\ref{conditions for inequalities}). Thus, our main theorem on inequalities will not apply to this $L$. }

\emph{Furthermore, we have another result (Theorem \ref{Converse to main theorem} below) which shows that one cannot possibly use Gowers norms to control inequalities given by such an $L$. Therefore, whatever methods we use to prove Theorem \ref{Main Theorem chapter 3}, these methods must necessarily break down when applied to this example.}\\
\end{Example}

\subsection{The main theorem and a partial converse}
\label{The main theorem and partial converse}
Having laid the groundwork, we may now state the main theorem of this paper. 

\begin{Theorem}[Main Theorem]
\label{Main Theorem chapter 3}
Let $m,d$ be natural numbers, satisfying $d\geqslant m+2$, and let $\varepsilon , c,  C, C^\prime$ be positive reals. Let $N$ be an integer parameter and let $L = L(N):\mathbb{R}^d \longrightarrow \mathbb{R}^m$ be a surjective linear map that satisfies $\Vert L\Vert_\infty\leqslant C$. Let $A_L:(0,1]\times (0,1]\longrightarrow (0,\infty)$ be the approximation function of $L$. Suppose further that $\operatorname{dist}(L,V_{\rank}^{\unif}(m,d)) \geqslant c$, that $\operatorname{dist}(L,V^*_{\degen}(m,d))\geqslant c^\prime$, and that $L$ has rational complexity at most $C^\prime$. Then there exists a natural number $s$ at most $d-2$, independent of $\varepsilon$, such that the following is true. Let $F:\mathbb{R}^d\longrightarrow [0,1]$ be the indicator function of $[1,N]^d$, and let $G:\mathbb{R}^m\longrightarrow [0,1]$ be the indicator function of a convex domain contained in $[-\varepsilon, \varepsilon]^m$. Let $f_1,\dots,f_d:[N]\longrightarrow [-1,1]$ be arbitrary functions, and suppose that \[ \min_j \Vert f_j\Vert_{U^{s+1}[N]} \leqslant \rho,\] for some parameter $\rho$ in the range $0<\rho \leqslant 1$. Then 
\begin{equation}
\label{main equation chapter 3 theorem}
T^{L}_{F,G,N}(f_1,\dots,f_{d})\ll_{c,c^\prime,C,C^\prime,\varepsilon} \rho^{\Omega(1)} + o_{\rho,A_L,c,c^\prime,C,C^\prime}(1)
\end{equation} as $N\rightarrow \infty$.  The $o_{\rho,A_L,c,c^\prime,C,C^\prime}(1)$ term may be bounded above by \[ N^{-\Omega(1)}\rho^{-O(1)} A_L(\Omega_{c,c^\prime,C,C^{\prime}}(1),\rho)^{-1}.\] 
\end{Theorem}
\noindent We remind the reader that the implied constants may depend on the dimensions $m$ and $d$. Also note that in the above statement one may replace $C$ and $C^\prime$ by a single constant $C$, and $c$ and $c^\prime$ by a single constant $c$, without weakening the conclusion. We proceed with this assumption. \\

Let us note some consequences of this theorem. Firstly, since $A_L(\Omega_{c,C}(1),\rho)^{-1}$ is finite (by Claim \ref{claim approximation function is positive}), Theorem \ref{Main Theorem chapter 3} immediately implies Corollary \ref{much easier to state} (this was the qualitative statement around which we structured Section \ref{Sec introduction}). Hence Theorem \ref{Main Theorem chapter 3} also implies all the other corollaries from Section \ref{Sec introduction}. Secondly, from (\ref{approximation function in the algebraic case}), or rather from our full quantitative version in Lemma \ref{Lemma algebraic coeffs implies gen appr}, we have another corollary for matrices $L$ with algebraic coefficients. 
\begin{Corollary}[Inequalities with algebraic coefficients]
\label{Corollary algebraic}
Assume the same hypotheses as Theorem \ref{Main Theorem chapter 3}, and assume further that $L$ has algebraic coefficients with algebraic degree at most $k$. Let $H$ denote the maximum absolute value of all of the coefficients of all of the minimal polynomials of the coefficients of $L$. Then \[ T^{L}_{F,G,N}(f_1,\dots,f_{d})\ll_{c,C,\varepsilon,H} \rho^{\Omega(1)} + N^{-\Omega(1)}\rho^{-O_k(1)}.\]
\end{Corollary}


The reader may wonder how the implied constant in these statements depends on $\varepsilon$. Ultimately the implied constant in (\ref{main equation chapter 3 theorem}) tends to infinity as $\varepsilon$ tends to zero, as our approximation argument in Section \ref{section transfer} will not be efficient in powers of $\varepsilon$. Yet, to prevent our notation becoming too unreadable, we choose not to keep track of the precise behaviour of implied constants involving $\varepsilon$.\\

As we remarked in Section \ref{Section approximate degeneracy} and Example \ref{Example bad behaviour}, we can also prove a partial converse to Theorem \ref{Main Theorem chapter 3}. This result demonstrates that the non-degeneracy condition $\dist(L,V_{\degen}^*(m,d)) \geqslant c$ is necessary in order to use Gowers norms to control inequalities given by $L$.  

\begin{Theorem}
\label{Converse to main theorem}
Let $m,d$ be natural numbers, satisfying $d\geqslant m+2$, and let $\varepsilon, c, C$ be positive constants. For each natural number $N$, let $L = L(N):\mathbb{R}^{d}\longrightarrow \mathbb{R}^m$ be a linear map satisfying $\Vert L\Vert_\infty\leqslant C$. Let $F:\mathbb{R}^d\longrightarrow [0,1]$ denote the indicator function of $[1,N]^d$ and $G:\mathbb{R}^m\longrightarrow [0,1]$ denote the indicator function of $[-\varepsilon,\varepsilon]^m$. Assume further that $\operatorname{dist}(L,V_{\rank}(m,d))\geqslant c$ and that $T_{F,G,N}^L(1,\dots,1)\gg_{c,C,\varepsilon} 1$ for large enough $N$.

Suppose that $$\liminf\limits_{N\rightarrow\infty}\operatorname{dist}(L,V^*_{\degen}(m,d)) = 0.$$ Let $s$ be a natural number, let $H:\mathbb{R}_{>0}\rightarrow \mathbb{R}_{>0}$ be any function satisfying $H(\rho) = \kappa(\rho)$, and for each $N$ let $E_{\rho}(N)$ denote some error term depending on a parameter $\rho$ and satisfying $E_{\rho}(N) = o_{\rho}(1)$ as $N \rightarrow \infty$. Then one can find infinitely many natural numbers $N$ such that there exist functions $f_1,\dots,f_d:[N]\rightarrow [-1,1]$ and some $\rho$ at most $1$ such that both \[ \min_j\Vert f_j\Vert_{U^{s+1}[N]}\leqslant\rho\] and 
\begin{equation}
\label{conclusion of converse theorem}
\vert T_{F,G,N}^L(f_1,\dots,f_d)\vert > H(\rho)+E_{\rho}(N).
\end{equation} 
\end{Theorem}

\noindent In other words, the conclusion of Theorem \ref{Main Theorem chapter 3} cannot possibly hold if\\ $\operatorname{dist}(L,V^*_{\degen}(m,d))$ is arbitrarily close to $0$, even if one replaces the $\rho^{\Omega(1)}$ dependence in (\ref{main equation chapter 3 theorem}) with a function $H(\rho)$ that could potentially decay to zero arbitrarily slowly as $\rho$ tends to zero. \\

The proof of Theorem \ref{Converse to main theorem} is contained in Section \ref{Constructions}, which can be read independently of the rest of the paper \\

\subsection{The proof strategy}
\label{Section proof strategy}
All the corollaries from Sections \ref{Sec introduction} and \ref{section history} are implied by Theorem \ref{Main Theorem chapter 3}, so our remaining task is to prove this theorem. Speaking somewhat informally, we wish to bound $T_{F,G,N}^L(f_1,\dots,f_d)$ in terms of some Gowers norms $\Vert f_j\Vert_{U^{s+1}[N]}$ when the functions $F$ and $G$ are the indicator functions of certain convex domains. Now, one might expect the proof to be easier if, instead, $F$ and $G$ were functions with nicer analytic properties -- Lipschitz functions, for example. This is indeed the case, and thus our proof splits naturally into two parts. The first part, contained in Sections \ref{section upper bounds} and \ref{section Reductions}, reduces Theorem \ref{Main Theorem chapter 3} to a similar statement in which the functions $F$ and $G$ are Lipschitz -- this will be Theorem \ref{Theorem rational set out version}. The second part of the paper is devoted to proving Theorem \ref{Theorem rational set out version}. For the rest of this subsection we will try to articulate the strategies for each part, and to elucidate the main technical difficulties. \\

In \cite{GT10}, replacing convex cut-offs with Lipschitz cut-offs was an easy operation, accomplished in a couple of pages in Appendices A and C of that paper. Somewhat surprisingly, this part turns out to be the trickiest element in the setting of inequalities, at least when $L$ is not purely irrational. 

Replacing $F$ with a Lipschitz cut-off is no issue, but the difficulty comes from replacing $G$. Consider the example \[ L:=\left( \begin{matrix} 1& 0 & - \sqrt{2} & -\sqrt{3} + 1 \\ 
0 & 1 & 5 \sqrt{2} & 5 \sqrt{3} \end{matrix} \right)\] from Section \ref{Section Rational dimension}, in which we established that $L$ has rational dimension $1$ and that \[L(\mathbb{Z}^4) \subset \{ \mathbf{x} \in \mathbb{R}^2 : \mathbf{x} \cdot \left(\begin{matrix}5 \\ 1 \end{matrix}\right) \in \mathbb{Z}\}. \]Take $G$ to be the indicator of the compact convex domain \[ \{\mathbf{x} \in \mathbb{R}^2: \Vert \mathbf{x}\Vert_\infty \leqslant 10, \, -1 \leqslant \mathbf{x} \cdot \Big(\begin{matrix}5 \\ 1 \end{matrix}\Big) \leqslant 1 \}.\] Then 
\begin{equation}
\label{explaining}
T_{F,G,N}^L(f_1,\dots,f_4) = \frac{1}{N^2} \sum\limits_{ \substack{ \mathbf{n} \in \mathbb{Z}^4 \\ \Vert L\mathbf{n}\Vert_\infty \leqslant 10 \\ (\begin{smallmatrix}5 & 1 \end{smallmatrix})L\mathbf{n} = -1,0,1}} \Big(\prod\limits_{j=1}^4 f_j(n_j)\Big)F(\mathbf{n}).
\end{equation} To replace a convex cut-off $G$ by a Lipschitz cut-off, a natural approach is to take a Lipschitz function $\widetilde{G}$ that is a minorant\footnote{One also finds a majorant Lipschitz function, but that won't feature in this example.} for $G$, with $\widetilde{G}$ supported on the set \[ \{\mathbf{x} \in \mathbb{R}^2: \Vert \mathbf{x}\Vert_\infty \leqslant 10 - \delta, \quad -1 + \delta \leqslant \mathbf{x} \cdot \left(\begin{matrix}5 \\ 1 \end{matrix}\right)  \leqslant 1 - \delta \}\] for some small positive parameter $\delta$, and $\widetilde{G}$ identically equal to $1$ on the set \[\{\mathbf{x} \in \mathbb{R}^2: \Vert \mathbf{x}\Vert_\infty \leqslant 10 - 2\delta, \quad -1 + 2\delta \leqslant \mathbf{x} \cdot \left(\begin{matrix}5 \\ 1 \end{matrix}\right)  \leqslant 1 - 2\delta \}.\] One has $\Vert G - \widetilde{G}\Vert_1 = \kappa(\delta)$, so one might hope that for any functions $f_1,\dots,f_4$ one would have 
\begin{equation}
\label{kappa delta}
\vert T_{F,G,N}^L(f_1,\dots,f_4) - T_{F,\widetilde{G},N}^L(f_1,\dots,f_4)\vert = \kappa(\delta).
\end{equation} However, no matter how small we choose $\delta$, 
\begin{equation}
\label{cut range of summation}
 T_{F, \widetilde{G},N}^L(f_1,\dots,f_4) \approx \frac{1}{N^2} \sum\limits_{\substack{\mathbf{n} \in \mathbb{Z}^4 \\ \Vert L \mathbf{n}\Vert_\infty \leqslant 10 - \delta \\ (\begin{smallmatrix}5 & 1 \end{smallmatrix})L\mathbf{n} = 0}}\Big(\prod\limits_{j=1}^4 f_j(n_j)\Big) F(\mathbf{n}).
 \end{equation} In moving from $G$ to $\widetilde{G}$ the range of summation for $\mathbf{n}$ between expressions (\ref{explaining}) and (\ref{cut range of summation}) has been cut by a factor of two-thirds! Thus we have no reason to expect that (\ref{kappa delta}) should hold for all functions $f_1,\dots,f_4$.\\

We circumvent these difficulties by employing the following idea. Rather than replacing $G$ with a Lipschitz cut-off straight away, when faced with an expression such as (\ref{explaining}) we can perform some initial reparametrisation, observing that there is a linear map $\Xi:\mathbb{R}^3 \longrightarrow \mathbb{R}^4$ with integer coefficients which gives a lattice parametrisation of those $\mathbf{n} \in \mathbb{Z}^4$ for which $(\begin{matrix}5 & 1 \end{matrix})L\mathbf{n}  = 0$, namely \[ \Xi\left(\begin{matrix} m_1 \\m_2 \\m_3 \end{matrix} \right) = \left(\begin{matrix} m_1 \\ -5m_1 - 5m_2 \\ m_3 \\ m_2 \end{matrix}\right).\] Moreover, $\mathbf{n} \in \mathbb{Z}^4$ with $(\begin{matrix}5 & 1 \end{matrix})L\mathbf{n} = \pm 1$ if and only if there are integers $m_1,m_2,m_3$ for which \[ \mathbf{n} = \Xi\left(\begin{matrix} m_1 \\m_2 \\m_3 \end{matrix} \right) + \left(\begin{matrix} 0 \\ \pm 1 \\ 0 \\0\end{matrix}\right).\] This enables us to decompose $T_{F,G,N}^L(f_1,\dots,f_4)$ into three separate expressions, each of the form 
\begin{equation}
\label{the expression that doesn't change}
\frac{1}{N^2} \sum\limits_{\mathbf{m} \in \mathbb{Z}^3 } \Big( \prod\limits_{j=1}^4 f_j(\Xi(\mathbf{m})_j+ \widetilde{\mathbf{r}}_j) \Big) F(\Xi(\mathbf{m}) + \widetilde{\mathbf{r}})1_{[-10,10]^2}(L(\Xi\mathbf{m} + \widetilde{\mathbf{r}}))
\end{equation} for some different vector $\widetilde{\mathbf{r}} \in \mathbb{Z}^4$, where $\Xi(\mathbf{m})_j$ denotes the $j^{th}$ coordinate of $\Xi(\mathbf{m})$. Now, replace the convex cut-off function $1_{[-10,10]^2}$ with some Lipschitz minorant $\widetilde{G}$ which is supported on $[-10 + \delta, 10 - \delta]^2$ and equal to $1$ on $[-10 + 2 \delta, 10 - 2 \delta]^2$, in each of the three expressions (\ref{the expression that doesn't change}) separately. Then the size of these expressions \emph{will} stay roughly constant.

To quantify this step, the approximation function $A_L$ enters the picture. Indeed, if $\widetilde{\mathbf{r}} = \mathbf{0}$ the error term introduced by applying such an approximation to (\ref{the expression that doesn't change}) is bounded above by \[ T_{F,G^*,N}^{L\Xi}(1,\dots,1),\] where $G^*$ is some other Lipschitz function supported on \[\{ \mathbf{x} \in \mathbb{R}^2 : 10- 2 \delta \leqslant \Vert \mathbf{x}\Vert_\infty \leqslant 10 + 2 \delta \}.\] Finding an upper bound on expressions such as $ T_{F,G^*,N}^{L\Xi}(1,\dots,1)$ is exactly the endeavour we discussed in Section \ref{section approximation function}, when motivating the introduction of the approximation function $A_L$. The only difference is that now we are dealing with the function $A_{L\Xi}$, rather than $A_L$. 

It turns out that the map $L\Xi$ is most naturally viewed as a map from $\mathbb{R}^3$ into a one dimensional space, i.e. \[ L\Xi: \mathbb{R}^3 \longrightarrow \{ \mathbf{x} \in \mathbb{R}^2 : \mathbf{x} \cdot \left(\begin{matrix}5 \\ 1 \end{matrix}\right) = 0 \},\] whereas $L$ maps from $\mathbb{R}^4$ to $\mathbb{R}^2$. This is the `dimension reduction' which gives Section \ref{section Reductions} its name. 

The reader will have noticed that, after replacing $1_{[-10,10]^2}$ with the Lipschitz function $\widetilde{G}$, expression (\ref{the expression that doesn't change}) is not equal to an expression of the form $T_{F,\widetilde{G},N}^L(f_1,\dots,f_4)$, since the map $\Xi$ and the shift $\widetilde{\mathbf{r}}$ are now both on the scene. This complicates matters in the second half of the proof, and thus Theorem \ref{Theorem rational set out version} will not be exactly the same statement as Theorem \ref{Main Theorem chapter 3} apart from the Lipschitz cut-offs. Rather, Theorem \ref{Theorem rational set out version} will bound an object that we will come to denote by $T_{F,G,N}^{L,\Xi,\widetilde{\mathbf{r}}}(f_1,\dots,f_d)$, which will be a general version of expression (\ref{the expression that doesn't change}). The reader may consult Definition \ref{Definition solution count rational separation} for the full definition.\\

In order to make this argument rigorous we will have to verify that in replacing the map $L$ with the map $L\Xi$ we haven't introduced any extra rational relations;\footnote{This is essentially the statement that $L\Xi$ should be purely irrational.} to work out how to relate $A_{L}$ and $A_{L\Xi}$; and to work out how to identify a suitable $\Xi$ in the general case. Furthermore, we will have to carry the non-degeneracy relations (such as $\dist(L,V_{\degen}^*(m,d)) \geqslant c$) through this reparametrisation by $\Xi$, and then establish what the new non-degeneracy notions should be for the pairs $(\Xi,L\Xi)$. This is all done in the (somewhat alarming) Lemma \ref{Lemma generating a purely irrational map}, which has 9 parts. The upper bounds on expressions like $T_{F,G^*,N}^{L\Xi}(1,\dots,1)$ are established earlier, in Lemma \ref{Lemma upper bound involving integral}, with everything combined at the end of Section \ref{section Reductions}. 

 The diagram of the dependency of the various lemmas -- excluding those which are found in Appendices A, B and D, which are somewhat standard -- is as follows. \\

  \begin{center}
\makebox[\textwidth]{\parbox{1.5\textwidth}{
\begin{center}
   \tikzstyle{interface}=[draw, text width=6em,
      text centered, minimum height=2.0em]
   \tikzstyle{daemon}=[draw, text width=6em,
      minimum height=2em, text centered, rounded corners]
   \tikzstyle{lemma}=[draw, text width=5em,
      minimum height=1.5em, text centered, rounded corners]
   \tikzstyle{dots} = [above, text width=6em, text centered]
   \tikzstyle{wa} = [daemon, text width=6em,
      minimum height=2em, rounded corners]
   \tikzstyle{ur}=[draw, text centered, minimum height=0.01em]
   \def\blockdist{2.0}
   \def\edgedist{0.5}
   \begin{tikzpicture}
      \node (wa)[interface]  {Theorem \ref{Main Theorem chapter 3}};
      \path (wa.west)+(-2,2) node (d1)[daemon] {\footnotesize Theorem \ref{Theorem rational set out version}};
      \path(wa.west) + (-2,1) node (d2) [daemon] {\footnotesize Lemma \ref{Lemma replacing F cut-off}};
      
            \path (wa.west) + (-2,0) node (d3) [daemon] {\footnotesize Lemma \ref{Lemma generating a purely irrational map}};
      
            \path (wa.west) + (-2,-1) node (d4) [daemon] {\footnotesize Lemma \ref{Lemma making G lipschitz}};

                \path (wa.west) + (-2,-3) node (d5) [daemon] {\footnotesize Proposition \ref{Proposition maximal rational dimension case of main theorem}};
                
                     \path (d2.west) + (-2,0) node (e1) [daemon] {\footnotesize Lemma \ref{Lemma slightly less crude bound on number of solutions}};
                     \path (d5.west)  + (-4,0) node (e2) [daemon]
                     {\footnotesize Theorem \ref{Linear equations in bounded functions}};
                     
                      \path (d4.west)  + (-2,0) node (e4) [daemon]
                     {\footnotesize Lemma \ref{Lemma upper bound involving integral}};
                     
 \path (d3.west)  + (-2,-2) node (f1) [daemon]
                     {\footnotesize Lemma \ref{Lemma parametrising the image lattice}}; 
                      \path (d3.west)  + (-2,0) node (f2) [daemon]
                     {\footnotesize Proposition \ref{rank matrix}};
                      \path (d3.west)  + (-2,2) node (f3) [daemon]
                     {\footnotesize Lemma \ref{Lemma dual space decomposition}};
                     
                     \path [draw, ->,>=stealth] (d1.east) -- node [above] {} (wa.west) ;
                     \path [draw, ->,>=stealth] (d2.east) -- node [above] {} (wa.west) ;
                     \path [draw, ->,>=stealth] (d3.east) -- node [above] {} (wa.west) ;
                     \path [draw, ->,>=stealth] (d4.east) -- node [above] {} (wa.west) ;
                     \path [draw, ->,>=stealth] (d5.east) -- node [above] {} (wa.west) ;
                     
                     \path [draw, ->,>=stealth] (f1.east) -- node [above] {} (d3.west) ;
                     \path [draw, ->,>=stealth] (f2.east) -- node [above] {} (d3.west) ;
                     \path [draw, ->,>=stealth] (f3.east) -- node [above] {} (d3.west) ;
                     
                     \path [draw, ->,>=stealth] (f2.north) -- node [above] {} (e1.south) ;
                     
                     \path [draw, ->,>=stealth] (f2.south) -- node [above] {} (e4.north) ;
                     
                     \path [draw, ->,>=stealth] (e4.east) -- node [above] {} (d4.west) ;
                     \path [draw, ->,>=stealth] (e1.east) -- node [above] {} (d2.west) ;
                     \path [draw, ->,>=stealth] (e2.east) -- node [above] {} (d5.west) ;

   \end{tikzpicture}
\end{center}}}
   \end{center}

\vspace{5mm}

It remains to resolve Theorem \ref{Theorem rational set out version}, and it turns out that this second part of the proof is significantly more straightforward than the first. In particular neither the statement of Theorem \ref{Theorem rational set out version} nor its proof make any reference to the rational dimension of $L$ nor to the approximation function $A_L$.

The idea is as follows. For a function $f: [N] \longrightarrow [-1,1]$ and a small parameter $\eta$, let $\widetilde{f}: \mathbb{R} \longrightarrow [-1,1]$ denote the function \[ \widetilde{f}(x) = \begin{cases}
f(n) & \text{if } \vert n-x\vert \leqslant \eta \\
0 & \text{otherwise,} \end{cases} \] i.e. $\widetilde{f}$ is a `fattened' version of $f$. Then, for Lipschitz functions $F$ and $G$, let \[\widetilde{T}_{F,G,N}^L(\widetilde{f_1},\dots,\widetilde{f_d}) :=  \frac{1}{N^{d-m}} \int\limits_{\mathbf{x} \in \mathbb{R}^d} \Big(\prod\limits_{j=1}^d \widetilde{f_j} (x_j)\Big) F(\mathbf{x}) G(L\mathbf{x})\, d \mathbf{x}\] represent the `real solution density' for the inequality weighted by the functions $\widetilde{f_j}$. The expression $\widetilde{T}_{F,G,N}^L(\widetilde{f_1},\dots,\widetilde{f_d})$ is more convenient to work with than $T_{F,G,N}^L(f_1,\dots,f_d)$, as we are now working in a setting in which the coefficients of $L$ are invertible.\footnote{This manoeuvre is somewhat analogous to the device used by Green-Tao in \cite{GT10} of passing from $[N]$ to some cyclic group $\mathbb{Z}/N^\prime \mathbb{Z}$, where $N^\prime$ is a prime number larger than $N$. }

The expression $\widetilde{T}_{F,G,N}^L(\widetilde{f_1},\dots,\widetilde{f_d})$ enjoys two properties. Firstly, it is closely related to $T_{F,G,N}^L(f_1,\dots,f_d)$. Indeed, just by expanding out the definition of $\widetilde{f_j}$, we see that
\begin{align}
\label{Lipschitz is needed equation}
\widetilde{T}_{F,G,N}^L(\widetilde{f_1},\dots,\widetilde{f_d}) &= \frac{1}{N^{d-m}} \sum\limits_{ \mathbf{n} \in \mathbb{Z}^d} \Big(\prod\limits_{j=1}^d f_j(n_j)\Big) \int\limits_{\mathbf{y} \in \mathbb{R}^d} F(\mathbf{y}) G(L\mathbf{y}) 1_{[-\eta,\eta]^d}(\mathbf{y} - \mathbf{n}) \, d\mathbf{y} \nonumber \\ &\approx \frac{1}{N^{d-m}} \sum\limits_{ \mathbf{n} \in \mathbb{Z}^d} \Big(\prod\limits_{j=1}^d f_j(n_j)\Big) F(\mathbf{n}) G(L\mathbf{n}) \int\limits_{\mathbf{y} \in \mathbb{R}^d} 1_{[-\eta,\eta]^d}(\mathbf{y} - \mathbf{n}) \, d\mathbf{y} \nonumber \\
& \approx (2 \eta)^dT_{F,G,N}^L(f_1,\dots,f_d),
 \end{align}
 \noindent by using the Lipschitz properties of $F$ and $G$ to replace $F(\mathbf{y})$ and $G(L\mathbf{y})$ by $F(\mathbf{n})$ and $G(L\mathbf{n})$ respectively. This analysis is performed rigorously in Section \ref{section transfer}, and is the only place in the proof where the Lipschitz property of $G$ is used. 
 
 Secondly, $\widetilde{T}_{F,G,N}^L(\widetilde{f_1},\dots,\widetilde{f_d})$ may be bounded above by expressions involving the Gowers norms (over the reals) of the functions $\widetilde{f_j}$. Indeed, after some small manipulations using the compact support of $G$, one ends up with the bound 
 \begin{equation}
 \label{the above expression}\vert\widetilde{T}_{F,G,N}^L(\widetilde{f_1},\dots,\widetilde{f_d}) \vert \ll_G  \frac{1}{N^{d-m}} \int\limits_{ \substack{ \mathbf{x} \in \mathbb{R}^d \\ L\mathbf\mathbf{x} = \mathbf{0}}} \Big(\prod\limits_{j=1}^d \widetilde{f_j} (x_j)\Big) F(\mathbf{x}) \, d \mu(\mathbf{x}),
 \end{equation} where $\mu(\mathbf{x})$ is a suitable measure supported on $\ker L$. The reader will then notice that the right-hand side of (\ref{the above expression}) bears a structural similarity to the expression considered in Theorem \ref{Linear equations in bounded functions non quantitative} above, i.e. in the Generalised von Neumann Theorem for equations with integer coefficients. One may then rejig Green-Tao's proof of Theorem \ref{Linear equations in bounded functions non quantitative} to apply in this setting, and thereby bound $\widetilde{T}_{F,G,N}^L(\widetilde{f_1},\dots,\widetilde{f_d})$ by the Gowers norms of the functions $\widetilde{f_j}$. This is done in Section \ref{section General proof of the real variable von Neumann Theorem}. Finally, there is an elementary argument (Lemma \ref{Lemma linking different Gowers norms}) that relates the Gowers norms of the functions $\widetilde{f_j}$ to the Gowers norms of the original functions $f_j$, thus completing the proof of our result.
 
As will be familiar to readers of \cite{GT10}, the key manoeuvre in analysing (\ref{the above expression}) is parametrising $\ker L$ in a certain special way (in \emph{normal form}, see Section \ref{section normal form}), in order to facilitate repeated applications of the Cauchy-Schwarz inequality. When working over the reals, maintaining quantitative control over the size of the coefficients after this reparametrisation is no longer trivial, and requires the assumption that $\dist(L,V_{\degen}^*(m,d)) \geqslant c$. The details of this piece of quantitative linear algebra are given in Proposition \ref{normal form algorithm} and Appendix \ref{section Rank matrix and normal form: Proofs}. It is this part of the argument which would break down were one to attempt to use Gowers norms to bound inequalities such as the one given by the matrix $L$ in Example \ref{Example bad behaviour}.\\

We have already remarked that, in reality, Theorem \ref{Theorem rational set out version} doesn't just concern the objects $T_{F,G,N}^L(f_1,\dots,f_d)$ but actually concerns the more general objects $T_{F,G,N}^{L,\Xi,\widetilde{\mathbf{r}}}(f_1,\dots,f_d)$, which are similar to (\ref{the expression that doesn't change}). This adds an extra veneer of complication, centred largely around the notion of degeneracy for the pair of maps $(\Xi,L\Xi)$. Matters are resolved by a linear algebra argument in Section \ref{section Degeneracy relations}, relating different notions of degeneracy. \\

The diagram of the dependency of the lemmas used in the proof of Theorem \ref{Theorem rational set out version} is as follows.

  \begin{center}
\makebox[\textwidth]{\parbox{1.5\textwidth}{
\begin{center}
   \tikzstyle{interface}=[draw, text width=6em,
      text centered, minimum height=2.0em]
   \tikzstyle{daemon}=[draw, text width=6em,
      minimum height=2em, text centered, rounded corners]
   \tikzstyle{lemma}=[draw, text width=5em,
      minimum height=1.5em, text centered, rounded corners]
   \tikzstyle{dots} = [above, text width=6em, text centered]
   \tikzstyle{wa} = [daemon, text width=6em,
      minimum height=2em, rounded corners]
   \tikzstyle{ur}=[draw, text centered, minimum height=0.01em]
   \def\blockdist{2.0}
   \def\edgedist{0.5}
   \begin{tikzpicture}
      \node (wa)[interface]  {Theorem \ref{Theorem rational set out version}};
      \path (wa.west)+(-2,1) node (d1)[daemon] {\footnotesize Lemma \ref{Lemma transfer equation}};
      \path (wa.west) + (-2,0) node (d2) [daemon] {\footnotesize Lemma \ref{Lemma linking different Gowers norms}};
        \path (wa.west) + (-2,-1) node (d3) [daemon] {\footnotesize Theorem \ref{Theorem Generalised von Neumann Theorem over reals}};
  \path (wa.west) + (-2,2) node (d4) [daemon] {\footnotesize Lemma \ref{Lemma crude bound on number of solutions}};
  \path (d4.west) + (-2,0) node (e4) [daemon] {\footnotesize Proposition \ref{rank matrix}};
  \path (d3.west) + (-2,2) node (e1) [daemon] {\footnotesize Proposition \ref{Proposition Cauchy}};
  \path (d3.west) + (-2,1) node(e2) [daemon]{\footnotesize Proposition \ref{Proposition separating out the kernel}};
  \path (d3.west) + (-2,0) node(e3) [daemon] {\footnotesize Proposition \ref{normal form algorithm}};
\path (e2.west) + (-2,0) node(f1) [daemon] {\footnotesize Lemma \ref{Quantitative orthonormal basis parametrisation}};
\path (e3.west) + (-2,0) node(f2) [daemon] {\footnotesize Proposition \ref{sec.Analysis algebra quant}};

\path [draw, ->,>=stealth] (d1.east) -- node [above] {} (wa.west) ;
\path [draw, ->,>=stealth] (d2.east) -- node [above] {} (wa.west) ;
\path [draw, ->,>=stealth] (d3.east) -- node [above] {} (wa.west) ;
\path [draw, ->,>=stealth] (d4.east) -- node [above] {} (wa.west) ;
\path [draw, ->,>=stealth] (e1.east) -- node [above] {} (d3.west) ;
\path [draw, ->,>=stealth] (e2.east) -- node [above] {} (d3.west) ;
\path [draw, ->,>=stealth] (e3.east) -- node [above] {} (d3.west) ;
\path [draw, ->,>=stealth] (e4.east) -- node [above] {} (d4.west) ;
\path [draw, ->,>=stealth] (f1.east) -- node [above] {} (e2.west) ;
\path [draw, ->,>=stealth] (f2.east) -- node [above] {} (e3.west) ;
   \end{tikzpicture}
\end{center}}}
   \end{center}
   \vspace{2mm}

The appendices contain some extra material which we felt to be best kept apart from the main narrative. In the case of the first two appendices, they comprise standard results from the literature on Gowers norms and Lipschitz functions, which we include to assist any readers who are unfamiliar with these topics. In the case of Appendices \ref{section Rank matrix and normal form: Proofs} and \ref{section additional linear algebra}, we present a handful of arguments of a linear algebraic nature which, though perhaps not already present in the literature in the exact form we require, are nonetheless easy to establish. Finally, Appendix \ref{section algebraic approximation} concerns the analysis of the approximation function $A_L$ when $L$ has algebraic coefficients. This argument has a similar flavour to Example \ref{Example approx function}, and is included for the sake of completeness. \\

\noindent \textbf{Acknowledgements}: We would like to thank several anonymous referees, for their very careful reading of earlier versions of this work, and Ben Green, for his advice and comments. We also benefited from conversations with Sam Chow, Trevor Wooley, and Yufei Zhao. Some of the paper was completed while the author was a Programme Associate at the Mathematical Sciences Research Institute in Berkeley, who provided excellent working conditions. During part of the project the author was supported by EPSRC grant no. EP/M50659X/1. \\

\section{Upper bounds}
\label{section upper bounds}
This section is devoted to proving three upper bounds on the expression\\ $T_{F,G,N}^L(1,\dots,1)$. For the definition of this quantity, the reader may refer to Definition \ref{Definition solution count}. \\

The following proposition, which represents a quantitative version of the `row-rank equals column-rank' principle, will be useful throughout.

\begin{Proposition}[Rank matrix]
\label{rank matrix}
Let $m,d$ be natural numbers, with $d\geqslant m+1$. Let $c,C$ be positive constants. Then there are positive constants $D_{c,C}, D^\prime_{c,C}$ for which the following holds. Let $L:\mathbb{R}^d\longrightarrow \mathbb{R}^m$ be a surjective linear map, denoted by matrix $(\lambda_{ij})_{i\leqslant m,j\leqslant d}$, and assume that $\Vert L\Vert_\infty \leqslant C$ and $\dist(L,V_{\rank}(m,d)) \geqslant c$. Then there exists a matrix $M$ that is an $m$-by-$m$ submatrix of $L$ and enjoys the following properties:
\begin{enumerate}[(1)]
\item  $\vert\operatorname{det}M\vert \geqslant D_{c,C}$; \\
\item $\Vert M^{-1}\Vert_{\infty} \leqslant D^\prime_{c,C}$.\\
\end{enumerate}
\noindent We call such a matrix $M$ a \emph{rank matrix} of $L$. Furthermore,\\
\begin{enumerate}[(1)]
\setcounter{enumi}{2}
\item Let $\mathbf{v}\in \mathbb{R}^d$ be a vector such that $\mathbf{v}^T$ is in the row-space of $L$, and suppose that $\Vert \mathbf{v}\Vert_\infty \leqslant C_1$ for some positive constant $C_1$. Then for all $i$ in the range $1\leqslant i\leqslant m$ there exist coefficients $a_i$ satisfying $\vert a_i\vert = O_{c,C,C_1}(1)$ such that $\sum\limits_{i=1}^m a_i \lambda_{ij} = v_{j}$ for all $j$ in the range $1\leqslant j\leqslant d$. \\
\end{enumerate}
Finally,\\
\begin{enumerate}
\setcounter{enumi}{3}
\item If $L$ satisfies the stronger hypothesis $\dist(L,V_{\rank}^{\unif}(m,d))\geqslant c$, then, for each $j$, there exists a rank matrix of $L$ that doesn't include the $j^{th}$ column of $L$.
\end{enumerate}
\end{Proposition}
\noindent We defer the proof to Appendix \ref{section Rank matrix and normal form: Proofs}.\\

Our first upper bound is exceptionally crude, but will nonetheless be useful in Section \ref{section transfer}. 

\begin{Lemma}
\label{Lemma crude bound on number of solutions}
Let $N,m,d$ be natural numbers, satisfying $d\geqslant m+1$, and let $c,C,\varepsilon$ be positive constants. Let $L:\mathbb{R}^d \longrightarrow \mathbb{R}^m$ be a  surjective linear map, and suppose that $\Vert L \Vert_\infty \leqslant C$ and $\dist (L,V_{\rank}(m,d)) \geqslant c$. Let $F:\mathbb{R}^d \longrightarrow [0,1]$ and $G:\mathbb{R}^m\longrightarrow [0,1]$ be two functions, with $F$ supported on $[-N,N]^d$ and $G$ supported on $[-\varepsilon,\varepsilon]^m$. Then \[ T_{F,G,N}^L(1,\dots,1) \ll_{c,C,\varepsilon} \Vert G \Vert_\infty.\] 
\end{Lemma}
\begin{proof}
Let $M$ be a rank matrix of $L$ (Proposition \ref{rank matrix}), and suppose without loss of generality that $M$ consists of the first $m$ columns of $L$. For $j$ in the range $m+1\leqslant j\leqslant d$, let the vector $\mathbf{v_{j}} \in\mathbb{R}^m$ be the $j^{th}$ column of the matrix $M^{-1} L$. Then $N^{d-m} T_{F,G,N}^L(1,\dots,1) \leqslant \Vert G\Vert_\infty \cdot Z$, where $Z$ is the number of solutions to $$ \left(\begin{matrix} n_1 \\ \vdots\\ n_m \end{matrix}\right) + \sum\limits_{j=m+1}^{d}\mathbf{v_j}n_j \in M^{-1}([-\varepsilon,\varepsilon]^m)$$ in which $n_1,\dots,n_d$ are integers satisfying $\vert n_1\vert,\dots,\vert n_d\vert \leqslant N$. Fixing a choice of the variables $n_{m+1},\dots,n_d$ forces the vector $(n_1,\dots,n_m)^T$ to lie in a convex region of diameter $O_{c,C,\varepsilon}(1)$. There are at most $O_{c,C,\varepsilon}(1)$ such points, so $Z \ll_{c,C,\varepsilon} N^{d-m}$. The claimed bound follows.
\end{proof}

Our second estimate is a slight strengthening of the above, albeit under stronger hypotheses. 
\begin{Lemma}
\label{Lemma slightly less crude bound on number of solutions}
Let $N,m,d$ be natural numbers, with $d\geqslant m+1$, and let $c,C,\varepsilon$ be positive constants. Let $L:\mathbb{R}^d \longrightarrow \mathbb{R}^m$ be a  surjective linear map, and suppose that $\Vert L \Vert_\infty \leqslant C$ and $\dist (L,V_{\rank}^{\unif}(m,d)) \geqslant c$.  Let $\sigma$ be a real number in the range $0<\sigma < 1/2$. Let $F:\mathbb{R}^d \longrightarrow [0,1]$ and $G:\mathbb{R}^m\longrightarrow [0,1]$ be two functions, with $F$ supported on \[\{\mathbf{x}\in\mathbb{R}^d: \dist(\mathbf{x}, \partial([1,N]^d)) \leqslant \sigma N \}\] and $G$ supported on $[-\varepsilon,\varepsilon]^m$. Then \[ T_{F,G,N}^L(1,\dots,1) \ll_{c,C,\varepsilon} \sigma \Vert G \Vert_\infty.\] 
\end{Lemma}
\begin{proof}
Without loss of generality, we may assume that $F$ is supported on \[\{\mathbf{x}\in\mathbb{R}^d: \Vert\mathbf{x}\Vert_\infty \leqslant 2N, \, \vert x_d - 1\vert \leqslant \sigma N\} \] or \[\{\mathbf{x}\in\mathbb{R}^d: \Vert\mathbf{x}\Vert_\infty \leqslant 2N, \, \vert x_d - N\vert \leqslant \sigma N\} \] 

Consider the first case. By Proposition \ref{rank matrix} there exists a rank matrix $M$ that does not contain the column $d$. By reordering columns, we can assume without loss of generality that $M$ consists of the first $m$ columns of $L$. Continuing as in the proof of Lemma \ref{Lemma crude bound on number of solutions}, for $j$ in the range $m+1\leqslant j\leqslant d$, let the vector $\mathbf{v_{j}} \in\mathbb{R}^m$ be the $j^{th}$ column of the matrix $M^{-1} L$. Then the expression $N^{d-m} T_{F,G,N}^L(1,\dots,1)$ may be bounded above by $\Vert G\Vert_\infty$ times the number of solutions to $$ \left(\begin{matrix} n_1 \\ \vdots\\ n_m \end{matrix}\right) + \sum\limits_{j=m+1}^{d}\mathbf{v_j}n_j \in M^{-1}([-\varepsilon,\varepsilon]^m)$$ satisfying $\vert n_1\vert,\dots,\vert n_{d-1}\vert\leqslant 2N$ and $\vert n_d\vert \leqslant \sigma N$. We conclude as in the previous proof.

In the second case, the relevant equation is \[\left(\begin{matrix} n_1 \\ \vdots\\ n_m \end{matrix}\right) + \sum\limits_{j=m+1}^{d}\mathbf{v_j}n_j + (N-1)\mathbf{v_d} \in M^{-1}([-\varepsilon,\varepsilon]^m),\] in which we count solutions satisfying $\vert n_1\vert,\dots,\vert n_{d-1}\vert\leqslant 2N$ and $\vert n_d - 1\vert \leqslant \sigma N$. We conclude as in the previous proof. 
\end{proof}

Our third estimate is more refined, and will be needed in Section \ref{section Reductions} when we replace the sharp cut-off $1_{[-\varepsilon,\varepsilon]^m}$ with a Lipschitz cut-off. For the definition of the approximation function $A_L$, we refer the reader to Definition \ref{Definition approximation function}. 

\begin{Lemma}
\label{Lemma upper bound involving integral}
Let $N,m,d$ be natural numbers, with $d\geqslant m+1$. Let $c,C,\varepsilon$ be positive constants, and let $\sigma_G$ be a parameter in the range $0<\sigma_G < 1/2$. Suppose that $L:\mathbb{R}^d \longrightarrow \mathbb{R}^m$ is a purely irrational surjective linear map, satisfying $\Vert L\Vert_\infty \leqslant C$ and $\dist(L,V_{\rank}(m,d)) \geqslant c$. Let $A_L$ denote the approximation function of $L$. Let $F:\mathbb{R}^d \longrightarrow [0,1]$ be supported on $[-N,N]^d$, and let $G:\mathbb{R}^m\longrightarrow [0,1]$ be a Lipschitz function, with Lipschitz constant $O(1/\sigma_G)$, supported on $[-\varepsilon,\varepsilon]^m$. Assume further that $\int_{\mathbf{x}} G(\mathbf{x}) \, d\mathbf{x} = O_\varepsilon(\sigma_G)$. Then for all $\tau_2$ in the range $0<\tau_2\leqslant 1$, 
\begin{equation*}
 T_{F,G,N}^L(1,\dots,1) \ll_{c,C,\varepsilon} \sigma_G + \frac{\tau_2^{1/2}}{\sigma _G} +  \frac{\tau_2^{-O(1)}A_L(\Omega_{c,C}(1),\tau_2)^{-1}}{N}.
 \end{equation*}
\end{Lemma}

\begin{proof}
Following the proof of Lemma \ref{Lemma crude bound on number of solutions} verbatim, we arrive at the bound 
\begin{equation}
\label{we arrive at the bound}
T_{F,G,N}^L(1,\dots,1) \ll_{c,C,\varepsilon} \frac{1}{N^{d-m}}\sum\limits_{\substack{n_{m+1},\dots,n_d \in \mathbb{Z} \\ \vert n_{m+1}\vert,\dots,\vert n_{d}\vert \leqslant N}} \widetilde{G}(\sum_{j=m+1}^d \mathbf{v_j} n_j),
\end{equation} where $\mathbf{v_j}$ denotes the $j^{th}$ column of the matrix $M^{-1}L$, and $\widetilde{G}:\mathbb{R}^m \longrightarrow [0,1]$ denotes the function \[ \widetilde{G}(\mathbf{x}) = \sum\limits_{\mathbf{a} \in \mathbb{Z}^m} (G\circ M)(\mathbf{a} + \mathbf{x}).\] It remains to estimate the right-hand side of (\ref{we arrive at the bound}). \\

We may consider $\widetilde{G}$ as a function on $\mathbb{R}^m/\mathbb{Z}^m$. With respect to the metric $\Vert \mathbf{x}\Vert_{\mathbb{R}^m/\mathbb{Z}^m}$, $\widetilde{G}$ is Lipschitz with Lipschitz constant $O_{c,C,\varepsilon}(1/\sigma_G)$. Also, \[\int\limits_{\mathbf{x}\in [0,1)^m} \widetilde{G}(\mathbf{x}) \, d\mathbf{x} = \int\limits_{\mathbf{x}\in \mathbb{R}^m} (G \circ M)(\mathbf{x}) \, d\mathbf{x} = O_{c,C,\varepsilon}(\sigma_G).\] By \cite[Lemma A.9]{GrTa08}, which we recall in Lemma \ref{Fourier transforms of Lipschitz functions on tori}, for any $X$ at least $2$ we may write
\begin{equation}
\label{equation discrete fourier approximation}
\widetilde{G}(\mathbf{x}) = \sum\limits_{ \substack{\mathbf{k} \in\mathbb{Z}^m \\ \Vert \mathbf{k} \Vert_\infty \leqslant X}} b_X(\mathbf{k}) e(\mathbf{k} \cdot \mathbf{x}) + O_{c,C,\varepsilon}\Big(\frac{ \log X}{\sigma_G X}\Big),
\end{equation}  where $b_X(\mathbf{k})\in\mathbb{C}$ and satisfies $\vert b_X(\mathbf{k}) \vert = O(1)$. Moreover\footnote{This final fact is not given explicitly in the statement of \cite[Lemma A.9]{GrTa08}, although it is given in the proof. In any case, it may be immediately deduced from (\ref{equation discrete fourier approximation}), by letting $X$ tend to infinity and integrating (\ref{equation discrete fourier approximation}) over all $\mathbf{x} \in \mathbb{R}^m/\mathbb{Z}^m$.} $b_X(\mathbf{0}) = \int_{\mathbf{x} \in [0,1)^m} \widetilde{G}(\mathbf{x}) \, d\mathbf{x}$. 

Returning to (\ref{we arrive at the bound}), we see that for any $X$ at least $2$ we may write 
\begin{equation}
\label{summing geomtric progression}
T_{F,G,N}^L(1,\dots,1) \ll_{c,C,\varepsilon} \sigma_G + \frac{\log X}{\sigma_G X} + X^{O(1)}\max\limits_{\substack{\mathbf{k} \in \mathbb{Z}^m \\ 0<\Vert \mathbf{k} \Vert_\infty \leqslant X}}\Big( \prod\limits_{j=m+1}^d \min (1,  N^{-1}\Vert \mathbf{k} \cdot \mathbf{v_j}\Vert^{-1}_{\mathbb{R}/\mathbb{Z}})\Big),
\end{equation} where the final error term comes from summing over the arithmetic progressions $[-N,N] \cap \mathbb{Z}$.

It remains to relate the final error term of (\ref{summing geomtric progression}) to the approximation function $A_L$. Since $L$ is purely irrational, 
\begin{align}
A_L(\tau_1,\tau_2) &=\inf\limits_{\substack{\varphi \in (\mathbb{R}^m)^*\\ \tau_1 \leqslant  \Vert \varphi \Vert_\infty \leqslant \tau_2^{-1}} } \dist (L^* \varphi, (\mathbb{Z}^d)^T )\nonumber.
\end{align}
\noindent Let $\tau_2$ be in the range $0 < \tau_2 \leqslant 1$. Then there exist positive parameters $D$ and $D^\prime$, depending only on $c$ and $C$, for which
\begin{align}
\label{purely irrational is really important here}
\min\limits_{\substack{ \mathbf{k} \in \mathbb{Z}^m \\ 0< \Vert \mathbf{k}\Vert_\infty \leqslant D\tau_2^{-1}}} \max( \{\Vert \mathbf{k}\cdot\mathbf{v_j}\Vert_{\mathbb{R}/\mathbb{Z}}: m+1\leqslant j\leqslant d\}) & = \min\limits_{\substack{ \mathbf{k} \in \mathbb{Z}^m \\ 0< \Vert \mathbf{k}\Vert_\infty \leqslant D\tau_2^{-1}}} \dist(\mathbf{k}^T M^{-1} L, (\mathbb{Z}^d)^T) \nonumber\\
& \geqslant \inf\limits_{\substack{ \mathbf{k} \in \mathbb{R}^m \\ 1\leqslant \Vert \mathbf{k}\Vert_\infty \leqslant D\tau_2^{-1}}} \dist(\mathbf{k}^T M^{-1} L, (\mathbb{Z}^d)^T) \nonumber \\
&\geqslant \inf\limits_{\substack{ \mathbf{k} \in \mathbb{R}^m \\ D^\prime \leqslant \Vert \mathbf{k}\Vert_\infty \leqslant \tau_2^{-1}}} \dist(\mathbf{k}^T L, (\mathbb{Z}^d)^T)\nonumber \\
& =  A_L(D^\prime, \tau_2).
\end{align}
\noindent Letting $X = D \tau_2^{-1}$, and substituting the bound (\ref{purely irrational is really important here}) into (\ref{summing geomtric progression}), one derives
\begin{equation*}
T_{F,G,N}^L(1,\dots,1) \ll_{c,C,\varepsilon} \sigma_G + \frac{\tau_2^{1/2}}{\sigma _G} +  \frac{\tau_2^{-O(1)}A_L(\Omega_{c,C}(1),\tau_2)^{-1}}{N}
\end{equation*}
\noindent as required.
\end{proof}

The relations (\ref{purely irrational is really important here}) formalise the estimate (\ref{approx function appears}), which we first discussed when introducing the approximation function $A_L$ in Definition \ref{Definition approximation function}. With the details all here, one can now see that it would have been enough to define the approximation function, at least if $L$ is purely irrational, to be the function \begin{equation}
\label{equivalent object}
 \tau_2 \mapsto \min\limits_{\substack{ \mathbf{k} \in \mathbb{Z}^m \\ 0< \Vert \mathbf{k}\Vert_\infty \leqslant \tau_2^{-1}}} \dist(\mathbf{k}^T M^{-1} L, (\mathbb{Z}^d)^T).
 \end{equation} 
 
One might now be concerned that, in defining $A_L$ using real vectors $\varphi$ rather than integer vectors $\mathbf{k}$, we might have constructed a much weaker object than (\ref{equivalent object}), making (\ref{purely irrational is really important here}) a wasteful step in our estimation. This is not the case, because if $\varphi \in (\mathbb{R}^m)^*$ and\[\dist(L^* \varphi, (\mathbb{Z}^d)^T) \leqslant \delta\] then $\dist(L^*(M^*)^{-1} M^* \varphi, (\mathbb{Z}^d)^T) \leqslant \delta$, and so in particular $\dist (M^* \varphi , (\mathbb{Z}^m)^T) \leqslant \delta$ (as $M$ is a rank matrix). Letting $\mathbf{k}^T \in (\mathbb{Z}^m)^T$ be the nearest integer vector to $M^*\varphi$, we have that \[\dist (\mathbf{k}^T M^{-1} L, (\mathbb{Z}^d)^T) \ll_{c,C} \delta.\] So, up to some constants depending on $c$ and $C$, there is essentially no difference between working with Definition \ref{Definition approximation function} or with (\ref{equivalent object}).

Restricting to integer vectors $\mathbf{k}$ may seem more natural from the point of view of diophantine approximation, but on the other hand the expression (\ref{equivalent object}) depends on the choice of the particular rank matrix $M$, which is not canonical. It was more to our taste to present a definition of $A_L$ which was intrinsic to $L$. Lemma 8.1 of our follow-up paper \cite{Wa19} is also a setting in which having real vectors in the definition of $A_L$ seems to be more natural.\\

It is also worth highlighting the exact moment in the proof of Lemma \ref{Lemma upper bound involving integral} in which it was vital that $L$ was purely irrational. Considering expression (\ref{summing geomtric progression}), if $L$ was not purely irrational and $X$ was bigger than the rational complexity of $L$ then the final error term is just $X^{O(1)}$, which is not $o(1)$ as $N\rightarrow \infty$. \\

\section{Normal form}
\label{section normal form}

In this section we recall a technical notion from \cite{GT10} that those authors refer to as \emph{normal form}. In Section \ref{section General proof of the real variable von Neumann Theorem} we will need to appeal to a quantitative refinement of this notion, which we also develop here.

Let $\Psi:\mathbb{R}^n \longrightarrow \mathbb{R}^m$ be a linear map. Putting the standard coordinates on $\mathbb{R}^n$ and $\mathbb{R}^m$, we may write  $(\psi_1,\dots,\psi_m) := \Psi:\mathbb{R}^n \longrightarrow \mathbb{R}^m$ as a system of homogeneous linear forms. The crux of the theory from \cite{GT10} is that, provided $\Psi$ is of so-called `finite Cauchy-Schwarz complexity', $\Psi$ may be reparametrised in such a way that it interacts particularly well with certain applications of the Cauchy-Schwarz inequality (see Proposition \ref{Proposition Cauchy}). Below we will give a brief overview of this terminology, before introducing our own quantitative versions; a much fuller discussion may be found in \cite[Section 1]{GT10} and \cite{GoWo10}.

In words, a reparametrisation into normal form is one in which each linear form is the only one that mentions all of its particular collection of variables. For example, the forms
\begin{align}
\label{not normal form}
\psi_1(t,u,v) & = u + v\nonumber\\
\psi_2(t,u,v) & = v + t\nonumber\\
\psi_3(t,u,v) & = u + t\nonumber\\
\psi_4(t,u,v) & = u+v+t
\end{align}
\noindent are in normal form with respect to $\psi_4$, since $\psi_4$ is the only form to utilise all three of the variables. However, this system is not in normal form with respect to $\psi_3$, say. However, the system 
\begin{align}
\label{not normal form becomes normal}
\psi_1(t,u,v,w) & = u + v + 2w\nonumber\\
\psi_2(t,u,v,w) & = v + t - w\nonumber\\
\psi_3(t,u,v,w) & = u + t - w\nonumber\\
\psi_4(t,u,v,w) & = u+v+t,
\end{align}
\noindent that parametrises the same subspace of $\mathbb{R}^4$, \emph{is} in normal form for all $i$.

We repeat the precise definition from \cite{GT10}.

\begin{Definition}
\label{Definition normal form}
Let $m,n$ be natural numbers, and let $(\psi_1,\dots,\psi_m) = \Psi:\mathbb{R}^n\longrightarrow \mathbb{R}^m$ be a system of homogeneous linear forms. Let $i\in [m]$. We say that $\Psi$ is in normal form with respect to $\psi_i$ if there exists a non-negative integer $s$ and a collection $J_i \subseteq \{ \mathbf{e_1},\dots,\mathbf{e_n}\}$ of the standard basis vectors, satisfying $\vert J_i\vert = s+1$, such that $$\prod\limits_{\mathbf{e}\in J_i} \psi_{i^\prime}(\mathbf{e})$$ is non-zero when $i^\prime = i$ and vanishes otherwise. We say that $\Psi$ is in normal form if it is in normal form with respect to $\psi_i$ for every $i$. 
\end{Definition}
\noindent Let us also recall what it means for a certain system of forms $\Psi^\prime$ to extend the system of forms $\Psi$.
\begin{Definition}
For a system of homogeneous linear forms $(\psi_1,\dots,\psi_m) = \Psi:\mathbb{R}^n\longrightarrow \mathbb{R}^m$, an extension $(\psi_1^\prime,\dots,\psi_m^\prime) = \Psi^\prime: \mathbb{R}^{n^\prime} \longrightarrow \mathbb{R}^m$ is a system of homogeneous linear forms on $\mathbb{R}^{n^\prime}$, for some $n^\prime$ with $n^\prime\geqslant n$, such that 
\begin{enumerate}[(1)]
\item $\Psi^\prime(\mathbb{R}^{n^\prime}) = \Psi(\mathbb{R}^n)$;
\item if we identify $\mathbb{R}^n$ with the subset $\mathbb{R}^{n}\times \{ 0\}^{n^\prime - n}$ in the obvious manner, then $\Psi$ is the restriction of $\Psi^\prime$ to this subset.
\end{enumerate}
\end{Definition}
\noindent The paper \cite{GT10} includes a result (Lemma 4.4) on the existence of extensions in normal form, but we will need a quantitative refinement of this analysis.\\

The reader will note from examples (\ref{not normal form}) and (\ref{not normal form becomes normal}) that the property of `being in normal form' is a property of the parametrisation, and not of the underlying space that is being parametrised. It is natural to wonder whether there is some property of a space that can enable one to find a parametrisation in normal form, even if the original parametrisation is not. Fortunately there is such a notion, and it is the notion of finite Cauchy-Schwarz complexity\footnote{In \cite{GT10} this is just called `complexity'.}  introduced in \cite{GT10}. We introduce this notion in the following definitions, which we have phrased in such a way as to help us formulate a quantitative version.

\begin{Definition}[Suitable partitions]
\label{Definition suitable partitions}
Let $m,n$ be natural numbers, with $m\geqslant 2$, and let $(\psi_1,\dots,\psi_m) = \Psi:\mathbb{R}^n \longrightarrow \mathbb{R}^m$ be a system of homogeneous linear forms. Fix $i\in [m]$. Let $\mathcal{P}_i$ be a partition of $[m]\setminus \{i\}$, i.e. $$[m]\setminus \{i\} = \bigcup\limits_{k=1}^{s+1} \mathcal{C}_k$$ for some $s$ satisfying $0\leqslant s \leqslant m-2$ and some disjoint sets $\mathcal{C}_k$. We say that $\mathcal{P}_i$ is \emph{suitable} for $\Psi$ if $$\psi_i \notin \operatorname{span}_\mathbb{R}(\psi_j: j\in \mathcal{C}_k)$$ for any $k$. 
\end{Definition}

\begin{Definition}[Degeneracy varieties]
\label{Definition degeneracy varieties}
Let $m,n$ be natural numbers, with $m\geqslant 2$. Let $\mathcal{P}_i$ be a partition of $[m]\setminus \{i\}$. We define the \emph{$\mathcal{P}_i$-degeneracy variety} $V_{\mathcal{P}_i}$ to be the set of all the systems of homogeneous linear forms $\Psi:\mathbb{R}^n\rightarrow \mathbb{R}^m$ for which $\mathcal{P}_i$ is not suitable for $\Psi$. Finally, the \emph{degeneracy variety} $V_{\degen}(n,m)$ is given by \[ V_{\degen}(n,m): = \bigcup\limits_{i = 1}^m \bigcap\limits_{\mathcal{P}_i} V_{\mathcal{P}_i},\] where the inner intersection is over all possible partitions $\mathcal{P}_i$.
\end{Definition}

It is easy to observe that $\Psi \in V_{\degen}(n,m)$ if and only if, for some $i\neq j$, $\psi_i$ is a real multiple of $\psi_j$. This also yields the following:

\begin{Proposition}[Relating $V_{\degen}^*(m,d)$ and $V_{\degen}(d-m,d)$]
\label{Proposition easy degeneracy relation}
Let $m,d$ be natural numbers with $d \geqslant m+2$, and let $L:\mathbb{R}^d \longrightarrow \mathbb{R}^m$ be a surjective linear map. Let $\Psi: \mathbb{R}^{d-m} \longrightarrow \mathbb{R}^d$ be any system of homogeneous linear forms whose image is $\ker L$. Then $L \in V_{\degen}^*(m,d)$ if and only if $\Psi \in V_{\degen}(d-m,d).$
\end{Proposition}
\begin{proof}
We know that $L \in V_{\degen}^*(m,d)$ if and only if there exist some non-zero vector $\mathbf{e_i}^* - \lambda{e_j}^* \in L^*((\mathbb{R}^m)^*)$. But $ L^*((\mathbb{R}^m)^*) = (\ker L)^0 = (\Psi(\mathbb{R}^{d-m}))^0$, so this occurs if and only if $\psi_i = \lambda \psi_j$ for some $i$ and $j$. 
\end{proof}
\noindent We will prove a more general version of this statement in Lemma \ref{Lemma non quantitative equivalence of degeneracies}.\\

In \cite[Definition 1.5]{GT10}, the authors refer to those $\Psi\in V_{\degen}(n,m)$ as having infinite Cauchy-Schwarz complexity, and develop their theory for $\Psi\notin V_{\degen}(n,m)$. As we did for describing degeneracy properties of $L$, we need to quantify such a notion. 

\begin{Definition}[$c_1$-Cauchy-Schwarz complexity]
\label{Definition Cauchy Schwarz complexity}
Let $m,n$ be natural numbers, with $m\geqslant 3$, and let $c_1$ be a positive constant. Let $(\psi_1,\dots,\psi_m) = \Psi:\mathbb{R}^n \longrightarrow \mathbb{R}^m$ be a system of homogeneous linear forms. For $i\in [ m]$, we define a quantity $s_i$ either by defining $s_i+1$ to be the minimal number of parts in a partition $\mathcal{P}_i$ of $[m]\setminus \{i\}$ such that $\operatorname{dist}(\Psi,V_{\mathcal{P}_i})\geqslant c_1$, or by $s_i = \infty$ if no such partition exists. Then we define $s:=\operatorname{max}(1,\operatorname{max}_i s_i)$. We say that $s$ is the $c_1$-Cauchy-Schwarz complexity of $\Psi$. 
\end{Definition}
\noindent We remark, for readers familiar with \cite{GT10}, that we preclude the `complexity $0$' case. This is for a mundane technical reason, that occurs when absorbing the exponential phases in Section \ref{section General proof of the real variable von Neumann Theorem}, when it will be convenient that $s+1\geqslant 2$. This is why we need the condition $m\geqslant 3$ in the above definition. We also take this opportunity to note that if $s$ satisfies the above definition, and $s\neq\infty$, then $2\leqslant s+1\leqslant m-1$. \\

We note an easy consequence of these definitions.
\begin{Lemma}
\label{Lemma different forms of non degen}
Let $m,n$ be natural numbers, with $m\geqslant 3$, and let $c_1$ be a positive constant. Let $(\psi_1,\dots,\psi_m) = \Psi:\mathbb{R}^n\longrightarrow \mathbb{R}^m$ be a system of homogeneous linear forms. Suppose that $\operatorname{dist}(\Psi, V_{\degen}(n,m))\geqslant c_1$. Then $\Psi$ has finite $c_1$-Cauchy-Schwarz complexity.
\end{Lemma}
\begin{proof}
We have already observed that $\Psi \in V_{\degen}(n,m)$ if and only if, for some $i\neq j$, $\psi_i$ is a real multiple of $\psi_j$. From now until the end of the proof, fix $\mathcal{P}_i$ to be the partition of $[m]\setminus \{i\}$ in which every form $\psi_k$ is in its own part. Our initial observation then implies that $\Psi\in V_{\degen}(n,m)$ if and only if $\Psi\in V_{\mathcal{P}_i}$ for some $i$. So $\operatorname{dist}(\Psi, V_{\degen}(n,m))\geqslant c_1$ implies that $\operatorname{dist}(\Psi, V_{\mathcal{P}_i})\geqslant c_1$ for all $i$. Therefore, by using these partitions $\mathcal{P}_i$ in Definition \ref{Definition Cauchy Schwarz complexity}, we conclude that $\Psi$ has finite $c_1$-Cauchy-Schwarz complexity. 
\end{proof}

After having built up these definitions, we state the key proposition on the existence of normal form extensions to systems of real linear forms. We remind the reader that all implied constants may depend on the dimensions of the underlying spaces. 

\begin{Proposition}[Normal form algorithm]
\label{normal form algorithm}
Let $m,n$ be natural numbers, with $m\geqslant 3$, and let $c_1,C_1$ be positive constants. Let $(\psi_1,\dots,\psi_m) = \Psi:\mathbb{R}^n \longrightarrow \mathbb{R}^m$ be a system of homogeneous linear forms, and suppose that the coefficients of $\Psi$ are bounded above in absolute value by $C_1$. Furthermore, suppose that $\Psi$ has $c_1$-Cauchy-Schwarz complexity $s$, for some finite $s$. Then, for each $i \in [m]$, there is an extension $\Psi^\prime:\mathbb{R}^{n^\prime}\longrightarrow \mathbb{R}^m$ such that:
\begin{enumerate}[(1)]
\item $n^\prime = n+s+1 \leqslant n+m-1$; 
\item $\Psi^\prime$ is of the form $$\Psi^\prime(\mathbf{u}, w_1,\dots,w_{s+1}): = \Psi(\mathbf{u} + w_1\mathbf{f_1}+ \cdots + w_{s+1}\mathbf{f_{s+1}})$$ for some vectors $\mathbf{f_k}\in \mathbb{R}^n$, such that $\Vert \mathbf{f_k}\Vert_\infty=O_{c_1,C_1}(1)$ for every $k$;
\item $\Psi^\prime$ is in normal form with respect to $\psi_i^\prime$;
\item $\psi_i^\prime(\mathbf{0},\mathbf{w}) = w_1 + \cdots + w_{s+1}$.
\end{enumerate}
\end{Proposition}
\noindent The proof is deferred to Appendix \ref{section Rank matrix and normal form: Proofs}, as it is very similar to the proof from \cite{GT10} (although with one important extra subtlety, which we mention in the appendix). \\

We conclude this discussion of normal form by noting an example of a system of homogeneous linear forms that may be reparametrised in normal form, but without quantitative control over the resulting extension. 

Indeed, take $\iota(N)$ to be some function such that $\iota(N)\rightarrow \infty$ as $N\rightarrow \infty$. Consider the forms
\begin{align*}
\psi_1(u_1,u_2,u_3) &= (1+\iota(N)^{-1})u_1 + u_2\\
\psi_2(u_1,u_2,u_3) &= u_1 + u_2\\
\psi_3(u_1,u_2,u_3) & = u_3.
\end{align*}
\noindent and let $\Psi := (\psi_1,\psi_2,\psi_3)$. Notice that $\dist(\Psi,V_{\degen}(3,3)) \rightarrow 0$ as $N\rightarrow \infty$. Therefore, for any $c_1 >0$, if $N$ is large enough then $\Psi$ does not have finite $c_1$-Cauchy-Schwarz complexity. One may nonetheless construct a normal form reparametrisation 
\begin{align*}
\psi^\prime_1(u_1,u_2,u_3,w_1,w_2)& = (1+\iota(N)^{-1})u_1 + u_2 + w_1\\
\psi^\prime_2(u_1,u_2,u_3,w_1,w_2)& =  u_1 + u_2 + w_2\\
\psi^\prime_3(u_1,u_2,u_3,w_1,w_2)& = u_3.
\end{align*}
\noindent However, since $$\Psi^\prime( u_1,u_2,u_3,w_1,w_2) = \Psi(u_1 + \iota(N)w_1- \iota(N)w_2, u_2 - \iota(N)w_1 + (\iota(N) + 1)w_2,u_3),$$ $\Psi^\prime$ is not obtained by bounded shifts of the $u_i$ variables, and so (if $N$ is large enough) it fails to satisfy part (2) of the conclusion of the above proposition. Such an extension $\Psi^\prime$ would not be suitable for our requirements in Section \ref{section General proof of the real variable von Neumann Theorem}. 

\begin{Remark}
\emph{In \cite{GT10}, the simple algorithm that constructs normal form extensions with respect to a fixed $i$ may easily be iterated, and so the authors work with systems that are in normal form with respect to every index $i$. A careful analysis of the proof in Appendix C of \cite{GT10} demonstrates that it is sufficient for $\Psi$ merely to admit, for each $i$ separately, an extension that is in normal form with respect to $\psi_i$, but this is of little consequence in \cite{GT10}. Yet certain quantitative aspects of the iteration of the normal form algorithm, critical to our application of these ideas, are not immediately clear to us. We have stated Proposition \ref{normal form algorithm} for normal forms only with respect to a single $i$, in order to avoid this technical annoyance.} \\
\end{Remark}

\section{Dimension reduction}
\label{section Reductions}
As we described in our proof strategy (Section \ref{Section proof strategy}), in this section we reduce Theorem \ref{Main Theorem chapter 3} to a different result, namely Theorem \ref{Theorem rational set out version}. This second theorem will be simpler in one key respect: the replacement of sharp cut-offs by Lipschitz cut-offs. It is the proof of Theorem \ref{Theorem rational set out version} in which the Lipschitz property is actually used, and this will begin in Section \ref{section transfer}. Any reader only wishing to consider the case of diophantine inequalities with Lipschitz cut-offs may eschew Section \ref{section Reductions} of this paper entirely. \\

We begin by dismissing the case of maximal rational dimension. 

\begin{Proposition}
\label{Proposition maximal rational dimension case of main theorem}
Theorem \ref{Main Theorem chapter 3} holds under the additional assumption that $L$ has rational dimension $m$. 
\end{Proposition}

To prove this, we will appeal to a quantitative version of Theorem \ref{Linear equations in bounded functions non quantitative}. 

\begin{Theorem}[Generalised von Neumann Theorem for rational forms (quantitative version)]
\label{Linear equations in bounded functions}
Let $N,m,d$ be natural numbers, satisfying $d\geqslant m+2$, and let $C_1,C_2$ be positive constants. Let $S = S(N)$ be an $m$-by-$d$ matrix with integer coefficients, satisfying $\Vert S\Vert_\infty \leqslant C_1$, and let $\mathbf{r}\in \mathbb{Z}^m$ be some vector with $\Vert \mathbf{r} \Vert_\infty \leqslant C_2N$. Suppose $S$ has rank $m$, and $S\notin V^*_{\degen}(m,d)$. Let $K\subseteq [-N,N]^d$ be convex. Then there exists some natural number $s$ at most $d-2$ that satisfies the following. Let $f_1,\dots,f_d:[N]\longrightarrow \mathbb{C}$ be arbitrary functions with $\Vert f_j\Vert_\infty \leqslant 1$ for all $j$, and assume that \[ \min_j \Vert f_j\Vert_{U^{s+1}[N]} \leqslant \rho \] for some $\rho$ in the range $0<\rho \leqslant 1$. Then $$\frac{1}{N^{d-m}}\sum\limits_{\substack{\mathbf{n}\in \mathbb{Z}^d \cap K\\ S\mathbf{n} = \mathbf{r}}} \prod\limits_{j=1}^d f_j(n_j) \ll_{C_1,C_2}  \rho^{\Omega(1)} + o_\rho(1) $$ as $N\rightarrow \infty$. Furthermore, the $o_{\rho}(1)$ term may be bounded above by $\rho^{-O(1)} N^{-\Omega(1)}$. 
\end{Theorem}
Let us sketch a proof of this result, assuming a certain familiarity with the methods and terminology of \cite{GT10}.
\begin{proof}[Proof sketch of Theorem \ref{Linear equations in bounded functions}]
One follows the proof of Theorem 1.8 of \cite{GT10}. Firstly, recall that in our language, the non-degeneracy condition in the statement of Theorem 1.8 of \cite{GT10} is exactly the condition that $S\notin V^*_{\degen}(m,d)$. One then follows the same linear algebraic reductions as those used in Section 4 of \cite{GT10} to reduce Theorem 1.8 to Theorem 7.1 of the same paper (the Generalised von Neumann Theorem). 

Theorem 7.1 may then be considered solely in the case of bounded functions $f_j$, as in \cite[Exercise 1.3.23]{Ta12}, rather than in the more general case of functions bounded by a pseudorandom measure. It is clear from the proof that, in this more restricted setting, the $\kappa(\rho)$ term that appears in the statement may be replaced by a polynomial dependence, and the $o_{\rho}(1)$ term may be bounded above by $\rho^{-O(1)}N^{-\Omega(1)}$.   

This settles Theorem \ref{Linear equations in bounded functions}, where $s$ is the Cauchy-Schwarz complexity of some system of forms $(\psi_1,\dots,\psi_d)$ that parametrises $\ker S$. But $s$ is at most $d-2$, as any system of $d$ forms with finite Cauchy-Schwarz complexity has Cauchy-Schwarz complexity at most $d-2$. Therefore Theorem \ref{Linear equations in bounded functions} is proved. 
\end{proof}

Now let us use Theorem \ref{Linear equations in bounded functions} to resolve Proposition \ref{Proposition maximal rational dimension case of main theorem}. 

\begin{proof}[Proof of Proposition \ref{Proposition maximal rational dimension case of main theorem}]
Let $L$ be as in Theorem \ref{Main Theorem chapter 3}, and assume that $L$ has rational dimension $m$ and rational complexity at most $C$. Let $\Theta:\mathbb{R}^m\longrightarrow \mathbb{R}^m$ be some linear isomorphism satisfying $\Theta L(\mathbb{Z}^d) \subseteq \mathbb{Z}^m$ and $\Vert \Theta \Vert_\infty \leqslant C$. Let $M$ be a rank matrix of $L$ (Proposition \ref{rank matrix}). Then the matrix $M^{-1}L$ satisfies $\Vert M^{-1} L\Vert_\infty \ll_{c,C} 1$ and has rational dimension $m$, since \\$((\Theta M )\circ (M^{-1}L))(\mathbb{Z}^d) = \Theta L(\mathbb{Z}^d) \subseteq \mathbb{Z}^m$. The matrix $M^{-1}L$ also has rational complexity $O_{c,C}(1)$. Therefore, replacing $L$ with $M^{-1} L$, we may assume that the first $m$ columns of $L$ form the identity matrix.

As in Lemma \ref{Lemma full rational dimension}, we write $\Theta L = S$, where $S$ has integer coefficients and $\Vert \Theta \Vert_\infty \ll_{c,C} 1$. Hence $\Vert S\Vert_\infty \ll_{c,C}1$. But $\Theta$ must also have integer coefficients, as the first $m$ columns of $L$ form the identity matrix, and hence $\Vert \Theta^{-1} \Vert _\infty \ll_{c,C} 1$ as well. Note finally that $S \notin V_{\degen}^*(m,d)$, since $L \notin V_{\degen}^*(m,d)$.

Now, suppose that $G:\mathbb{R}^m \longrightarrow [0,1]$ is the indicator function of some convex domain $D$, with $D \subseteq [-\varepsilon,\varepsilon]^m$. Then there are at most $O_{c,C,\varepsilon}(1)$ possible vectors $\mathbf{r} \in \mathbb{Z}^m$ such that $\mathbf{r} \in S(\mathbb{Z}^d) \cap \Theta(D).$ Let $R$ be the set of all such vectors.  Therefore, with $F$ being the indicator function of the set $[1,N]^d$, we have
 \begin{equation}
 T_{F,G,N}^L(f_1,\dots,f_d) = \sum\limits_{\mathbf{r} \in R} \sum\limits_{\substack{\mathbf{n} \in [N]^d \\ S\mathbf{n} = \mathbf{r}}}  \prod\limits_{j=1}^d f_j(n_j) \ll_{c,C,\varepsilon} \rho^{\Omega(1)} + o_{\rho}(1)
\end{equation}
\noindent as $N\rightarrow \infty$, by Theorem \ref{Linear equations in bounded functions}. The $o_{\rho}(1)$ term may be bounded above by $\rho^{-O(1)} N^{-\Omega(1)}$. This is the desired conclusion of Theorem \ref{Main Theorem chapter 3} in the case when $L$ has rational dimension $m$. 
\end{proof}

Having dismissed this case, we prepare to state Theorem \ref{Theorem rational set out version}. We begin with a definition that generalises Definition \ref{Definition solution count}.
\begin{Definition}
\label{Definition solution count rational separation}
Let $N,m,d,h$ be natural numbers, with $d\geqslant h\geqslant m+2$. Let $\varepsilon$ be positive. Let $\Xi = (\xi_1,\dots,\xi_d) :\mathbb{R}^h\longrightarrow \mathbb{R}^d$ and $L:\mathbb{R}^{h}\longrightarrow \mathbb{R}^m$ be linear maps. Let $F:\mathbb{R}^{h}\rightarrow [0,1]$ and $G:\mathbb{R}^m\rightarrow [0,1]$ be two functions, with $F$ supported on $[-N,N]^h$ and $G$ compactly supported. Let $\widetilde{\mathbf{r}} \in \mathbb{Z}^d$ be some vector, and let $f_1,\dots,f_{d}:\mathbb{R}\longrightarrow [-1,1]$ be arbitrary functions. We then define
\begin{equation}
\label{equation solution count rational separation}
T^{L,\Xi,\widetilde{\mathbf{r}}}_{F,G,N}(f_1,\dots,f_{d}) := \frac{1}{N^{h-m}}\sum\limits_{\mathbf{n} \in \mathbb{Z}^h}\Big(\prod\limits_{j=1}^{d}f_j(\xi_j(\mathbf{n})+ \widetilde{\mathbf{r}}_j)\Big)F(\mathbf{n})G(L\mathbf{n}).
\end{equation}
\end{Definition}

In the paper so far we have introduced many degeneracy relations (Definitions \ref{Definition rank variety}, \ref{Definition dual degeneracy variety}, \ref{Definition degeneracy varieties}). In order to state Theorem \ref{Theorem rational set out version}, we must introduce another. 

\begin{Definition}[Dual pair degeneracy variety]
\label{Definition dual pair degeneracy variety}
Let $m,d,h$ be natural numbers satisfying $d\geqslant h\geqslant m+2$. Let $\mathbf{e_1}, \dots,\mathbf{e_d}$ denote the standard basis vectors of $\mathbb{R}^d$, and let $\mathbf{e_1}^\ast, \dots,\mathbf{e_d}^\ast$ denote the dual basis of $(\mathbb{R}^d)^\ast$. Then let $V^*_{\degen,2}(m,d,h)$ denote the set of all pairs of linear maps $\Xi:\mathbb{R}^h \longrightarrow \mathbb{R}^d$ and $L:\mathbb{R}^h \longrightarrow \mathbb{R}^m$ for which there exist two indices $i,j \leqslant d$, and some real number $\lambda$, such that $(\mathbf{e_i}^* - \lambda \mathbf{e_j}^*)$ is non-zero and $\Xi^*(\mathbf{e_i}^* - \lambda \mathbf{e_j}^*) \in L^*((\mathbb{R}^m)^*)$. We call $V^*_{\degen,2}(m,d,h)$ the \emph{dual pair degeneracy variety}. 
\end{Definition}

One can motivate this definition as follows. We noted in Proposition \ref{Proposition easy degeneracy relation} that if $L:\mathbb{R}^d \longrightarrow \mathbb{R}^m$ is a surjective linear map then saying that $ L \notin V_{\degen}^*(m,d)$ is equivalent to saying that any parametrisation $\Psi = (\psi_1,\dots,\psi_d): \mathbb{R}^{d-m} \longrightarrow \mathbb{R}^d$ of $\ker L$ has finite Cauchy-Schwarz complexity. In this paper, following our sketched idea in expression (\ref{the expression that doesn't change}), we will end up needing to replace the map $L$ with two maps, an injective map $\Xi: \mathbb{R}^{d-u} \longrightarrow \mathbb{R}^d$ and a purely irrational surjective map $L^\prime: \mathbb{R}^{d-u} \longrightarrow \mathbb{R}^{m-u}$ (here $u$ will be the rational dimension of $L$). It will turn out that after this manipulation the system of forms that we will require to have finite Cauchy-Schwarz complexity (in order to bring in Gowers norms) will be $\Xi \Psi^\prime: \mathbb{R}^{d-m} \longrightarrow \mathbb{R}^d$, where $\Psi^\prime: \mathbb{R}^{d-m} \longrightarrow \mathbb{R}^{d-u}$ is a parametrisation of $\ker L^\prime$. One can easily show (and we do, in Lemma \ref{Lemma non quantitative equivalence of degeneracies}), that $(\Xi,L^\prime) \notin V_{\degen}^*(m-u,d,d-u)$ is the exactly the right condition to ensure that $\Xi \Psi^\prime: \mathbb{R}^{d-m} \longrightarrow \mathbb{R}^d$ has finite Cauchy-Schwarz complexity. \\

As ever, we need a quantitative version of non-degeneracy.

\begin{Definition}[Distance metric for pairs of matrices]
\label{Definition distance metric for pairs of matrices}
Let $m,d,h$ be natural numbers, with $d\geqslant h\geqslant m+2$, and let $V^*_{\degen,2}(m,d,h)$ be the dual pair degeneracy variety. Let $\Xi:\mathbb{R}^h \longrightarrow \mathbb{R}^d$ and $L:\mathbb{R}^h \longrightarrow \mathbb{R}^m$ be linear maps. We say that $\operatorname{dist}((\Xi,L),V^*_{\degen,2}(m,d,h))\geqslant c$ if $(\Xi+ Q,L)\notin V^*_{\degen,2}(m,d,h)$ for all $Q:\mathbb{R}^h \longrightarrow \mathbb{R}^{d}$ with $\Vert Q\Vert_\infty < c$. 
\end{Definition}
\noindent Although this is no great subtlety, we should emphasise that in the above definition we only consider perturbations to $\Xi$, and not perturbations to $L$ as well. \\


We are now ready to state our theorem on linear inequalities with Lipschitz cut-offs. 


\begin{Theorem}[Lipschitz case]
\label{Theorem rational set out version}
Let $N,m,d, h$ be natural numbers, with $d\geqslant h\geqslant m+2$, and let $c,C,\varepsilon$ be positive constants. Let $\Xi = \Xi(N):\mathbb{R}^h\longrightarrow \mathbb{R}^d$ be an injective linear map with integer coefficients, and assume that $\Xi(\mathbb{Z}^{h}) = \mathbb{Z}^d \cap \im \Xi$. Let $L = L(N):\mathbb{R}^h \longrightarrow \mathbb{R}^m$ be a surjective linear map. Assume that $\Vert \Xi \Vert_\infty \leqslant C$, $\Vert L \Vert_\infty \leqslant C$, $\dist(L,V_{\rank}(m,d)) \geqslant c$ and $\dist((\Xi,L),V_{\degen,2}^*(m,d,h))\geqslant c$. Then there exists a natural number $s$ at most $d-2$, independent of $\varepsilon$, such that the following holds. Let $\sigma_F,\sigma_G$ be any two parameters in the range $0<\sigma_F,\sigma_G < 1/2$. Let $F:\mathbb{R}^h\longrightarrow [0,1]$ be a Lipschitz function supported on $[-N,N]^h$ with Lipschitz constant $O(1/\sigma_FN)$, and let $G:\mathbb{R}^{m}\longrightarrow [0,1]$ be a Lipschitz function supported on $[-\varepsilon,\varepsilon]^m$ with Lipschitz constant $O(1/\sigma_G)$. Let $\widetilde{\mathbf{r}}$ be a fixed vector in $\mathbb{Z}^d$, satisfying $\Vert \widetilde{\mathbf{r}}\Vert_\infty = O_{c,C,\varepsilon}(1)$. Suppose that $f_1,\dots,f_d:[N]\longrightarrow [-1,1]$ are arbitrary bounded functions satisfying  \[\min_j \Vert f_j\Vert_{U^{s+1}[N]} \leqslant \rho,\] for some $\rho$ in the range $0<\rho \leqslant 1$. Then
\begin{equation}
\label{expression rational separated}
 T^{L,\Xi,\widetilde{\mathbf{r}}}_{F,G,N}(f_1,\dots,f_{d})\ll_{c,C,\varepsilon}  \rho^{\Omega(1)} (\sigma_F^{-O(1)} + \sigma_G^{-O(1)}) + \sigma_F^{-O(1)}N^{-\Omega(1)}. 
\end{equation}
\end{Theorem}
Although the above theorem contains more technical conditions than even Theorem \ref{Main Theorem chapter 3} did, it does represent a significant reduction in complexity from the original problem. Note in particular that the approximation function $A_L$ does not feature in the estimate (\ref{expression rational separated}). 

As we described in Section \ref{Section proof strategy}, the presence of Lipschitz cut-offs rather than convex cut-offs will be especially convenient when approximating the discrete solution count by a continuous solution count. This will be done in Section \ref{section transfer}. \\

The remainder of this section will be devoted to proving the main theorem (Theorem \ref{Main Theorem chapter 3}), assuming the truth of Theorem \ref{Theorem rational set out version}. \\

We begin with two lemmas: one concerning lattices, and the other concerning a quantitative decomposition of the dual space $(\mathbb{R}^d)^*$. Their proofs are entirely standard, but we state them prominently, as we will need to refer to them often in the dimension reduction argument of Lemma \ref{Lemma generating a purely irrational map}.  
\begin{Lemma}[Parametrising the image lattice]
\label{Lemma parametrising the image lattice}
Let $u,d$ be integers with $d\geqslant u+1$. Let $S:\mathbb{R}^d \longrightarrow \mathbb{R}^u$ be a surjective linear map with $S(\mathbb{Z}^d) \subseteq \mathbb{Z}^u$, and suppose that $\Vert S\Vert_\infty \leqslant C$. Then there exists a set $\{\mathbf{a_1},\dots,\mathbf{a_u}\}\subset \mathbb{Z}^u$ that is a basis for the lattice $S(\mathbb{Z}^d)$ and for which $\Vert \mathbf{a_i}\Vert_\infty = O_C(1)$ for every $i$. Furthermore there exist $\mathbf{x_1},\dots,\mathbf{x_u}\in\mathbb{Z}^d$ such that, for every $i$, $S(\mathbf{x_i}) = \mathbf{a_i}$ and $\Vert \mathbf{x_i}\Vert_\infty = O_{C}(1)$. 
\end{Lemma}
\begin{proof}
The lattice $S(\mathbb{Z}^d)$ is $u$ dimensional, as $S$ is surjective. If $\{ \mathbf{e_j}:j\leqslant d\}$ denotes the standard basis of $\mathbb{R}^d$ then integer combinations of elements from the set $\{S(\mathbf{e_j}):j\leqslant d\}$ span $S(\mathbb{Z}^d)$. Since $\Vert S\Vert_\infty \leqslant C$, these vectors also satisfy $\Vert  S(\mathbf{e_j})\Vert_\infty = O_{C}(1)$. Therefore the $u$ successive minima of the lattice $S(\mathbb{Z}^d)$ are all $O_{C}(1)$, and so, by Mahler's theorem (\cite[Theorem 3.34]{TaVu06}) the lattice $S(\mathbb{Z}^d)$ has a basis $\{\mathbf{a_1},\dots,\mathbf{a_u}\}$ of the required form. 

Note that $S$ has integer coefficients. The construction of suitable $\mathbf{x_1},\dots,\mathbf{x_u}$ may be achieved by applying any of the standard algorithms. For example, using Gaussian elimination one may find a basis for $\ker S$ that, by inspection of the algorithm, consists of vectors with rational coordinates of naive height $O_C(1)$. By clearing denominators, one gets vectors $\mathbf{v_1},\dots,\mathbf{v_{d-u}} \in \mathbb{Z}^d$ whose integer span is a full-dimensional sublattice of the $d-u$ dimensional lattice $\mathbb{Z}^d \cap \ker S$, and that satisfy $\Vert \mathbf{v_i}\Vert_\infty = O_{C}(1)$ for all $i$. Now given some $\mathbf{a_i}$, by its construction there must be some $\mathbf{x_i} \in \mathbb{Z}^d$ that satisfies $S(\mathbf{x_i}) = \mathbf{a_i}$. Write $\mathbf{x_i} = \mathbf{x_i}|_{\ker S} + \mathbf{x_i} |_{(\ker S)^\perp}$ as the sum of the obvious projections. By adding a suitable integer combination of the vectors $\mathbf{v_1},\dots,\mathbf{v_{d-u}}$ to $\mathbf{x_i}$ one may find such an $\mathbf{x_i}$ that satisfies $\Vert\mathbf{x_i}|_{\ker S}\Vert_\infty = O_{C}(1)$. Furthermore, $\dist(S,V_{\rank}(m,d)) = \Omega_C(1)$, since $S$ has integer coordinates, and so (by Lemma \ref{nonsingularity of map on orthogonal complement of Kernel}) $\Vert \mathbf{x_i} |_{(\ker S)^\perp}\Vert_\infty = O_{C}(1)$. Hence $\Vert \mathbf{x_i}\Vert_\infty = O_{C}(1)$, as desired. 
\end{proof}

Having established that such a lattice basis $\{\mathbf{a_1},\dots,\mathbf{a_u}\}$ exists, we can now use it to quantitatively decompose $(\mathbb{R}^d)^*$. 
\begin{Lemma}[Dual space decomposition]
\label{Lemma dual space decomposition}
Let $u,d,$ be integers with $d\geqslant u+1$, and let $C,\eta$ be constants. Let $S:\mathbb{R}^d \longrightarrow \mathbb{R}^u$ be a surjective linear map with $S(\mathbb{Z}^d) \subseteq \mathbb{Z}^u$, and suppose that $\Vert S\Vert_\infty \leqslant C$. Let $\{\mathbf{a_1},\dots,\mathbf{a_u}\}$ be a basis for the lattice $S(\mathbb{Z}^d)$ that satisfies $\Vert \mathbf{a_i}\Vert_\infty = O_{C}(1)$ for every $i$. Let $\mathbf{x_1},\dots,\mathbf{x_u}\in\mathbb{Z}^d$ be vectors such that, for every $i$, $S(\mathbf{x_i}) = \mathbf{a_i}$ and $\Vert \mathbf{x_i}\Vert_\infty = O_{C}(1)$. Suppose that $\Xi:\mathbb{R}^{d-u}\longrightarrow \mathbb{R}^d$ is an injective linear map such that $\im \Xi = \ker S$ and such that $\Xi(\mathbb{Z}^{d-u}) = \mathbb{Z}^d \cap\im \Xi$. Suppose further that $\Vert \Xi\Vert_\infty \leqslant C$. 

Let $\mathbf{w_1},\dots,\mathbf{w_{d-u}}$ denote the standard basis vectors in $\mathbb{R}^{d-u}$. Then 
\begin{enumerate}[(1)]
\item the set $\mathcal{B}: = \{ \mathbf{x_i}: i\leqslant u\} \cup \{\Xi(\mathbf{w_j}): j\leqslant d-u\}$ is a basis for $\mathbb{R}^d$, and a lattice basis for $\mathbb{Z}^d$;
\item writing $\mathcal{B}^* := \{ \mathbf{x_i^*}: i\leqslant u\} \cup \{\Xi(\mathbf{w_j})^*: j\leqslant d-u\}$ for the dual basis, both the change of basis matrix between the standard dual basis and $\mathcal{B}^*$ and the inverse of this matrix have integer coordinates. The coefficients of both of these matrices are bounded in absolute value by $O_{C}(1)$.
\end{enumerate}

\noindent  Write $V: = \spn(\mathbf{x_i^*}: i\leqslant u)$ and $W: = \spn(\Xi(\mathbf{w_j})^*: j\leqslant d-u)$. Then

\begin{enumerate}[(1)]
\setcounter{enumi}{2}
\item $V = S^*((\mathbb{R}^u)^*)$;
\item Suppose that $\varphi \in (\mathbb{R}^d)^*$ satisfies $\Vert \Xi^*(\varphi)\Vert_\infty \leqslant \eta$. Then, writing $\varphi = \varphi_V + \varphi_W$ with $\varphi_V \in V$ and $\varphi_W \in W$, we have $\Vert \varphi_W\Vert_\infty = O_C(\eta)$. 
\end{enumerate}

\end{Lemma}
\begin{proof}
For part (1), the fact that $\mathcal{B}$ is a basis for $\mathbb{R}^d$ is just a manifestation of the familiar principle $\mathbb{R}^d \cong \ker S \oplus \im S$. To show that $\mathcal{B}$ is a lattice basis for $\mathbb{Z}^d$, let $\mathbf{n}\in\mathbb{Z}^d$ and write  \[\mathbf{n} = \sum\limits_{i=1}^u \lambda_i \mathbf{x_i} + \sum\limits_{j=1}^{d-u} \mu_j \Xi(\mathbf{w_j})\] for some $\lambda_i,\mu_j\in\mathbb{R}$. Applying $S$, we see $S(\mathbf{n}) = \sum_i^u\lambda_i \mathbf{a_i}$, and hence $\lambda_i\in\mathbb{Z}$ for all $i$, as $\{\mathbf{a_1},\dots,\mathbf{a_u}\}$ is a basis for the lattice $S(\mathbb{Z}^d)$ . But this implies $\sum_{j=1}^{d-u} \mu_j \Xi(\mathbf{w_j}) \in \mathbb{Z}^d \cap \im (\Xi)$. Therefore, as $\Xi(\mathbb{Z}^{d-u}) = \mathbb{Z}^d\cap \ker S$, $\mu_j \in \mathbb{Z}$ for all $j$. 

Part (2) follows immediately from part (1). Part (3) is immediate from the definitions.

For part (4), let $j$ be at most $d-u$. Then the assumption $\Vert \Xi^*(\varphi)\Vert_\infty \leqslant \eta$ means that $\vert \Xi^*(\varphi)(\mathbf{w_j})\vert \leqslant \eta$. Hence $\vert \varphi(\Xi(\mathbf{w_j}))\vert \leqslant \eta$. But, writing $\varphi_W = \sum_{j=1}^{d-u} \mu_j \Xi(\mathbf{w_j})^*$, this implies that $\vert \mu_j\vert \leqslant \eta$. Since the coefficients of the change of basis matrix between $\mathcal{B}^*$ and the standard dual basis are bounded in absolute value by $O_{C}(1)$, this implies that $\Vert \varphi_W\Vert_\infty \leqslant O_C(\eta)$.
\end{proof}

We now begin the attack on Theorem \ref{Main Theorem chapter 3} in earnest. Assume the hypotheses of Theorem \ref{Main Theorem chapter 3}. As a reminder, we have natural numbers $m,d$ satisfying $d\geqslant m+2$, and positive reals $\varepsilon,c,C$. For a natural number $N$, we have $L = L(N):\mathbb{R}^d \longrightarrow \mathbb{R}^m$ being a surjective linear map with approximation function $A_L$, with $\dist(L,V_{\rank}^{\unif}(m,d)) \geqslant c$, $\dist(L,V_{\degen}^*(m,d)) \geqslant c$, and with rational complexity at most $C$. We have $F:\mathbb{R}^d \longrightarrow [0,1]$ being the indicator function of $[1,N]^d$ and $G:\mathbb{R}^m\longrightarrow [0,1]$ being the indicator function of a convex domain contained in $[-\varepsilon,\varepsilon]^m$. For some $s \leqslant d-2$, to be determined, we also have functions $f_1,\dots,f_d:[N]\longrightarrow [-1,1]$ that satisfy $\min_j \Vert f_j\Vert_{U^{s+1}[N]} \leqslant \rho$ for some $\rho$ in the range $0<\rho \leqslant 1$.

The proof has four parts: 
\begin{itemize}
\item Lemma \ref{Lemma replacing F cut-off}, in which we replace the indicator function of $[1,N]^d$ with a Lipschitz cut-off;
\item Lemma \ref{Lemma generating a purely irrational map}, in which we replace $L$ by a pair of maps $(\Xi,L^\prime)$ where $L^\prime$ is purely irrational;
\item Lemma \ref{Lemma making G lipschitz}, in which we replace the function $G$ by a Lipschitz cut-off (using Lemma \ref{Lemma upper bound involving integral});
\item and finally the application of Theorem \ref{Theorem rational set out version} to the pair $(\Xi,L^\prime)$.
\end{itemize} 
\noindent The second of these steps is by far the most technically intricate, and, as we mentioned when discussing our proof strategy in Section \ref{Section proof strategy}, Lemma \ref{Lemma generating a purely irrational map} will have 9 sub-parts. One might well ask why it is necessary to expend so much effort creating a purely irrational map $L^\prime$, given that Theorem \ref{Theorem rational set out version} does not include this condition in its hypotheses. The point is that in order to replace $G$ with a Lipschitz cut-off (and thus in order to be able to apply Theorem \ref{Theorem rational set out version} at all) it is vital that $L^\prime$ is purely irrational. If $L^\prime:\mathbb{R}^{d-u} \longrightarrow \mathbb{R}^{m-u}$ failed to be purely irrational then $L^\prime \mathbb{Z}^{d-u}$ would not equidistribute in $\mathbb{R}^{m-u}$; it would instead be restricted to certain proper affine subspaces. This would affect our ability to perturb the function $G$ without drastically altering the number of solutions to the inequality. For more on this issue, the reader may consult Section \ref{Section proof strategy}.

One does note from the above discussion, however, that in order to deduce Theorem \ref{Main Theorem chapter 3} it would be enough to prove Theorem \ref{Theorem rational set out version} under the additional assumption that $L$ is purely irrational. Yet it turns out that the general version of Theorem \ref{Theorem rational set out version} that we have stated is no harder to prove than the restricted version. \\

We begin with the first of our four parts. 

\begin{Lemma}[Replacing variable cut-off]
\label{Lemma replacing F cut-off}
Assume the hypotheses of Theorem \ref{Main Theorem chapter 3} (in particular let $F$ be the indicator function $1_{[1,N]^d}$), and let $\sigma_F$ be any parameter in the range $0<\sigma_F<1/2$. Then there exists a Lipschitz function $F_{1,\sigma_F}:\mathbb{R}^d \longrightarrow [0,1]$, supported on $[-2N,2N]^d$ and with Lipschitz constant $O(1/\sigma_FN)$, such that \[\vert T_{F,G,N}^L(f_1,\dots,f_d)\vert \ll \vert T_{F_{1,\sigma_F},G,N}^L(f_1,\dots,f_d)\vert + O_{c,C}(\sigma_F).\]
\end{Lemma}

\begin{proof}
By Lemma \ref{Lipschitz approximation of convex cutoffs}, for any parameter $\sigma_F$ in the range $0< \sigma_F < 1/2$ we may write \[ 1_{[1,N]^d} = F_{1,\sigma_F} + O(F_{2,\sigma_F}),\] where $F_{1,\sigma_F},F_{2,\sigma_F}$ are Lipschitz functions supported on $[-2N,2N]^d$, with Lipschitz constants $O(1/\sigma_FN)$, and with $\int_\mathbf{x} F_{2,\sigma_F}(\mathbf{x}) \, d\mathbf{x} = O(\sigma_F N^d)$. Moreover, $F_{2,\sigma_F}$ is supported on \[ \{\mathbf{x} \in \mathbb{R}^d: \dist(\mathbf{x}, \partial ([1,N]^d)) = O(\sigma_F N)\}.\] 

Therefore \[T_{F,G,N}^L(f_1,\dots,f_d) \ll \vert T_{F_{1,\sigma_F},G,N}^L(f_1,\dots,f_d)\vert + \vert T_{F_{2,\sigma_F},G,N}^L(1,\dots,1)\vert.\] Therefore, since $\dist(L,V_{\rank}^{\unif}(m,d)) \geqslant c$, by Lemma \ref{Lemma slightly less crude bound on number of solutions} we have \[\vert T_{F_{2,\sigma_F},G,N}^L(f_1,\dots,f_d)\vert = O_{c,C}(\sigma_F).\] This gives the lemma.
\end{proof}

Next comes the critical lemma, in which we successfully replace the map $L$ by a purely irrational map $L^\prime$. For the definition of the approximation function $A_L$, one may consult Definition \ref{Definition approximation function}.
\begin{Lemma}[Generating a purely irrational map]
\label{Lemma generating a purely irrational map}
Let $\sigma_F$ be a parameter in the range $0<\sigma_F < 1/2$. Assume the hypotheses of Theorem \ref{Main Theorem chapter 3}, with the exception that $F:\mathbb{R}^d \longrightarrow [0,1]$ now denotes a Lipschitz function supported on $[-2N,2N]^d$ and with Lipschitz constant $O(1/\sigma_FN)$. Let $u$ be the rational dimension of $L$, and assume that $u\leqslant m-1$. Then there exists a surjective linear map $L^\prime:\mathbb{R}^{d-u} \longrightarrow \mathbb{R}^{m-u}$, an injective linear map $\Xi:\mathbb{R}^{d-u} \longrightarrow \mathbb{R}^d$, a finite subset $\widetilde{R}\subset \mathbb{Z}^d$, and, for each $\widetilde{\mathbf{r}} \in \widetilde{R}$, functions $F_{\widetilde{\mathbf{r}}}:\mathbb{R}^{d-u} \longrightarrow \mathbb [0,1]$ and $G_{\widetilde{\mathbf{r}}}:\mathbb{R}^{m-u} \longrightarrow [0,1]$, that together satisfy the following properties:
\begin{enumerate}[(1)]
\item $\Xi$ has integer coefficients, $\Vert \Xi \Vert_\infty  = O_{c,C}(1)$, and $\Xi(\mathbb{Z}^{d-u}) = \mathbb{Z}^d \cap \im \Xi$;
\item $\vert \widetilde{R}\vert = O_{c,C}(1)$, and $\Vert \widetilde{\mathbf{r}}\Vert_\infty = O_{c,C}(1)$ for all $\widetilde{\mathbf{r}} \in \widetilde{R}$;
\item $F_{\widetilde{\mathbf{r}}}$ is supported on $[-O_{c,C}(N),O_{c,C}(N)]^{d-u}$, with Lipschitz constant $O_{c,C}(1/\sigma_FN)$, and $G_{\widetilde{\mathbf{r}}}$ is the indicator function of a convex domain contained in \\$[-O_{c,C,\varepsilon}(1),O_{c,C,\varepsilon}(1)]^{m-u}$;
\item $T_{F,G,N}^L(f_1,\dots,f_d) = \sum\limits_{\widetilde{\mathbf{r}} \in \widetilde{R}}T_{F_{\widetilde{\mathbf{r}}},G_{\widetilde{\mathbf{r}}},N}^{L^\prime,\Xi,\widetilde{\mathbf{r}}}(f_1,\dots,f_d);$
\item $L^\prime$ is purely irrational;
\item $\Vert L^\prime\Vert_\infty = O_{c,C}(1)$ and $\dist(L^\prime, V_{\rank}(m-u,d-u)) = \Omega_{c,C}(1)$;
\item $\dist((\Xi,L^\prime),V_{\degen,2}^*(m-u,d,d-u)) = \Omega_{c,C}(1)$;
\item for all $\tau_1,\tau_2 \in (0,1]$, $A_{L^\prime}(\tau_1,\tau_2) \gg_{c,C} A_L(\Omega_{c,C}(\tau_1),\Omega_{c,C}(\tau_2))$;
\item for all $\tau_1,\tau_2 \in (0,1]$, $A_{L^\prime}(\tau_1,\tau_2) \ll_{c,C} A_L(\Omega_{c,C}(\tau_1),\Omega_{c,C}(\tau_2))$.
\end{enumerate}
\end{Lemma}
\noindent The fundamental aspect of this lemma is part (4), of course, as this directly concerns how we control the number of solutions to the diophantine inequality itself when passing from $L$ to $L^\prime$. However, we do need to establish parts (1) - (8), in order to be able to ensure that the hypotheses of Lemma \ref{Lemma upper bound involving integral} and Theorem \ref{Theorem rational set out version} are satisfied. Part (9) is included for completeness, and to assist the calculations in Appendix \ref{section algebraic approximation}. \\

Before giving the full details of the proof, we sketch the idea. Let $\Theta:\mathbb{R}^m \longrightarrow \mathbb{R}^u$ be a rational map for $L$. The space $\ker (\Theta L)$ has dimension $d-u$, and so we may parametrise it by some injective map $\Xi:\mathbb{R}^{d-u} \longrightarrow \ker (\Theta L)$. Without too much difficultly, $\Xi$ can be chosen to satisfy $\Xi(\mathbb{Z}^{d-u}) = \mathbb{Z}^d \cap \im \Xi$. Then \[ L\Xi: \mathbb{R}^{d-u} \longrightarrow \ker \Theta,\] is a map from a $d-u$ dimensional space to an $m-u$ dimensional space, and it turns out that $L\Xi$ is purely irrational, and $L^\prime = L\Xi$ may be used in Lemma \ref{Lemma generating a purely irrational map}.

Of course this isn't quite possible, as we only defined the notion of purely irrational maps between vector spaces of the form $\mathbb{R}^a$. But it is true after choosing a judicious isomorphism from $\ker \Theta$ to $\mathbb{R}^{m-u}$ (though this does complicate the notation). \\

Let us complete the details. 

\begin{proof}
First we note that the lemma is obvious when $u=0$, since one may take $\Xi:\mathbb{R}^d \longrightarrow \mathbb{R}^d$ to be the identity map, $\widetilde{\mathbf{r}}$ to be $\mathbf{0}$, and $L^\prime$ to be $L$. So assume that $u > 1$.\\

We proceed with a general reduction, familiar from our proof of Proposition \ref{Proposition maximal rational dimension case of main theorem}, in which we may assume that the first $m$ columns of $L$ form the identity matrix.

Indeed, let $\Theta:\mathbb{R}^m\longrightarrow \mathbb{R}^u$ be a rational map for $L$ with $\Vert \Theta \Vert_\infty \leqslant C$. Now let $\widetilde{L} : = M^{-1} L$, where $M$ is a rank matrix of $L$ (Proposition \ref{rank matrix}), which, without loss of generality, consists of the first $m$ columns of $L$. Let $\widetilde{\Theta} : = \Theta M$ and let $\widetilde{G}: = G\circ M$. Then \[ T_{F,G,N}^L(f_1,\dots,f_d) = T_{F,\widetilde{G},N} ^{\widetilde{L}}(f_1,\dots,f_d),\] and, considering $\widetilde{\Theta}$, $\widetilde{L}$ has rational complexity $O_{c,C}(1)$. Furthermore, $\widetilde{G}$ is the indicator function of a convex domain contained in $[-O_{c,C}(\varepsilon), O_{c,C}(\varepsilon)]^m$. We also have $\dist(\widetilde{L},V_{\degen}^*(m,d)) = \Omega_{c,C}(1)$. Finally, for all $\tau_1,\tau_2 \in (0,1]$, we have that\\ $A_{\widetilde{L}}(\tau_1,\tau_2) \asymp_{c,C} A_L(\Omega_{c,C}(\tau_1),\Omega_{c,C}(\tau_2))$. 

Therefore, by replacing $L$ with $\widetilde{L}$ and $G$ with $\widetilde{G}$, we may assume throughout the proof of Lemma \ref{Lemma generating a purely irrational map} that the first $m$ columns of $L$ form the identity matrix. This is at the cost of replacing $\varepsilon$ by $O_{c,C}(\varepsilon)$, $C$ by $O_{c,C}(1)$, and $c$ by $\Omega_{c,C}(1)$.  \\

Now let $\Theta:\mathbb{R}^m\longrightarrow \mathbb{R}^u$ be a rational map for $L$ with $\Vert \Theta \Vert_\infty = O_{c,C}(1)$. Since the first $m$ columns of $L$ form the identity matrix, $\Theta$ must have integer coefficients. \\

\textbf{Part (1)}: By rank-nullity $\ker (\Theta L)$ is a $d-u$ dimensional subspace of $\mathbb{R}^d$. The matrix of $\Theta L$ has integer coefficients and $\Vert \Theta L \Vert_\infty = O_{c,C}(1)$. Combining these two facts, we see that $\ker (\Theta L) \cap \mathbb{Z}^d$ is a $d-u$ dimensional lattice, and by the standard algorithms one can find a lattice basis $\mathbf{v^{(1)}},\dots, \mathbf{v^{(d-u)}} \in \mathbb{Z}^d$ that satisfies $\Vert \mathbf{v^{(i)}}\Vert_\infty = O_{c,C}(1)$ for every $i$. Define $\Xi:\mathbb{R}^{d-u} \longrightarrow \mathbb{R}^{d}$ by \[ \Xi(\mathbf{w}):= \sum\limits_{i=1}^{d-u} w_i \mathbf{v^{(i)}}.\] Then $\Xi$ satisfies property (1) of the lemma. Note that image of the map $L\Xi:\mathbb{R}^{d-u} \longrightarrow \mathbb{R}^m$ is exactly $\ker \Theta$. \\

\textbf{Part (2)}: Since $\Vert \Theta \Vert_\infty = O_{c,C}(1)$, if $\mathbf{y} \in\mathbb{R}^m$ and $\Theta(\mathbf{y}) = \mathbf{r}$ then  $ \Vert \mathbf{y} \Vert_\infty \gg_{c,C} \Vert\mathbf{r} \Vert_\infty$. Recall that the support of $G$ is contained within $[-O_{c,C,\varepsilon}(1),O_{c,C,\varepsilon}(1)]^m$, and that $\Theta L(\mathbb{Z}^d) \subseteq \mathbb{Z}^u$. It follows that there are at most $O_{c,C,\varepsilon}(1)$ possible vectors $\mathbf{r} \in \mathbb{Z}^u$ for which there exists a vector $\mathbf{n} \in\mathbb{Z}^d$ for which both $G(L\mathbf{n}) \neq 0$ and $\Theta L \mathbf{n} = \mathbf{r}$. Let $R$ denote the set of all such vectors $\mathbf{r}$. 

For each $\mathbf{r} \in R$, there exists a vector $\widetilde{\mathbf{r}}\in \mathbb{Z}^d$ such that $\Theta L \widetilde{\mathbf{r}} = \mathbf{r}$ and $\Vert \widetilde{\mathbf{r}} \Vert_\infty = O_{c,C,\varepsilon}(1)$. Let $\widetilde{R}$ denote the set of these $\widetilde{\mathbf{r}}$. Then $\widetilde{R}$ satisfies part (2).\\

Before proceeding to prove part (3) of the lemma, we pause to apply Lemmas \ref{Lemma parametrising the image lattice} and \ref{Lemma dual space decomposition}. Indeed, applying these lemmas to the map $S: = \Theta L$, there exists a set $\{ \mathbf{a_1},\dots,\mathbf{a_u}\} \subset \mathbb{Z}^u$ that is a basis for the lattice $\Theta L(\mathbb{Z}^d)$ and for which $\Vert \mathbf{a_i}\Vert_\infty = O_{c,C}(1)$ for each $i$. Also, there exists a set of vectors $\{ \mathbf{x_1},\dots,\mathbf{x_u}\} \subset \mathbb{Z}^d$ such that $\Theta L(\mathbf{x_i})= \mathbf{a_i}$ for each $i$, and $\Vert \mathbf{x_i}\Vert_\infty = O_{c,C}(1)$. By Lemma \ref{Lemma dual space decomposition}, 
\begin{equation}
\label{basis for Rd}
 \mathcal{B}: = \{\mathbf{x_i}:i\leqslant u\} \cup \{ \Xi(\mathbf{w_j}):j\leqslant d-u\}
 \end{equation} is a basis for $\mathbb{R}^d$ and a lattice basis for $\mathbb{Z}^d$, where $\mathbf{w_1},\dots,\mathbf{w_{d-u}}$ denotes the standard basis of $\mathbb{R}^{d-u}$.\\

\textbf{Part (3)}: By the definition of $\widetilde{R}$, and the fact that $\Xi(\mathbb{Z}^{d-u}) = \mathbb{Z}^d \cap \ker (\Theta L)$, we have  
\begin{equation}
\label{equation getting integer rows}
T_{F,G,N}^L(f_1,\dots,f_d) = \sum\limits_{\widetilde{\mathbf{r}} \in \widetilde{R}}\frac{1}{N^{d-m}} \sum\limits_{\mathbf{n} \in \mathbb{Z}^{d-u}} \Big(\prod\limits_{j=1}^d f_j(\xi_j(\mathbf{n}) + \widetilde{\mathbf{r}}_j) \Big)F(\Xi(\mathbf{n}) + \widetilde{\mathbf{r}}) G(L\Xi(\mathbf{n}) + L\widetilde{\mathbf{r}}),
\end{equation}
\noindent where $\widetilde{\mathbf{r}}_j$ denotes the $j^{\text{th}}$ coordinates of $\widetilde{\mathbf{r}}$. Now by an easy linear algebraic argument (recorded in Lemma \ref{Lemma construction of P}), 
\begin{equation}
\label{direct sum}
\mathbb{R}^m = \spn(L\mathbf{x_i}:i\leqslant u) \oplus \ker \Theta
\end{equation} as an algebraic direct sum, and there exists an invertible linear map $P:\mathbb{R}^m \longrightarrow \mathbb{R}^m$ such that \begin{align}
\label{equation properties of P}
P((\spn(L\mathbf{x_i}:i\leqslant u))) &= \mathbb{R}^u \times \{0\}^{m-u}, \\
\label{equation properties of P 2}
P(\ker \Theta) &= \{0\}^{u} \times \mathbb{R}^{m-u}, 
\end{align}
\noindent and both $\Vert P\Vert_\infty = O_{c,C}(1)$ and $\Vert P^{-1} \Vert_\infty = O_{c,C}(1)$. 

We have \[ G(L\Xi(\mathbf{n}) + L\widetilde{\mathbf{r}}) = (G \circ P^{-1})(PL\Xi(\mathbf{n}) + PL\widetilde{\mathbf{r}}),\] and we note that $PL\Xi(\mathbf{n}) \in \{0\}^u \times \mathbb{R}^{m-u}$ for every $\mathbf{n} \in \mathbb{Z}^{d-u}$. Define $G_{\widetilde{\mathbf{r}}}:\mathbb{R}^{m-u} \longrightarrow [0,1]$ by \[G_{\widetilde{\mathbf{r}}}(\mathbf{x}) : = (G\circ P^{-1})(\mathbf{x_0} + PL\widetilde{\mathbf{r}}),\] where $\mathbf{x_0}$ is the extension of $\mathbf{x}$ by $0$ in the first $u$ coordinates. Then the function $G_{\widetilde{\mathbf{r}}}$ is the indicator function of a convex set contained in $[-O_{c,C,\varepsilon}(1),O_{c,C,\varepsilon}(1)]^{m-u}$. 

Define \[F_{\widetilde{\mathbf{r}}}(\mathbf{n}) : =  F(\Xi(\mathbf{n}) + \widetilde{\mathbf{r}}).\] Then $F_{\widetilde{\mathbf{r}}}$ has Lipschitz constant $O_{c,C}(1/\sigma_F N)$ and $F_{\widetilde{\mathbf{r}}}$ is supported on \\$[-O_{c,C,\varepsilon}(N),O_{c,C,\varepsilon}(N)]^{d-u}$. (For a full proof of this fact, apply Lemma \ref{Lemma bounded inverse} to the map $\Xi$). So $F_{\widetilde{\mathbf{r}}}$  and $G_{\widetilde{\mathbf{r}}}$ satisfy part (3).\\

\textbf{Part (4)}: Writing $\pi_{m-u}:\mathbb{R}^{m}\longrightarrow \mathbb{R}^{m-u}$ for the projection onto the final $m-u$ coordinates, expression (\ref{equation getting integer rows}) is equal to 
\begin{equation}
\label{equation end of second part}
\sum\limits_{\widetilde{\mathbf{r}} \in \widetilde{R}}\frac{1}{N^{d-m}} \sum\limits_{\mathbf{n} \in \mathbb{Z}^{d-u}} \Big(\prod\limits_{j=1}^d f_j(\xi_j(\mathbf{n}) + \widetilde{\mathbf{r}}_j) \Big)F_{\widetilde{\mathbf{r}}}(\mathbf{n})G_{\widetilde{\mathbf{r}}}( \pi_{m-u} PL\Xi(\mathbf{n})).
\end{equation} 
\noindent Let 
\begin{equation}
\label{equation definition of L prime}
L^\prime: = \pi_{m-u} PL\Xi.
\end{equation} Then $L^\prime:\mathbb{R}^{d-u} \longrightarrow \mathbb{R}^{m-u}$ is surjective, and \[T_{F,G,N}^L(f_1,\dots,f_d) = \sum\limits_{\widetilde{\mathbf{r}} \in \widetilde{R}}T_{F_{\widetilde{\mathbf{r}}},G_{\widetilde{\mathbf{r}}},N}^{L^\prime,\Xi,\widetilde{\mathbf{r}}}(f_1,\dots,f_d).\] This resolves part (4).\\

\textbf{Part (5)}:  We wish to show that $L^\prime$ is purely irrational. Suppose for contradiction that there exists some surjective linear map $\varphi:\mathbb{R}^{m-u} \longrightarrow \mathbb{R}$ with $\varphi L^\prime (\mathbb{Z}^{d-u}) \subseteq \mathbb{Z}$, i.e. with $\varphi \pi_{m-u} PL\Xi(\mathbb{Z}^{d-u}) \subseteq \mathbb{Z}$. Then define the map $\Theta^\prime:\mathbb{R}^m \longrightarrow \mathbb{R}^{u+1}$ by \[ \Theta^\prime(\mathbf{x}) : = (\Theta(\mathbf{x}),\varphi \pi_{m-u} P(\mathbf{x})).\] Then $\Theta^\prime$ is surjective, and $\Theta^\prime L(\mathbb{Z}^d)\subseteq \mathbb{Z}^{u+1}$. (This second fact is immediately seen by writing $\mathbb{Z}^d$ with respect to the lattice basis $\mathcal{B}$ from (\ref{basis for Rd})). This contradicts the assumption that $L$ has rational dimension $u$. So $L^\prime$ is purely irrational.  \\

\textbf{Part (6)}: The bound $\Vert L^\prime\Vert_\infty = O_{c,C}(1)$ follows immediately from the bounds on the coefficients of $\Xi$, $L$, $P$, and $\pi_{m-u}$ separately. 

We wish to prove that $\dist(L^\prime, V_{\rank}(m-u,d-u)) \gg_{c,C} 1$, i.e. that\\ $\dist(\pi_{m-u} PL\Xi, V_{\rank}(m-u,d-u)) \gg_{c,C} 1$. Suppose for contradiction that, for a small parameter $\eta$, there exists a linear map $Q: \mathbb{R}^{d-u} \longrightarrow \mathbb{R}^{m-u}$ such that $\Vert Q\Vert_{\infty}<\eta$ and $\pi_{m-u} PL\Xi + Q$ has rank less than $m-u$. 
Recall that $PL\Xi (\mathbb{R}^{d-u}) = \{0\}^u \times \mathbb{R}^{m-u}$. So, extending $Q$ by zeros to a map $Q:\mathbb{R}^{d-u} \longrightarrow \{0\}^u \times \mathbb{R}^{m-u}$, and applying $P^{-1}$, there is a map $Q^\prime:\mathbb{R}^{d-u} \longrightarrow \mathbb{R}^m$ such that $\Vert Q^\prime\Vert_{\infty}=O_{c,C}(\eta)$ and $L\Xi + Q^\prime$ has rank less than $m-u$. 

We may factorise $Q^\prime = H\Xi$ for some $m$-by-$d$ matrix $H$. Indeed let \[\mathcal{B} := \{ \mathbf{x_i}: i\leqslant u\} \cup \{\Xi(\mathbf{w_j}): j\leqslant d-u\}\] be the basis of $\mathbb{R}^d$ from (\ref{basis for Rd}), i.e. the basis formed by applying Lemma \ref{Lemma dual space decomposition} to the map $S: = \Theta L$. Define the linear map $H$ by $H(\Xi(\mathbf{w_j})): = Q^\prime(\mathbf{w_{j}})$ for each $j$ and $H(\mathbf{x_i}): = \mathbf{0}$ for each $i$. Since the change of basis matrix between $\mathcal{B}$ and the standard basis of $\mathbb{R}^d$ has integer coefficients with absolute values at most $O_{c,C}(1)$, it follows that the matrix representing $H$ with respect to the standard bases satisfies $\Vert H\Vert_\infty = O_{c,C}(\eta)$.

So we know that $(L+H)\Xi$ has rank less than $m-u$. But $\Xi:\mathbb{R}^{d-u} \longrightarrow \mathbb{R}^d$ is injective, so this implies that the rank of $L+H$ is less than $m$. Hence $\dist(L,V_{\rank}(m,d)) = O_{c,C}(\eta)$, which contradicts the assumptions of the lemma (if $\eta$ is small enough). So $\dist(L^\prime,V_{\rank}(m-u,d-u)) \gg_{c,C} 1$ as required. \\

\textbf{Part (7)}: We wish to show that $\dist((\Xi,L^\prime),V_{\degen,2}^*(m-u,d,d-u)) = \Omega_{c,C}(1)$. Suppose for contradiction that, for a small parameter $\eta$, there exists a linear map $Q:\mathbb{R}^{d-u} \longrightarrow \mathbb{R}^d$ such that  $\Vert Q\Vert_\infty \leqslant \eta$ and $\dist((\Xi + Q, L^\prime ),V_{\degen,2}^*(m-u,d,d-u)) \leqslant \eta$. In other words, we suppose there exist two indices $i,j\leqslant d$, and a real number $\lambda$, such that $\mathbf{e_i}^* - \lambda\mathbf{e_j}^*$ is non-zero and \[ (\Xi + Q)^*(\mathbf{e_i}^* - \lambda\mathbf{e_j}^*) \in (L^\prime)^*((\mathbb{R}^{m-u})^*),\] where $\{\mathbf{e_1},\dots,\mathbf{e_d}\}$ denotes the standard basis of $\mathbb{R}^d$ and $\{\mathbf{e_1}^*,\dots,\mathbf{e_d}^*\}$ denotes the dual basis. Expanding out the definition of $L^\prime$, this means that there exists some $\varphi \in (\mathbb{R}^{m-u})^*$ such that \[\Xi^*(\mathbf{e_i}^* - \lambda\mathbf{e_j}^* - L^*(P^* \pi_{m-u}^*(\varphi))) = -Q^*(\mathbf{e_i}^* - \lambda\mathbf{e_j}^*).\] Because $\Vert Q\Vert_\infty \leqslant \eta$, this means that 
\begin{equation}
\label{about to be dualed}
\Vert \Xi^*(\mathbf{e_i}^* - \lambda\mathbf{e_j}^* - L^*(P^* \pi_{m-u}^*(\varphi)))\Vert_\infty = O(\eta).
\end{equation}

Let 
\begin{equation}
\label{funky dual basis}
\mathcal{B}^* := \{ \mathbf{x_i^*}: i\leqslant u\} \cup \{\Xi(\mathbf{w_j})^*: j\leqslant d-u\}
\end{equation} denote the basis of $(\mathbb{R}^d)^*$ that is dual to the basis $\mathcal{B}$ from (\ref{basis for Rd}). It follows from part (4) of Lemma \ref{Lemma dual space decomposition} and (\ref{about to be dualed}) that \[\mathbf{e_i}^* - \lambda\mathbf{e_j}^* - L^*(P^* \pi_{m-u}^*(\varphi)) = \omega_V + \omega_W,\] where $\omega_V \in L^* \Theta^*((\mathbb{R}^u)^*)$, $\omega_W \in \spn(\Xi(\mathbf{w_j})^*:j\leqslant d-u)$, and $\Vert\omega_W\Vert_\infty = O_{c,C}(\eta)$. So therefore \[ \mathbf{e_i}^* - \lambda\mathbf{e_j}^* = L^*(\alpha) + \omega_W,\] for some $\alpha \in (\mathbb{R}^m)^*$. 

This is enough to derive a contradiction. Indeed, without loss of generality one may assume that $\Vert \mathbf{e_i}^* - \lambda\mathbf{e_j}^*\Vert_\infty \geqslant 1$ (this is obvious if $i\neq j$, and if $i = j$ we may just pick $\lambda = 0$ at the outset). Therefore $\Vert \mathbf{e_i}^* - \lambda\mathbf{e_j}^* - \omega_W\Vert \geqslant 1/2$, provided $\eta$ is small enough. Since $\Vert L^*\Vert_\infty = O_{c,C}(1)$, we conclude that $\Vert \alpha\Vert_\infty = \Omega_{c,C}(1)$. 

This means that there exists a linear map $E:\mathbb{R}^d \longrightarrow \mathbb{R}^m$ with $\Vert E \Vert_\infty = O_{c,C}(\eta)$ for which $E^*(\alpha) = \omega_W$. Then \[ \mathbf{e_i}^* - \lambda\mathbf{e_j}^* \in (L+E)^*((\mathbb{R}^{m})^*),\] and hence $\dist(L, V_{\degen}^*(m,d)) = O_{c,C}(\eta)$. This is a contradiction to the hypotheses of Theorem \ref{Main Theorem chapter 3}, provided $\eta$ is small enough, and hence $\dist((\Xi,L^\prime),V_{\degen,2}^*(m-u,d,d-u)) = \Omega_{c,C}(1)$. \\

\textbf{Part (8)}: Let $\tau_1,\tau_2 \in (0,1]$. We desire to prove the relationship 
\begin{equation}
\label{claim of part 8}
A_{L^\prime}(\tau_1,\tau_2) \gg_{c,C} A_L(\Omega_{c,C}(\tau_1),\Omega_{c,C}(\tau_2)),
\end{equation}
\noindent where $L^\prime$ is as in (\ref{equation definition of L prime}).

We have already proved that $L^\prime$ is purely irrational (that was part (5) of the lemma). So, if $A_{L^\prime}(\tau_1,\tau_2)  <\eta$, for some $\eta$, there exists some $\varphi \in (\mathbb{R}^{m-u})^*$ for which $ \tau_1\leqslant \Vert \varphi\Vert_\infty \leqslant \tau_2^{-1}$ and for which \[\dist( (\pi_{m-u} PL\Xi)^*(\varphi),(\mathbb{Z}^{d-u})^T) <\eta,\] where, one recalls, we use $(\mathbb{Z}^{d-u})^T$ to denote the set of those functions in $(\mathbb{R}^{d-u})^*$ that have integer coordinates with respect to the standard dual basis. 

We claim that 
\begin{align}
\label{first approximation equation of part 8 of the lemma}
\dist(L^*(P^*\pi_{m-u}^*(\varphi)),(\mathbb{Z}^d)^T) \ll_{c,C}\eta;\\
\label{second approximation equation of part 8 of the lemma}
\Vert P^*\pi_{m-u}^*(\varphi)\Vert_\infty \ll_{c,C} \tau_2^{-1};\\
\label{third approximation equation of part 8 of the lemma}
\dist(P^*\pi_{m-u}^*(\varphi), \Theta^*((\mathbb{R}^{u})^*)) \gg_{c,C} \tau_1, 
\end{align}
\noindent from which (\ref{claim of part 8}) immediately follows. \\

Let us prove (\ref{first approximation equation of part 8 of the lemma}). Indeed, we already know that\\ $\dist( \Xi^* L^* P^* \pi_{m-u}^*(\varphi),(\mathbb{Z}^{d-u})^T) <\eta$, i.e. that 
\begin{equation}
\label{first approximation equation}
 \Vert \Xi^* L^* P^* \pi_{m-u}^*(\varphi) - \alpha\Vert_\infty <\eta,
 \end{equation} for some $\alpha \in (\mathbb{Z}^{d-u})^T$. Let us write $\alpha = \sum_{j=1}^{d-u} \lambda_j\mathbf{w_j}^*$ for some $\lambda_j \in \mathbb{Z}$, where $\mathbf{w_1},\dots,\mathbf{w_{d-u}}$ denotes the standard basis for $\mathbb{R}^{d-u}$ and $\mathbf{w_1}^*,\dots,\mathbf{w_{d-u}}^*$ denotes the dual basis. Let $\mathcal{B}^*$ be as in (\ref{funky dual basis}). Then $\mathbf{w_j}^* = \Xi^*((\Xi(\mathbf{w_j})^*)$, and so \[ \alpha = \Xi^*(\sum\limits_{j=1}^{d-u} \lambda_j \Xi(\mathbf{w_j}) ^*).\] So from (\ref{first approximation equation}) and the final part of Lemma \ref{Lemma dual space decomposition}, 
\begin{equation}
L^*P^* \pi_{m-u}^*(\varphi) - \sum\limits_{j=1}^{d-u} \lambda_j \Xi(\mathbf{w_j}) ^* = \omega_V + \omega_W,
\end{equation}
\noindent where $\omega_V \in \spn(\mathbf{x_i^*}:i\leqslant u)$, $\omega_W \in \spn(\Xi(\mathbf{w_j})^*:j\leqslant d-u)$, and $\Vert\omega_W\Vert_\infty = O_{c,C}(\eta)$. 

But $L^*P^* \pi_{m-u}^*(\varphi) \in \spn(\Xi(\mathbf{w_j})^*:j\leqslant d-u)$ too. Indeed, for every $i$ at most $d-u$,
\[L^*P^* \pi_{m-u}^*(\varphi)(\mathbf{x_i}) = \varphi (\pi_{m-u} PL\mathbf{x_i}) = \varphi(\mathbf{0}) = 0,\] by the properties of $P$ (see (\ref{equation properties of P})). Therefore $\omega_V = \mathbf{0}$, and so \[\Vert L^*P^* \pi_{m-u}^*(\varphi) - \sum\limits_{j=1}^{d-u} \lambda_j \Xi(\mathbf{w_j}) ^*\Vert_\infty = O_{c,C}(\eta).\] Since  $\sum_{j=1}^{d-u} \lambda_j \Xi(\mathbf{w_j}) ^* \in (\mathbb{Z}^d)^T$, this implies (\ref{first approximation equation of part 8 of the lemma}) as claimed.\\

The bound (\ref{second approximation equation of part 8 of the lemma}) is immediate from the bounds on the coefficients of $P^*$ and $\pi_{m-u}^*$, so it remains to prove (\ref{third approximation equation of part 8 of the lemma}). Suppose for contradiction that, for some small parameter $\delta$, \[P^*\pi_{m-u}^*(\varphi) = \alpha_1 + \alpha_2,\] where $\alpha_1 \in \Theta^*((\mathbb{R}^{u})^*)$ and $\Vert \alpha_2 \Vert_\infty \leqslant \delta \tau_1$. We know that $\Vert \varphi \Vert_\infty \geqslant \tau_1$, which means that there is some standard basis vector $\mathbf{f_k} \in \mathbb{R}^{m-u}$ for which $\vert \varphi(\mathbf{f_k})\vert \geqslant \tau_1$. Let $\mathbf{b_{k+u}}$ be the standard basis vector of $\mathbb{R}^m$ for which $\pi_{m-u}(\mathbf{b_{k+u}}) = \mathbf{f_k}$. Recall the properties of $P$ (given in (\ref{equation properties of P}) and (\ref{equation properties of P 2})), in particular recall that $P:\ker \Theta \longrightarrow \{0\}^u \times \mathbb{R}^{m-u}$ is an isomorphism. Then \[\vert P^* \pi_{m-u}^*(\varphi)(P^{-1}(\mathbf{b_{k+u}}))\vert = \vert\pi_{m-u}^*(\varphi)(\mathbf{b_{k+u}})\vert = \vert \varphi(\mathbf{f_k})\vert \geqslant \tau_1.\] Note that $\Theta^*((\mathbb{R}^{u})^*) = (\ker \Theta)^{0}$, and so \[\vert P^* \pi_{m-u}^*(\varphi)(P^{-1}(\mathbf{b_{k+u}}))\vert = \vert(\alpha_1 + \alpha_2)(P^{-1}(\mathbf{b_{k+u}}))\vert = \vert \alpha_2(P^{-1}(\mathbf{b_{k+u}})) \vert \ll_{c,C} \delta \tau_1,\] as $P^{-1}(\mathbf{b_{k+u}}) \in \ker \Theta$ and satisfies $\Vert P^{-1}(\mathbf{b_{k+u}})\Vert_\infty = O_{c,C}(1)$. This is a contradiction if $\delta$ is small enough, and so (\ref{third approximation equation of part 8 of the lemma}) holds. This resolves part (8).\\

\textbf{Part (9)}: Let $\tau_1,\tau_2 \in (0,1]$. We desire to prove the relationship 
\begin{equation}
\label{claim of part 9}
A_{L^\prime}(\tau_1,\tau_2) \ll_{c,C} A_L(\Omega_{c,C}(\tau_1),\Omega_{c,C}(\tau_2)),
\end{equation}
\noindent where $L^\prime$ is as in (\ref{equation definition of L prime}).This inequality is the reverse inequality of part (8), and in fact it will not be required in the proof of any of our main theorems. However, it will be required in order to analyse $A_L(\tau_1,\tau_2)$ when $L$ has algebraic coefficients (in Appendix \ref{section algebraic approximation}), so we choose to state and prove it here, close to our argument for part (8).

Suppose that $A_L(\tau_1,\tau_2)<\eta$, for some parameter $\eta$. Then there exists some $\varphi\in (\mathbb{R}^m)^*$ such that $\dist(\varphi, \Theta^*((\mathbb{R}^u)^*)) \geqslant \tau_1$, $\Vert \varphi \Vert_\infty \leqslant \tau_2^{-1}$, and $\dist(L^*\varphi, (\mathbb{Z}^d)^T) < \eta$. So there exists some $\omega \in (\mathbb{Z}^d)^T$ for which \[ \Vert L^* \varphi - \omega\Vert_\infty < \eta.\] 

We expand both $L^*\varphi$ and $\omega$ with respect to the dual basis $\mathcal{B}^*$ from (\ref{funky dual basis}). So, 
\begin{align*}
L^*\varphi &= \sum\limits_{i=1}^u \lambda_i \mathbf{x_i^*} + \sum\limits_{j=1}^{d-u} \mu_j \Xi(\mathbf{w_j})^*\\
\omega& = \sum\limits_{i=1}^u \lambda_i^\prime \mathbf{x_i^*} + \sum\limits_{j=1}^{d-u} \mu_j^\prime \Xi(\mathbf{w_j})^*.
\end{align*}
\noindent Since $\mathcal{B}^*$ is a lattice basis for $(\mathbb{Z}^d)^T$, we have $\lambda_i^\prime \in \mathbb{Z}$ and $\mu_j^\prime \in\mathbb{Z}$ for each $i$ and $j$. Since the change of basis matrix between $\mathcal{B}^*$ and the standard dual basis has integer coefficients that are bounded in absolute value by $O_{c,C}(1)$ (part (2) of Lemma \ref{Lemma dual space decomposition}), one has $\vert \lambda_i - \lambda_i^\prime\vert = O_{c,C}(\eta)$ and $\vert \mu_j - \mu_j^\prime\vert = O_{c,C}(\eta)$ for each $i$ and $j$. 

Let $\mathbf{w_1^*}, \dots,\mathbf{w_{d-u}^*}$ denote the standard dual basis of $(\mathbb{R}^{d-u})^*$, and define \[\omega^\prime: = \sum\limits_{j=1}^{d-u} \mu_j^\prime \mathbf{w_j^*}.\] Certainly $\omega^\prime \in (\mathbb{Z}^{d-u})^T$. We claim that there exists a map $\varphi^\prime \in (\mathbb{R}^{m-u})^*$ such that $\tau_1 \ll_{c,C} \Vert \varphi^\prime \Vert_\infty \ll_{c,C} \tau_2^{-1}$ and $\Vert (L^\prime)^* \varphi^\prime - \omega^\prime\Vert_\infty \ll_{c,C} \eta$, which will immediately resolve (\ref{claim of part 9}) and part (9). \\

Indeed, recall the decomposition $\mathbb{R}^m =  (\spn(L\mathbf{x_i}: i\leqslant u)) \oplus \ker \Theta$ as an algebraic direct sum from (\ref{direct sum}). Let $\varphi = \varphi_1 + \varphi_2$, where $\varphi_1 \in(\spn(L\mathbf{x_i}: i\leqslant u))^{0} $ and $\varphi_2 \in (\ker \Theta)^{0}.$ Since $\dist(\varphi, (\ker \Theta)^{0}) \geqslant \tau_1$, we have $\Vert \varphi_1\Vert_\infty \geqslant \tau_1$. By the properties of the matrix $P$ ((\ref{equation properties of P}) and (\ref{equation properties of P 2})) there exists some $\varphi^\prime \in (\mathbb{R}^{m-u})^*$ such that \[\varphi_1 = P^* \pi_{m-u}^* \varphi^\prime.\] Furthermore, by evaluating $\varphi^\prime$ at the standard basis vectors, one sees that \[ \tau_1 \ll_{c,C} \Vert \varphi^\prime \Vert_\infty \ll_{c,C} \tau_2^{-1}.\] We shall use this $\varphi^\prime$.

By evaluating $L^*\varphi_1$ at the elements of $\mathcal{B}$ one immediately sees that  \[L^*\varphi_1 = \sum\limits_{j=1}^{d-u} \mu_j \Xi(\mathbf{w_j})^*.\] Hence \[\Xi^*L^*P^* \pi_{m-u}^* \varphi^\prime = \sum\limits_{j=1}^{d-u} \mu_j \mathbf{w_j^*},\] in other words $(L^\prime)^* \varphi^\prime = \sum_{j=1}^{d-u}\mu_j \mathbf{w_j^*}.$ But since $\vert \mu_j - \mu_j^\prime\vert = O_{c,C}(\eta)$ for each $j$, one has $\Vert (L^\prime)^* \varphi^\prime - \omega^\prime\Vert_\infty = O_{c,C} (\eta)$ as required. This settles part (9). \\

\noindent The entire lemma is settled. 
\end{proof}

The final lemma we need in order to deduce Theorem \ref{Main Theorem chapter 3} involves removing the sharp cut-off $G$. 

\begin{Lemma}[Removing image cut-off]
\label{Lemma making G lipschitz}
Let $m,d,h$ be natural numbers, satisfying $d\geqslant h\geqslant m+1$. Let $c,C,\varepsilon$ be positive, and let $\sigma_G$ be any parameter in the range $0<\sigma_G<1/2$. Let $L^\prime:\mathbb{R}^h \longrightarrow \mathbb{R}^m$ be a purely irrational surjective map, and let $\Xi:\mathbb{R}^{h} \longrightarrow \mathbb{R}^d$ be an injective map. Suppose that $\Vert L^\prime\Vert_\infty \leqslant C$ and that $\dist(L^\prime, V_{\rank}(m,h)) \geqslant c$. Let $F_{\widetilde{\mathbf{r}}}:\mathbb{R}^h \longrightarrow [0,1]$ be any function supported on $[-N,N]^h$, and let $G_{\widetilde{\mathbf{r}}}:\mathbb{R}^m \longrightarrow [0,1]$ be the indicator function of a convex set contained within $[-\varepsilon,\varepsilon]^m$. Then there exists a Lipschitz function $G_{\widetilde{\mathbf{r}}, \sigma_G,1}$ supported on $[-O_{c,C,\varepsilon}(1),O_{c,C,\varepsilon}(1)]^{m}$, and with Lipschitz constant $O_{c,C,\varepsilon}(1/\sigma _G)$, such that, for any parameter $\tau_2$ in the range $0<\tau_2\leqslant 1$ and for any functions $f_1,\dots,f_d:[N] \longrightarrow [-1,1]$,
\begin{align*}
\vert &T_{F_{\widetilde{\mathbf{r}}},G_{\widetilde{\mathbf{r}}},N}^{L^\prime,\Xi,\widetilde{\mathbf{r}}}(f_1,\dots,f_d)\vert \\ &\ll_{c,C,\varepsilon} \vert T_{F_{\widetilde{\mathbf{r}}},G_{\widetilde{\mathbf{r}},\sigma_G,1},N}^{L^\prime,\Xi,\widetilde{\mathbf{r}}}(f_1,\dots,f_d)\vert + \sigma_G + \frac{\tau_2^{1/2}}{\sigma _G} +  \frac{\tau_2^{-O(1)}A_L(\Omega_{c,C}(1),\tau_2)^{-1}}{N}.
\end{align*}
\end{Lemma}

\begin{proof}
Applying Lemma \ref{Lipschitz approximation of convex cutoffs} to the function $G_{\widetilde{\mathbf{r}}}$, we have \[G_{\widetilde{\mathbf{r}}} = G_{\widetilde{\mathbf{r}}, \sigma_G,1} + O(G_{\widetilde{\mathbf{r}}, \sigma_G,2}),\] where $G_{\widetilde{\mathbf{r}}, \sigma_G,1}, G_{\widetilde{\mathbf{r}}, \sigma_G,2}:\mathbb{R}^{m} \longrightarrow [0,1]$ are Lipschitz functions with Lipschitz constant $O_{c,C,\varepsilon}(1/\sigma _G)$, both supported on $[-O_{c,C,\varepsilon}(1),O_{c,C,\varepsilon}(1)]^{m}$, and with $\int_{\mathbf{x}} G_{\widetilde{\mathbf{r}}, \sigma_G,2}(\mathbf{x}) \, d\mathbf{x} = O_{c,C,\varepsilon}(\sigma_G)$. 

By the triangle inequality, \[ \vert T_{F_{\widetilde{\mathbf{r}}},G_{\widetilde{\mathbf{r}},\sigma_G,2},N}^{L^\prime,\Xi,\widetilde{\mathbf{r}}}(1,\dots,1)\vert \leqslant T_{F_{\widetilde{\mathbf{r}}},G_{\widetilde{\mathbf{r}},\sigma_G,2},N}^{L^\prime}(1,\dots,1).\] We now apply Lemma \ref{Lemma upper bound involving integral}, with linear map $L^\prime$ and Lipschitz function $G_{\widetilde{\mathbf{r}}, \sigma_G,2}$. Inserting the bound from Lemma \ref{Lemma upper bound involving integral}, the present lemma follows. 
\end{proof}

We conclude this section by combining the three previous lemmas, along with Theorem \ref{Theorem rational set out version}, to deduce our main result.

\begin{proof}[\textbf{Proof of Theorem \ref{Main Theorem chapter 3} assuming Theorem \ref{Theorem rational set out version}}]
Assume the hypotheses of Theorem \ref{Main Theorem chapter 3}. Let $\sigma_F$ and $\sigma_G$ be any parameters satisfying $0<\sigma_F,\sigma_G<1/2$, and let $\tau_2$ be any parameter satisfying $0<\tau_2\leqslant 1$. 

By Lemma \ref{Lemma replacing F cut-off}, \[\vert T_{F,G,N}^L(f_1,\dots,f_d)\vert \leqslant \vert T_{F_{1,\sigma_F},G,N}^L(f_1,\dots,f_d)\vert + O_{c,C}(\sigma_F),\] for some function $F_{1,\sigma_F}:\mathbb{R}^{d} \longrightarrow [0,1]$ supported on $[-2N,2N]^d$ and with Lipschitz constant $O(1/\sigma_F N)$. By part (4) of Lemma \ref{Lemma generating a purely irrational map}, writing $F_{1,\sigma_F}$ for $F$, we have \[\vert T_{F_{1,\sigma_F},G,N}^L(f_1,\dots,f_d)\vert  \leqslant \sum\limits_{\widetilde{\mathbf{r}} \in \widetilde{R}}\vert T_{F_{\widetilde{\mathbf{r}}},G_{\widetilde{\mathbf{r}}},N}^{L^\prime,\Xi,\widetilde{\mathbf{r}}}(f_1,\dots,f_d)\vert, \] where the objects $F_{\widetilde{\mathbf{r}}}$, $G_{\widetilde{\mathbf{r}}}$, $L^\prime$, $\Xi$ and $\widetilde{R}$ satisfy all the conclusions of that lemma. \\

 Parts (1), (5) and (6) of Lemma \ref{Lemma generating a purely irrational map} show that $\Xi$ and $L^\prime$ satisfy the hypotheses of Lemma \ref{Lemma making G lipschitz}, where in the notation of Lemma \ref{Lemma making G lipschitz} we take $h: = d-u$ and rewrite $m$ for $m-u$. So, applying Lemma \ref{Lemma making G lipschitz}, there are some Lipschitz functions $G_{\widetilde{\mathbf{r}},\sigma_G,1}:\mathbb{R}^{m-u} \longrightarrow [0,1]$ supported on $[-O_{c,C,\varepsilon}(1),O_{c,C,\varepsilon}(1)]^{m-u}$ and with Lipschitz constant $O_{c,C,\varepsilon}(1/\sigma_G)$ such that 
\begin{align}
&\vert T_{F,G,N}^L(f_1,\dots,f_d)\vert \nonumber \\
&\ll_{c,C,\varepsilon}  \sum\limits_{\widetilde{\mathbf{r}} \in \widetilde{R}}\vert T_{F_{\widetilde{\mathbf{r}}},G_{\widetilde{\mathbf{r}},\sigma_G,1},N}^{L^\prime,\Xi,\widetilde{\mathbf{r}}}(f_1,\dots,f_d)\vert + \sigma_G + \frac{\tau_2^{1/2}}{\sigma _G} +  \frac{\tau_2^{-O(1)}A_{L^\prime}(\Omega_{c,C}(1),\tau_2)^{-1}}{N} + \sigma_F.
\end{align}
\noindent (Recall that $\vert \widetilde{R}\vert = O_{c,C,\varepsilon}(1)$, by part (2) of Lemma \ref{Lemma generating a purely irrational map}.)

By conclusion (8) of Lemma \ref{Lemma generating a purely irrational map}, we may replace the term $A_{L^\prime}(\Omega_{c,C}(1),\tau_2)^{-1}$ with the term $A_{L}(\Omega_{c,C}(1),\Omega_{c,C}(\tau_2))^{-1}$. \\

Since $F_{\widetilde{\mathbf{r}}}$, $L^\prime$, $\Xi$, and $\widetilde{R}$ together satisfy conclusions (1), (2), (3), (6), and (7) of Lemma \ref{Lemma generating a purely irrational map}, the hypotheses are satisfied so that we may apply Theorem \ref{Theorem rational set out version} to the expression $T_{F_{\widetilde{\mathbf{r}}},G_{\widetilde{\mathbf{r}},\sigma_G,1},N}^{L^\prime,\Xi,\widetilde{\mathbf{r}}}(f_1,\dots,f_d)$. (We take $h = d-u$ and rewrite $m$ for $m-u$, as above). Therefore there exists an $s$ at most $d-2$, independent of $F_{\widetilde{\mathbf{r}}}$, $ G_{\widetilde{\mathbf{r}}}$ and $\widetilde{\mathbf{r}}$, such that, if \[ \min_j \Vert f_j\Vert_{U^{s+1}[N]} \leqslant \rho,\] for some $\rho$ in the range $0<\rho \leqslant 1$ then $\vert T_{F,G,N}^L(f_1,\dots,f_d)\vert$ is 

\begin{align}
\label{final expression of reduction section}
\ll_{c,C,\varepsilon} \rho^{\Omega(1)}& (\sigma_F^{-O(1)} + \sigma_G^{-O(1)}) + \sigma_F^{-O(1)}N^{-\Omega(1)} \nonumber \\
&+ \sigma_G + \frac{\tau_2^{1/2}}{\sigma _G} +  \frac{\tau_2^{-O(1)}A_L(\Omega_{c,C}(1),\Omega_{c,C}(\tau_2))^{-1}}{N} + \sigma_F.
\end{align}

It remains to pick appropriate parameters. Let $C_1$ be a constant that is suitably large in terms of $c$, $C$, and all $O(1)$ constants, and let $c_1$ be a constant that is suitably small in terms of all $O(1)$ constants. Pick $\sigma_F := \sigma_G := \rho^{c_1}$ and $\tau_2 :=  C_1 \rho$. Then \[ \vert T_{F,G,N}^L(f_1,\dots,f_d)\vert \ll_{c,C,\varepsilon} \rho^{\Omega(1)} + o_{\rho,A_L,c,C}(1)\] as $N\rightarrow \infty$, where, after the combining the various error terms from (\ref{final expression of reduction section}), the $o_{\rho,A_L,c,C}(1)$ term may be bounded above by \[N^{-\Omega(1)} \rho^{-O(1)}A_L(\Omega_{c,C}(1), \rho)^{-1},\] as $A_L(\tau_1,\tau_2)$ is monotonically decreasing as $\tau_2$ decreases. This is the desired conclusion of Theorem \ref{Main Theorem chapter 3}.  
\end{proof}

\section{Transfer from $\mathbb{Z}$ to $\mathbb{R}$}
\label{section transfer}
Our remaining task is to prove Theorem \ref{Theorem rational set out version}. We devote this section to the formulation and proof of a certain `transfer' argument, whereby we replace the discrete summation in the definition of $T_{F,G,N}^{L,\Xi,\widetilde{\mathbf{r}}}(f_1,\dots,f_d)$ with an integral $\widetilde{T}^{L,\Xi,\widetilde{\mathbf{r}}}_{F,G,N}(g_1,\dots,g_{d})$. This manoeuvre will be extremely useful in the sequel, as it gives us access to the standard techniques of manipulating real integrals (in particular reparametrisation of variables). These reparametrisations may be attempted directly in the context of the discrete summation $T_{F,G,N}^{L,\Xi,\widetilde{\mathbf{r}}}(f_1,\dots,f_d)$, but the results will be messy, and one will need to control the error term each time such a reparamterisation is undertaken. It is easier in our view to do a single approximation at the beginning, so that we may subsequently reparametrise at will. As we remarked in Section \ref{Section proof strategy}, there is a somewhat analogous device in \cite{GT10}, in which the authors transfer their combinatorial expressions into summations over a field (a finite field $\mathbb{Z}/N^\prime \mathbb{Z}$ for some prime $N^\prime$, in their case), in order that their algebraic manipulations may be simplified. The natural field to use in our setting is $\mathbb{R}$. \\

Let us introduce some notation for the integral in question. 

\begin{Definition}
\label{Definition really continuous solution count}
Let $N,m,d,h$ be natural numbers, with $d\geqslant h\geqslant m+2$. Let $\varepsilon$ be positive. Let $\Xi = (\xi_1,\dots,\xi_d):\mathbb{R}^h\longrightarrow \mathbb{R}^d$ and $L:\mathbb{R}^{h}\longrightarrow \mathbb{R}^m$ be linear maps. Let $F:\mathbb{R}^{h}\rightarrow [0,1]$ and $G:\mathbb{R}^m\rightarrow [0,1]$ be two functions, with $F$ supported on $[-N,N]^h$ and $G$ supported on $[-\varepsilon,\varepsilon]^m$. Let $g_1,\dots,g_{d}:\mathbb{R}\longrightarrow [-1,1]$ be arbitrary functions. We define
\begin{equation}
\label{definiton of the solution count form}
\widetilde{T}^{L,\Xi,\widetilde{\mathbf{r}}}_{F,G,N}(g_1,\dots,g_{d}) := \frac{1}{N^{h-m}}\int\limits_{\mathbf{x}\in \mathbb{R}^h}\Big(\prod\limits_{j=1}^{d}g_j(\xi_j(\mathbf{x})+ \widetilde{\mathbf{r}_j})\Big)F(\mathbf{x})G(L\mathbf{x})\, d\mathbf{x}.
\end{equation}
\end{Definition}

Next, we determine a particular class of measurable functions that will be useful to us. 

\begin{Definition}[$\eta$-supported]
\label{Defintion eta supported}
Let $\chi:\mathbb{R}\longrightarrow [0,1]$ be a measurable function, and let $\eta$ be a positive parameter. We say that $\chi$ is \emph{$\eta$-supported} if $\chi$ is supported on $[-\eta,\eta]$ and $\chi(x) \equiv 1$ for all $x\in [-\eta/2,\eta/2]$. 
\end{Definition}

\begin{Definition}[Convolution]
If $f:\mathbb{Z}\longrightarrow \mathbb{R}$ has finite support, and $\chi:\mathbb{R}\longrightarrow [0,1]$ is a measurable function, we may define the (rather singular) convolution  $(f\ast \chi)(x) : \mathbb{R}\longrightarrow \mathbb{R}$ by \[ (f\ast \chi)(x): = \sum\limits_{n\in \mathbb{Z}} f(n) \chi(x -n).\] 
\end{Definition}
We note that if $\chi$ is $\eta$-supported, for small enough $\eta$, then there is only one possible integer $n$ that makes a non-zero contribution to above summation. \\

We now state the key lemma. 
\begin{Lemma}
\label{Lemma transfer equation}
Let $N,m,d, h$ be natural numbers, with $d\geqslant h\geqslant m+2$, and let $c,C,\varepsilon,\eta$ be positive constants. Let $\Xi:\mathbb{R}^h\longrightarrow \mathbb{R}^d$ be an injective linear map with integer coefficients, and assume that $\Xi(\mathbb{Z}^{h}) = \mathbb{Z}^d \cap \im \Xi$. Let $L:\mathbb{R}^h \longrightarrow \mathbb{R}^m$ be a surjective linear map. Assume that $\Vert \Xi \Vert_\infty \leqslant C$, $\Vert L \Vert_\infty \leqslant C$, and $\dist(L,V_{\rank}(m,h)) \geqslant c$. Let $F:\mathbb{R}^h\longrightarrow [0,1]$ be a Lipschitz function supported on $[-N,N]^h$ with Lipschitz constant $O(1/\sigma_FN)$, and let $G:\mathbb{R}^{m}\longrightarrow [0,1]$ be a Lipschitz function supported on $[-\varepsilon,\varepsilon]^m$ with Lipschitz constant $O(1/\sigma_G)$. Let $\widetilde{\mathbf{r}}$ be a fixed vector in $\mathbb{Z}^d$, satisfying $\Vert \widetilde{\mathbf{r}}\Vert_\infty = O_{c,C,\varepsilon}(1)$. Let $\chi:\mathbb{R}\longrightarrow [0,1]$ be an $\eta$-supported measurable function. Then, if $\eta$ is small enough (in terms of the dimensions $m,d,h$, $C$, and $\varepsilon$) there exists some positive real number $C_{\Xi,\chi}$ such that, if $f_1,\dots,f_d:[N]\longrightarrow [-1,1]$ are arbitrary functions,
\begin{equation}
\label{eqn transfer equation}
T_{F,G,N}^{\Xi,L,\widetilde{\mathbf{r}}}(f_1,\dots,f_d) = \frac{1}{C_{\Xi,\chi}\eta^{h}}\widetilde{T}_{F,G,N}^{\Xi,L,\widetilde{\mathbf{r}}}(f_1\ast \chi,\dots,f_d\ast \chi) + O_{C,c,\varepsilon}(\eta/\sigma_G) + O_{C,c,\varepsilon}(\eta/\sigma_F N).
\end{equation}
\noindent Moreover, $C_{\Xi,\chi} \asymp_C 1$.
\end{Lemma}
\noindent This lemma is a rigorous formulation of equation (\ref{Lipschitz is needed equation}) from the proof strategy in Section \ref{Section proof strategy}. It is in fact the only part of the proof of Theorem \ref{Theorem rational set out version} in which we use the fact that $G$ is Lipschitz. 
\begin{proof}
Let $\boldsymbol{\chi}:\mathbb{R}^d \longrightarrow [0,1]$ denote the function $\mathbf{x}\mapsto \prod\limits_{i=1}^d \chi(x_i)$. We choose \[ C_{\Xi,\chi} : = \frac{1}{\eta^h}\int\limits_{\mathbf{x} \in \mathbb{R}^h} \bc(\Xi(\mathbf{x})) \, d\mathbf{x} .\] Since $\chi$ is $\eta$-supported, $C_{\Xi,\chi} \asymp_C 1$ . 

Then, expanding the definition of the convolution, \[\frac{1}{C_{\Xi,\chi}\eta^{h}}\widetilde{T}_{F,G,N}^{L,\Xi,\widetilde{\mathbf{r}}}(f_1\ast \chi,\dots,f_d\ast \chi)\] equals
\begin{align}
\label{equation transferring zero}
&\frac{1}{N^{h-m}}\sum\limits_{\mathbf{n}\in \mathbb{Z}^d}\Big(\prod\limits_{j=1}^d f_j(n_j)\Big) \frac{1}{C_{\Xi,\chi}\eta^h} \int\limits_{\mathbf{y}\in \mathbb{R}^h} F(\mathbf{y}) G(L\mathbf{y}) \boldsymbol{\chi}(\Xi(\mathbf{y})  + \widetilde{\mathbf{r}} - \mathbf{n}) \, d\mathbf{y}. 
\end{align}
Note that any vector $\mathbf{n} \in \mathbb{Z}^d$ that gives a non-zero contribution to expression (\ref{equation transferring zero}) satisfies $\Vert \mathbf{n} - \Xi(\mathbf{y}) - \widetilde{\mathbf{r}}\Vert_\infty \ll \eta$, for some $\mathbf{y} \in \mathbb{R}^h$. Therefore, $\mathbf{n}$ must be of the form $\Xi(\mathbf{n^\prime}) + \widetilde{\mathbf{r}}$ for some unique $\mathbf{n^\prime} \in\mathbb{Z}^h$. (This is proved in full in Lemma \ref{Lemma integer distance}). Therefore, writing $\Xi = (\xi_1,\dots,\xi_d)$, we may reformulate (\ref{equation transferring zero}) as
\begin{align}
&\frac{1}{N^{h-m}}\sum\limits_{\mathbf{n}\in \mathbb{Z}^h} \Big(\prod\limits_{j=1}^d f_j(\xi_j(\mathbf{n}) + \widetilde{\mathbf{r}}_j)\Big) \frac{1}{C_{\Xi,\chi}\eta^h}\int\limits_{\mathbf{y}\in \mathbb{R}^h} F(\mathbf{y}) G(L\mathbf{y}) \boldsymbol{\chi}(\Xi(\mathbf{y}-\mathbf{n}))\, d\mathbf{y},\nonumber
\end{align}
\noindent which is equal to 
\begin{equation}
\label{extra equation to be referenced in transfer argument}
\frac{1}{N^{h-m}}\sum\limits_{\mathbf{n}\in \mathbb{Z}^h} \Big(\prod\limits_{j=1}^d f_j(\xi_j(\mathbf{n}) + \widetilde{\mathbf{r}}_j)\Big) \frac{1}{C_{\Xi,\chi}\eta^h}\int\limits_{\mathbf{y}\in \mathbb{R}^h}  (F(\mathbf{n}) + O_C(\eta/\sigma_F N))G(L\mathbf{y}) \boldsymbol{\chi}(\Xi(\mathbf{y} - \mathbf{n})) \, d\mathbf{y}.
\end{equation}
Indeed, the inner integral is only non-zero when  $\Vert \Xi(\mathbf{y}) - \Xi(\mathbf{n})\Vert_\infty\ll\eta$, and this implies that $\Vert \mathbf{y} - \mathbf{n}\Vert_\infty \ll C^{-O(1)} \eta$. (This is proved in full in Lemma \ref{Lemma bounded inverse}). Then recall that $F$ has Lipschitz constant $O(1/\sigma_F N)$.\\

 Continuing, expression (\ref{extra equation to be referenced in transfer argument}) is equal to 
\begin{equation}
\label{equation transferring}
\frac{1}{N^{h-m}} \sum\limits_{\mathbf{n}\in\mathbb{Z}^h}\Big(\prod\limits_{j=1}^d f_j(\xi_j(\mathbf{n}) + \widetilde{\mathbf{r}}_j)\Big) F(\mathbf{n}) H(L\mathbf{n} ) + E
\end{equation}
\noindent where \[H(\mathbf{x}) = \frac{1}{C_{\Xi,\chi}\eta^h}\int\limits_{\mathbf{y}\in\mathbb{R}^h}  \boldsymbol{\chi}(\Xi(\mathbf{y}))G(\mathbf{x} + L\mathbf{y})\ \, d\mathbf{y}\] and $E$ is a certain error, that may be bounded above by 
\begin{equation}
\label{transfer error bounded above}
\ll_C \frac{\eta}{\sigma_F N} \frac{1}{N^{h-m}} \sum\limits_{\mathbf{n}\in[-O(N),O(N)]^h}H(L\mathbf{n}).
\end{equation}  

Let us deal with the first term of (\ref{equation transferring}), in which we wish to replace $H$ with $G$. We therefore consider
\begin{equation*}
\Big\vert\frac{1}{N^{h-m}} \sum\limits_{\mathbf{n}\in\mathbb{Z}^h}\Big(\prod\limits_{j=1}^d f_j(\xi_j(\mathbf{n}) + \widetilde{\mathbf{r}}_j)\Big) F(\mathbf{n}) (G(L\mathbf{n}) - H(L\mathbf{n})) \Big\vert,
\end{equation*}
which is
\begin{equation}
\label{equation transferring main term differencing}
\leqslant \frac{1}{N^{h-m}} \sum\limits_{\mathbf{n}\in\mathbb{Z}^h}F(\mathbf{n}) \vert G-H\vert(L\mathbf{n}).
\end{equation} Using Lemma \ref{Lemma bounded inverse} again, the function $H$ is supported on $[-\varepsilon - O_C(\eta), \varepsilon + O_C(\eta)]^m$. Thus, if $\eta$ is small enough in terms of $\varepsilon$, the function $\vert G-H\vert:\mathbb{R}^m\longrightarrow \mathbb{R}$ is supported on $[-O_C(\varepsilon), O_C(\varepsilon)]^m$. Furthermore, $\Vert G - H\Vert_\infty = O_C(\eta/\sigma_G)$. Indeed, 
\begin{align*}
&G(\mathbf{x}) - \frac{1}{C_{\Xi,\chi}\eta^h}\int\limits_{\mathbf{y}\in\mathbb{R}^h} G(\mathbf{x}+ L\mathbf{y}) \boldsymbol{\chi}(\Xi(\mathbf{y})) \, d\mathbf{y} \\
&= G(\mathbf{x}) - \frac{1}{C_{\Xi,\chi}\eta^h} \int\limits_{\mathbf{y} \in\mathbb{R}^h} (G(\mathbf{x}) + O_C(\eta/\sigma_G))\boldsymbol{\chi}(\Xi(\mathbf{y})) \, d\mathbf{y}\\
& = O_C(\eta/\sigma_G),\nonumber
\end{align*}
\noindent by the definition of $C_{\Xi,\chi}$ and using the Lipschitz property of $G$. So, by the crude bound given in Lemma \ref{Lemma crude bound on number of solutions}, (\ref{equation transferring main term differencing}) may be bounded above by $O_{c,C,\varepsilon}(\eta/\sigma_G)$. \\

Turning to the error $E$ from (\ref{equation transferring}), we've already remarked that it may be bounded above by expression (\ref{transfer error bounded above}). Applying Lemma \ref{Lemma crude bound on number of solutions} again, expression (\ref{transfer error bounded above}) may be bounded above by $O_{c,C,\varepsilon}(\eta/\sigma_F N)$. \\

Lemma \ref{Lemma transfer equation} follows immediately upon substituting the estimates on (\ref{transfer error bounded above}) and (\ref{equation transferring main term differencing}) into \eqref{equation transferring}.
\end{proof}

We finish this section by noting a simple relationship between the Gowers norms $\Vert f\ast \chi \Vert_{U^{s+1}(\mathbb{R},2N)}$ and the Gowers norms $\Vert f \Vert_{U^{s+1}[N]}$. 

\begin{Lemma}[Relating different Gowers norms]
\label{Lemma linking different Gowers norms}
Let $s$ be a natural number, and assume that $\eta$ is a positive parameter that is small enough in terms of $s$. Let $\chi:\mathbb{R}\longrightarrow [0,1]$ be an $\eta$-supported measurable function. Let $N$ be a natural number, and let $f:[N]\longrightarrow \mathbb{R}$ be an arbtirary function. View $f\ast \chi$ as a function supported on $[-2N,2N]$. Then we have \begin{equation}
\label{linking different gowers norms}
\Vert f\ast \chi\Vert_{U^{s+1}(\mathbb{R},2N)}\ll \eta^{\frac{s+2}{2^{s+1}}} \Vert f\Vert_{U^{s+1}[N]}. 
\end{equation}
\end{Lemma}
\noindent The definition of the real Gowers norm $\Vert f\ast \chi\Vert_{U^{s+1}(\mathbb{R},2N)}$ is recorded in Definition \ref{Definition of Gowers norms over R}. 
\begin{proof}
From expression (\ref{explicit real gowers norms}), we have $$\Vert f\ast \chi \Vert _{U^{s+1}(\mathbb{R},2N)}^{2^{s+1}} \ll \frac{1}{N^{s+2}}\int\limits_{(x,\mathbf{h})\in \mathbb{R}^{s+2}}\prod\limits_{\boldsymbol{\omega}\in \{0,1\}^{s+1}} (f\ast \chi)(x+ \mathbf{h}\cdot\boldsymbol{\omega}) \, dx\, d\mathbf{h}.$$ Substituting in the definition of $f\ast \chi$, this is equal to 
\begin{equation}
\label{comparing gowers norms equation}
\frac{1}{N^{s+2}} \sum\limits_{(n_{\boldsymbol{\omega}})_{\boldsymbol{\omega}\in \{0,1\}^{s+1}} \in \mathbb{Z}^{\{0,1\}^{s+1}}} \Big(\prod\limits_{\boldsymbol{\omega} \in\{0,1\}^{s+1}} f(n_{\boldsymbol{\omega}})\Big)\int\limits_{(x,\mathbf{h})\in \mathbb{R}^{s+2}} \boldsymbol{\chi}(\Psi(x,\mathbf{h}) - \mathbf{n}) \, dx \, d\mathbf{h}, 
\end{equation}
where $\Psi:\mathbb{R}^{s+2}\longrightarrow \mathbb{R}^{2^{s+1}}$ has coordinate functions $\psi_{\boldsymbol{\omega}}$, indexed by $\bo \in \{0,1\}^{s+1}$, where $\psi_{\boldsymbol{\omega}}(x,\mathbf{h}) : = x + \mathbf{h} \cdot \bo$. In similar notation to that used in the previous proof, for $\mathbf{x} \in \mathbb{R}^{2^{s+1}}$, we let $\bc(\mathbf{x}) : = \prod_{i=1}^{2^{s+1}} \chi(x_i)$.  Note that $\Psi$ is injective, $\Psi(\mathbb{Z}^{s+2}) = \mathbb{Z}^{2^{s+1}} \cap \im \Psi$, and $\Vert \Psi\Vert_\infty = O_s(1)$.  

The contribution to the inner integral of (\ref{comparing gowers norms equation}) from a particular $\mathbf{n}$ is zero unless $\Vert \mathbf{n} - \Psi(x,\mathbf{h})\Vert_\infty \ll \eta$, for some $(x,\mathbf{h}) \in \mathbb{R}^{s+2}$. Therefore, if $\eta$ is small enough we can conclude that $\mathbf{n}$ must be of the form $\Psi(p,\mathbf{k})$, for some unique $(p,\mathbf{k}) \in \mathbb{Z}^{s+2}$. (To spell it out, apply Lemma \ref{Lemma integer distance} with the map $\Psi$ in place of the map $\Xi$). So (\ref{comparing gowers norms equation}) is equal to 
\begin{equation}
\frac{1}{N^{s+2}} \sum\limits_{(p,\mathbf{k}) \in \mathbb{Z}^{s+2}} \Big(\prod\limits_{\bo \in \{0,1\}^{s+1}} f(p + \mathbf{k} \cdot \bo) \Big) \int\limits_{(x,\mathbf{h})\in \mathbb{R}^{s+2}} \boldsymbol{\chi}(\Psi(x -p,\mathbf{h} - \mathbf{k})) \, dx \, d\mathbf{h}, 
\end{equation}
\noindent which, after a change of variables , is equal to 
\begin{equation}
\label{real gn becomes integer gn}
\frac{C}{N^{s+2}}\sum\limits_{(p,\mathbf{k})\in \mathbb{Z}^{s+2}}\prod\limits_{\boldsymbol{\omega}\in \{0,1\}^{s+1}} f(p+ \mathbf{k}\cdot\boldsymbol{\omega}),
\end{equation}
\noindent where \[C:=\int\limits_{(x,\mathbf{h})\in \mathbb{R}^{s+2}} \boldsymbol{\chi} (\Psi(x,\mathbf{h})) \, dx \, d\mathbf{h}.\] Since $\boldsymbol{\chi}$ has support contained within $[-\eta,\eta]^{2^{s+1}}$, a vector $(x,\mathbf{h})$ only makes a non-zero contribution to the above integral if $\Vert \Psi(x,\mathbf{h})\Vert_\infty \ll \eta$. This implies that $\Vert (x,\mathbf{h})\Vert_\infty \ll \eta$. (To prove this is full, apply Lemma \ref{Lemma bounded inverse} to the linear map $\Psi$). Since $\Vert \bc\Vert_\infty = O(1)$, this means $C = O(\eta^{s+2})$. The lemma then follows from (\ref{real gn becomes integer gn}). 
\end{proof}

\section{Degeneracy relations}
\label{section Degeneracy relations}
Our aim for this short section is to establish a quantitative relationship between the dual pair degeneracy variety $V_{\degen,2}^*(m,d,h)$ and the degeneracy variety $V_{\degen}(h-m,d)$ (see Definitions \ref{Definition dual pair degeneracy variety} and \ref{Definition degeneracy varieties} respectively), which will be needed in the next section. It is here that we show that $V_{\degen,2}^*(m,d,h)$ was indeed the appropriate notion for guaranteeing finite Cauchy-Schwarz complexity of the relevant system of homogeneous linear forms. We direct the reader to Proposition \ref{Proposition easy degeneracy relation} and the discussion after Definition \ref{Definition dual pair degeneracy variety} for more on this issue.

To introduce the ideas, we first prove a non-quantitative proposition (which is a generalisation of Proposition \ref{Proposition easy degeneracy relation}).

\begin{Lemma}
\label{Lemma non quantitative equivalence of degeneracies}
Let $m,d,h$ be natural numbers, with $d\geqslant h\geqslant  m+2$. Let $\Xi:\mathbb{R}^h\longrightarrow \mathbb{R}^d$ be an injective linear map, let $L:\mathbb{R}^h\longrightarrow \mathbb{R}^m$ be a surjective linear map, and suppose that $(\Xi,L)\notin V^*_{\degen,2}(m,d,h)$. Let $\Phi:\mathbb{R}^{h-m}\longrightarrow \ker L$ be any surjective linear map. Then the linear map $\Xi \Phi:\mathbb{R}^{h-m} \longrightarrow \mathbb{R}^d$, viewed as a system of homogenous linear forms, is not in $V_{\degen}(h-m,d)$.  
\end{Lemma}
\begin{proof}
Let $\mathbf{e_1},\dots,\mathbf{e_d}$ denote the standard basis vectors in $\mathbb{R}^d$, and let $\mathbf{e}_\mathbf{1}^*, \dots,\mathbf{e}_\mathbf{d}^*$ denote the dual basis of $(\mathbb{R}^d)^*$. Suppose for contradiction that $\Xi\Phi \in V_{\degen}(h-m,d)$. Then by Proposition \ref{Proposition easy degeneracy relation} there exist two indices $i,j\leqslant d$, and a real number $\lambda$, such that $\mathbf{e}_{\mathbf{i}}^* - \lambda \mathbf{e}_{\mathbf{j}}^*$ is non-zero and $\Xi\Phi(\mathbb{R}^{h-m})\subset \ker (\mathbf{e}_{\mathbf{i}}^* - \lambda \mathbf{e}_{\mathbf{j}}^*)$. 

But then $\Phi(\mathbb{R}^{h-m}) \subset \ker (\Xi^*(\mathbf{e}_{\mathbf{i}}^* - \lambda \mathbf{e}_{\mathbf{j}}^*))$, i.e. $\Xi^*(\mathbf{e}_{\mathbf{i}}^* - \lambda \mathbf{e}_{\mathbf{j}}^*) \in (\ker L)^{0}$. But $(\ker L)^{0} = L^*((\mathbb{R}^m)^*)$, and so $\Xi^*(\mathbf{e}_{\mathbf{i}}^* - \lambda \mathbf{e}_{\mathbf{j}}^*) \in L^*((\mathbb{R}^m)^*)$. 

Then, by the definition of $V_{\degen,2}^*(m,d,h)$, we have $(\Xi,L)\in V^*_{\degen,2}(m,d,h)$, which is a contradiction. 
\end{proof}

The ideas having been introduced, we state the quantitative version we require. 

\begin{Lemma}
\label{Quantitative orthonormal basis parametrisation}
Let $m,d,h$ be natural numbers, with $d\geqslant h\geqslant m+2$, and let $c,C$ be positive constants. Let $\Xi:\mathbb{R}^h \longrightarrow \mathbb{R}^d$ be a linear map, and let $L:\mathbb{R}^h \longrightarrow \mathbb{R}^m$ be a surjective linear map. Suppose that $\Vert \Xi\Vert_\infty \leqslant C$, and $\operatorname{dist}((\Xi,L),V^*_{\degen,2}(m,d,h))\geqslant c$. Let $K$ denote $\ker L$, choose any orthonormal basis $\{\mathbf{v^{(1)}},\dots, \mathbf{v^{(h-m)}}\}$ for $K$, and let $\Phi:\mathbb{R}^{h-m}\longrightarrow K$ denote the associated parametrisation, i.e. $\Phi(\mathbf{x}) := \sum_{i=1}^{h-m}x_i\mathbf{v^{(i)}}$. Then $\Vert \Xi  \Phi\Vert_{\infty} = O(C)$ and $\operatorname{dist}(\Xi\Phi, V_{\degen}(h-m,d)) = \Omega(c)$.  
\end{Lemma}
\noindent For the definition of $\operatorname{dist}((\Xi,L),V^*_{\degen,2}(m,d,h))$, consult Definition \ref{Definition distance metric for pairs of matrices}.
\begin{proof}
Certainly $\Vert \Phi\Vert_{\infty} = O(1)$, as the chosen basis  $\{\mathbf{v^{(1)}},\dots, \mathbf{v^{(h-m)}}\}$ is orthonormal. Therefore $\Vert \Xi \Phi\Vert_{\infty} = O(C)$.

Let $\mathbf{e_1},\dots,\mathbf{e_d}$ denote the standard basis vectors in $\mathbb{R}^d$, and let $\mathbf{e}_\mathbf{1}^*, \dots,\mathbf{e}_\mathbf{d}^*$ denote the dual basis of $(\mathbb{R}^d)^*$. Suppose for contradiction that $\operatorname{dist}(\Xi\Phi, V_{\degen}(h-m,d)) \leqslant \eta$ for some small parameter $\eta$. In other words, assume that there exists a linear map $P:\mathbb{R}^{h-m}\longrightarrow \mathbb{R}^d$ with $\Vert P \Vert_\infty \leqslant  \eta$ such that $\Xi\Phi + P \in V_{\degen}(h-m,d)$. By definition, this means that \[ (\Xi\Phi + P)(\mathbb{R}^{h-m}) \subset \ker (\mathbf{e}_\mathbf{i}^* - \lambda \mathbf{e}_\mathbf{j}^* ),\] for some two indices $i,j \leqslant d$, and some real number $\lambda$, such that $\mathbf{e}_\mathbf{i}^* - \lambda \mathbf{e}_\mathbf{j}^*$ is non-zero. 

We can factorise $P = Q \Phi$, for some linear map $Q:\mathbb{R}^{h} \longrightarrow \mathbb{R}^d$ with $\Vert Q\Vert_\infty \ll \eta$. Indeed, let $\mathbf{f_1},\dots,\mathbf{f_{h-m}}$ denote the standard basis vectors in $\mathbb{R}^{h-m}$, and for all $k$ at most $h-m$ define \[Q(\mathbf{v^{(k)}}): = P(\mathbf{f_k}).\] (If the notation for the indices seems odd here, it is designed to match the notation in Proposition \ref{Proposition separating out the kernel}, in which having superscript on the vectors $\mathbf{v^{(k)}}$ seems to be natural). Complete $\{\mathbf{v^{(1)}},\dots, \mathbf{v^{(h-m)}}\}$ to an orthonormal basis $\{\mathbf{v^{(1)}},\dots, \mathbf{v^{(h)}}\}$ for $\mathbb{R}^h$ and, for $k$ in the range $h-m + 1 \leqslant k\leqslant h-m$, define $Q(\mathbf{v^{(k)}}): = \mathbf{0}$. Then $P = Q\Phi$, and $\Vert Q\Vert_\infty =O(\eta)$, since $\{\mathbf{v^{(1)}},\dots, \mathbf{v^{(h)}}\}$ is an orthonormal basis.

Thus, \[  (\Xi\Phi + Q  \Phi)(\mathbb{R}^{h-m}) \subset \ker (\mathbf{e}_\mathbf{i}^* - \lambda \mathbf{e}_\mathbf{j}^* ).\] So \[ \Phi(\mathbb{R}^{h-m})\subset \ker((\Xi + Q)^* (\mathbf{e}_\mathbf{i}^* - \lambda \mathbf{e}_\mathbf{j}^* )).\] Like the previous proof, we conclude that \[ (\Xi + Q)^* (\mathbf{e}_\mathbf{i}^* - \lambda \mathbf{e}_\mathbf{j}^* ) \in L^*((\mathbb{R}^{m})^*).\]

Hence $((\Xi + Q),L) \in V_{\degen,2}^*(m,d,h)$, which, if $\eta$ is small enough, contradicts the assumption that $\operatorname{dist}((\Xi,L),V^*_{\degen,2}(m,d,h))\geqslant c$.
\end{proof}

\section{A Generalised von Neumann Theorem}
\label{section General proof of the real variable von Neumann Theorem}
In this section we complete the proof of Theorem \ref{Theorem rational set out version}, and therefore complete the proof of Theorem \ref{Main Theorem chapter 3}. It will be enough to prove the following statement.
\begin{Theorem}
\label{Theorem Generalised von Neumann Theorem over reals}
Let $N,m,d,h$ be natural numbers, with $d\geqslant h\geqslant m+2$, and let $c,C,\varepsilon$ be positive reals. Let $\Xi = \Xi(N):\mathbb{R}^{h}\longrightarrow \mathbb{R}^d$ be an injective linear map with integer coefficients, and let $L = L(N):\mathbb{R}^d \longrightarrow\mathbb{R}^m$ be a surjective linear map. Suppose further that $\Vert L\Vert_\infty \leqslant C$, $\Vert \Xi\Vert_\infty \leqslant C$, $\dist( L,V_{\rank}(m,d)) \geqslant c$ and $\dist ((\Xi,L), V_{\degen,2}^*(m,d,h))\geqslant c$. Then there is some natural number $s$ at most $d-2$, independent of $\varepsilon$, such that the following holds. Let $\widetilde{\mathbf{r}} \in \mathbb{Z}^d$ be some vector with $\Vert \widetilde{\mathbf{r}}\Vert_\infty = O_{c,C,\varepsilon}(1)$, and let $\sigma_F$ be a parameter in the range $O < \sigma_F < 1/2$. Let $F:\mathbb{R}^h\longrightarrow [0,1]$ be a Lipschitz function supported on $[-N,N]^h$, with Lipschitz constant $O(1/\sigma_FN)$, and let $G:\mathbb{R}^{m} \longrightarrow [0,1]$ be any function supported on $[-\varepsilon, \varepsilon]^m$.  Let $g_1,\dots,g_d:[-2N,2N]^d\longrightarrow [-1,1]$ be arbitrary measurable functions. Suppose \[ \min\limits_{j\leqslant d}\Vert g_j\Vert_{U^{s+1}(\mathbb{R},2N)}\leqslant \rho\] for some $\rho$ at most $1$. Then 
\begin{equation}
\vert \widetilde{T}_{F,G,N}^{L,\Xi,\widetilde{\mathbf{r}}}(g_1,\dots,g_d)\vert\ll_{c,C,\varepsilon} \rho^{\Omega(1)}\sigma_F^{-1} 
\end{equation}
\noindent 
\end{Theorem}
\begin{proof}[Proof that \ref{Theorem Generalised von Neumann Theorem over reals} implies Theorem \ref{Theorem rational set out version}] Assume the hypotheses of Theorem \ref{Theorem rational set out version}. This gives natural numbers $N,m,d,h$, linear maps $L:\mathbb{R}^h\longrightarrow \mathbb{R}^m$ and $\Xi:\mathbb{R}^h \longrightarrow \mathbb{R}^d$, and functions $F:\mathbb{R}^h \longrightarrow [0,1]$ and $G:\mathbb{R}^m \longrightarrow [0,1]$. Let $f_1,\dots,f_d:[N]\longrightarrow [-1,1]$ be arbitrary functions, and for ease of notation let \[\delta: = T _{F,G,N}^{L,\Xi,\widetilde{\mathbf{r}}}(f_1,\dots,f_d).\] From Lemma \ref{Lemma crude bound on number of solutions} and the triangle inequality, we have the crude bound $\delta = O_{c,C,\varepsilon}(1)$. 

Let $\eta: = c_1\delta \sigma_G$, where $c_1$ is small enough depending on $m,d,h,c,C,$ and $\varepsilon$, and let $\chi: \mathbb{R}\longrightarrow [0,1]$ be an $\eta$-supported measurable function (see Definition \ref{Defintion eta supported}). For all $j$ at most $d$, let $g_j: = f_j\ast \chi$. Finally, suppose $\min _j \Vert f_j\Vert_{U^{s+1}[N]} \leqslant \rho$, for some parameter $\rho$ in the range $0<\rho \leqslant 1$.

We proceed by bounding $\widetilde{T}_{F,G,N}^{L,\Xi,\widetilde{\mathbf{r}}}(g_1,\dots,g_d)$. Indeed, by Lemma \ref{Lemma linking different Gowers norms}, if $c_1$ is small enough \[\min_j\Vert g_j\Vert_{U^{s+1}(\mathbb{R})} \ll \eta^{\frac{s+2}{2^{s+1}}} \min_j\Vert f_j\Vert_{U^{s+1}[N]} \ll_{c,C,\varepsilon} \rho.\]  Applying Theorem \ref{Theorem Generalised von Neumann Theorem over reals} to these functions $g_1,\dots,g_d$, the above implies 
\begin{equation}
\label{extra extra}
\widetilde{T}_{F,G,N}^{L,\Xi,\widetilde{\mathbf{r}}}(g_1,\dots,g_d)\ll_{c,C,\varepsilon} \rho^{\Omega(1)} \sigma_F^{-1}.
\end{equation}

Now we use this to bound $\delta$ by Gowers norms. Indeed, by Lemma \ref{Lemma transfer equation}, we have 
\begin{align*}
\delta &\ll_{c,C,\varepsilon} \frac{1}{(c_1 \delta \sigma_G)^h}\widetilde{T}_{F,G,N}^{L,\Xi,\widetilde{\mathbf{r}}}(g_1,\dots,g_d) + c_1\delta + c_1\delta\sigma_G\sigma_F^{-1} N^{-1} .
\end{align*} 

Picking $c_1$ small enough, we may move the $c_1\delta$ term to the left-hand side to get an $\Omega(\delta)$ term. The bound (\ref{extra extra}) then yields
\begin{align*}
\delta^{h+1} \ll_{c,C,\varepsilon} \rho^{\Omega(1)} \sigma_F^{-1} \sigma_G^{-h} + \sigma_F^{-1} N^{-1},
\end{align*} 
and so \[ \delta \ll_{c,C,\varepsilon} \rho^{\Omega(1)} (\sigma_F^{-O(1)} + \sigma_G^{-O(1)}) + \sigma_F^{-O(1)}N^{-\Omega(1)}.\]  This yields the desired conclusion of Theorem \ref{Theorem rational set out version}.
\end{proof}

So it remains to prove Theorem \ref{Theorem Generalised von Neumann Theorem over reals}, for which the bulk of the work will be done in the following two propositions. In Proposition \ref{Proposition separating out the kernel}, we will reduce the integral in $\widetilde{T}_{F,G,N}^{L,\Xi,\widetilde{\mathbf{r}}}(g_1, \dots, g_d)$ to an integral over the kernel of $L$. This kernel will be parametrised by a map $\Psi$, which will have finite $c_1$-Cauchy-Schwarz complexity for some suitable $c_1$. In Proposition \ref{Proposition Cauchy} we will then work out the details of applying the Cauchy-Schwarz inequality to such a map, thereby producing Gowers norms. 

\begin{Proposition}[Separating out the kernel]
\label{Proposition separating out the kernel}
Let $N,m,d, h$ be natural numbers, with $d\geqslant h\geqslant m+2$, and let $c,C,\varepsilon$ be positive constants. Let $\sigma_F$ be a parameter in the range $0 <\sigma_F<1/2$. Let $\Xi:\mathbb{R}^h\rightarrow \mathbb{R}^d$ be an injective linear map with integer coefficients, and let $L:\mathbb{R}^h \longrightarrow \mathbb{R}^m$ be a surjective linear map. Assume further that $\Vert L \Vert_\infty \leqslant C$, $\Vert \Xi\Vert_\infty \leqslant C$, $\dist(L,V_{\rank}(m,h))\geqslant c$ and $\dist ((\Xi,L),V_{\degen,2}^\ast (m,d,h)) \geqslant c$. Let $F:\mathbb{R}^h\longrightarrow [0,1]$ be a Lipschitz function supported on $[- CN,CN]^h$, with Lipschitz constant $O_C(1/\sigma_F N)$, and let $G:\mathbb{R}^{m}\longrightarrow [0,1]$ be a measurable function supported on $[-\varepsilon,\varepsilon]^m$. Let $\widetilde{\mathbf{r}}$ be a fixed vector in $\mathbb{Z}^d$, satisfying $\Vert \widetilde{\mathbf{r}}\Vert_\infty = O_{C}(1)$. Then there exists a system of linear forms $(\psi_1,\dots,\psi_d) = \Psi:\mathbb{R}^{h-m}\longrightarrow \mathbb{R}^d$ satisfying $\Vert \Psi\Vert_\infty = O_C(1)$, and a Lipschitz function $F_1:\mathbb{R}^{h-m}\longrightarrow [0,1]$ supported on $[-O_{c,C,\varepsilon}(N),O_{c,C,\varepsilon}(N)]^{h-m}$ with Lipschitz constant $O(1/\sigma_F N)$, such that, if $g_1,\dots,g_d:[-2N,2N]\longrightarrow [-1,1]$ are arbitrary functions, 
\begin{equation}
\vert\widetilde{T}_{F,G,N}^{L,\Xi,\widetilde{\mathbf{r}}}(g_1,\dots,g_d)\vert\ll_{c,C,\varepsilon} \Big\vert\frac{1}{N^{h-m}}\int\limits_{\mathbf{x}} \prod\limits_{j=1}^d g_j(\psi_j(\mathbf{x}) + a_j)F_1(\mathbf{x}) \, d\mathbf{x}\Big\vert ,
\end{equation}
\noindent where, for each $j$, $a_j$ is some real number that satisfies $a_j = O_{c,C,\varepsilon}(1)$.

Furthermore, there exists a natural number $s$ at most $d-2$ such that the system $\Psi$ has $\Omega_{c,C}(1)$-Cauchy-Schwarz complexity at most $s$, in the sense of Definition \ref{Definition Cauchy Schwarz complexity}.
\end{Proposition}

\begin{proof}[Proof of Proposition \ref{Proposition separating out the kernel}]

For ease of notation, let 
\[ \beta: = \widetilde{T}_{F,G,N}^{L,\Xi,\widetilde{\mathbf{r}}}(g_1,\dots,g_d).\]
Noting that $\ker L$ is a vector space of dimension $h-m$, define $\{\mathbf{v^{(1)}},\dots,\mathbf{v^{(h-m)}}\} \subset \mathbb{R}^h$ to be an orthonormal basis for $\ker L$. Then the map $\Phi:\mathbb{R}^{h-m} \longrightarrow \mathbb{R}^h$, defined by
\begin{equation}
\label{parametrise the kernel}
\Phi(\mathbf{x}) := \sum\limits_{i=1}^{h-m}x_i \mathbf{v}^{\mathbf{(i)}},
\end{equation}
\noindent is an injective map that parametrises $\ker L$. (This is reminiscent of Lemma \ref{Quantitative orthonormal basis parametrisation}). 


Now, extend the orthonormal basis $\{\mathbf{v^{(1)}},\dots,\mathbf{v^{(h-m)}}\}$ for $\ker L$ to an orthonormal basis $\{\mathbf{v^{(1)}},\dots,\mathbf{v^{(h)}}\}$ for $\mathbb{R}^h$. By implementing a change of basis, we may rewrite $\beta$ as
\begin{equation}
\label{separating out the kernel}
\frac{1}{N^{h-m}} \int\limits_{\mathbf{x}\in \mathbb{R}^h} F(\sum\limits_{i=1}^{h}x_i\mathbf{v^{(i)}})G(L(\sum\limits_{i=1}^{h}x_i\mathbf{v^{(i)}}))(\prod\limits_{j=1}^{d} g_j(\xi_j(\Phi(\mathbf{x_1^{h-m}})+\sum\limits_{i=h-m+1}^{h} x_i \mathbf{v}^{\mathbf{(i)}})+ \widetilde{\mathbf{r}}_j)) \, d\mathbf{x},
\end{equation}
\noindent using $\mathbf{x_1^{h-m}}$ to refer to the vector in $\mathbb{R}^{h-m}$ given by the first the first $h-m$ coordinates of $\mathbf{x}$.

We wish to remove the presence of the variables $x_{h-m+1},\dots,x_{h}$. To set this up, note that, by the choice of the vectors $\mathbf{v^{(i)}}$, $$G(L(\sum\limits_{i=1}^{h}x_i\mathbf{v^{(i)}})) = G(L(\sum\limits_{i=h-m+1}^{h}x_i\mathbf{v^{(i)}})). $$ The vector $\sum_{i=h-m+1}^{h}x_i\mathbf{v^{(i)}}$ is in $(\ker L)^{\perp}$. Hence, due to the limited support of $G$, there is a domain $D$, contained in $[-O_{\varepsilon,c,C}(1), O_{\varepsilon,c,C}(1)]^m$, such that \\$G(L(\sum_{i=h-m+1}^{h}x_i\mathbf{v^{(i)}}))$ is equal to zero unless $(x_{h-m+1},\dots,x_h)^T\in D$. (This is proved in full in Lemma \ref{nonsingularity of map on orthogonal complement of Kernel}).

We can use this observation to bound the right-hand side of (\ref{separating out the kernel}). Indeed, we have 
\begin{align}
\label{equation bound beta by a sup}
\beta \ll \operatorname{vol} D \times \sup\limits_{\mathbf{x_{h-m+1}^h}\in D}\frac{1}{N^{h-m}} \Big\vert \int\limits_{\mathbf{x_{1}^{h-m}} \in \mathbb{R}^{h-m}} F(\sum\limits_{i=1}^{h}x_i\mathbf{v^{(i)}})G(L(\sum\limits_{i=h-m+1}^{h}x_i\mathbf{v^{(i)}})) \nonumber\\
(\prod\limits_{j=1}^{d} g_j(\xi_j(\Phi(\mathbf{x_{1}^{h-m}})+\sum\limits_{i=h-m+1}^{h} x_i\mathbf{v}^{\mathbf{(i)}})+ \widetilde{\mathbf{r}}_j)) \, d \mathbf{x_1^{h-m}} \Big\vert.
\end{align} So there exists some fixed vector $(x_{h-m+1},\dots,x_h)^T$ in $D$ such that \begin{align}
\label{equation bound beta second}
\beta \ll_{c,C,\varepsilon}\frac{1}{N^{h-m}} \Big\vert \int\limits_{\mathbf{x_{1}^{h-m}} \in \mathbb{R}^{h-m}} &F(\sum\limits_{i=1}^{h}x_i\mathbf{v^{(i)}})G(L(\sum\limits_{i=h-m+1}^{h}x_i\mathbf{v^{(i)}})) \nonumber\\
&(\prod\limits_{j=1}^{d} g_j(\xi_j(\Phi(\mathbf{x_{1}^{h-m}})+\sum\limits_{i=h-m+1}^{h} x_i\mathbf{v}^{\mathbf{(i)}})+ \widetilde{\mathbf{r}}_j)) \, d \mathbf{x_1^{h-m}} \Big\vert. 
\end{align}  

Define the function $F_1:\mathbb{R}^{h-m} \longrightarrow [0,1]$ by \[ F_1(\mathbf{x_1^{h-m}}) : = F(\Phi(\mathbf{x_1^{h-m}}) + \sum\limits_{i=h-m+1}^{h} x_i\mathbf{v^{(i)}})\] and for each $j$ at most $d$, a shift \[a_j :=\xi_j(\sum\limits_{i=h-m+1}^{h} x_i \mathbf{v}^{\mathbf{(i)}}) + \widetilde{\mathbf{r}}_j.\] Then
\begin{equation}
\label{just before we put everything into normal form}
\beta \ll _{c,C,\varepsilon}\Big\vert\frac{1}{N^{h-m}}\int\limits_{\mathbf{x} \in \mathbb{R}^{h-m}}F_1(\mathbf{x}) \prod\limits_{j=1}^d g_j(\xi_j(\Phi(\mathbf{x}))+a_j) \, d\mathbf{x}\Big\vert,
\end{equation} 
\noindent and $F_1$ and $a_j$ satisfy the conclusions of the proposition.

Finally, since $\dist((\Xi,L), V_{\degen,2}^*(m,d,h)) \geqslant c$ and $\Vert \Xi\Vert_\infty, \Vert L\Vert_\infty \leqslant C$, Lemma \ref{Quantitative orthonormal basis parametrisation} tells us that $\Xi\Phi: \mathbb{R}^{h-m}\longrightarrow \mathbb{R}^d$ satisfies $\dist(\Xi\Phi, V_{\degen}(h-m,d))\gg_{c,C} 1$. (One may consult Definitions \ref{Definition degeneracy varieties} and Definition \ref{Definition dual pair degeneracy variety} for the definitions of $V_{\degen}(h-m,d)$ and $V_{\degen,2}^*(m,d,h)$). Thus, by Lemma \ref{Lemma different forms of non degen}, there exists some $s$ at most $d-2$ for which $\Xi\Phi$ has $\Omega_{c,C}(1)$-Cauchy-Schwarz complexity at most $s$. 

Writing $\Psi$ for $\Xi\Phi$, the proposition is proved. 
\end{proof}

\begin{Proposition}[Cauchy-Schwarz argument]
\label{Proposition Cauchy}
Let $s,d$ be natural numbers, with $d\geqslant 3$, and let $C$ be a positive constant. Let $\sigma_F$ be a parameter in the range $0 <\sigma_F<1/2$. Let $(\psi_1,\dots,\psi_d) = \Psi:\mathbb{R}^{s+1} \longrightarrow \mathbb{R}^d$ be a linear map, and suppose that $ \psi_1(\mathbf{e_k}) = 1$, for all the standard basis vectors $\mathbf{e_k} \in \mathbb{R}^{s+1}$. Suppose that, for all $j$ in the range $2\leqslant j\leqslant s+1$, there exists some $k$ such that $\psi_j(\mathbf{e_k}) = 0$. Let $N\geqslant 1$ be real, and let $g_1,\dots,g_d:[-N,N]\longrightarrow [-1,1]$ be arbitrary measurable functions, and, for each $j$ at most $d$, let $a_j$ be some real number with $\vert a_j\vert \leqslant CN$. Let $F:\mathbb{R}^{s+1} \longrightarrow [0,1]$ be any Lipschitz function, supported on $[-CN,CN]^{s+1}$ with Lipschitz constant $O(1/\sigma_F N)$. Suppose that $\Vert g_1 \Vert _{U^{s+1}(\mathbb{R},N)} \leqslant \rho$, for some parameter $\rho$ in the range $0<\rho\leqslant 1$. Then 
\begin{equation}
\label{equation abstract Gen von Neu}
\Big\vert\frac{1}{N^{s+1}} \int\limits_{\mathbf{w}\in \mathbb{R}^{s+1}} \prod\limits_{j=1}^d g_j (\psi_j(\mathbf{w}) + a_j)  F(\mathbf{w})\, d\mathbf{w}\Big\vert \ll_C \rho^{-\Omega(1)} \sigma_F^{-1}.
\end{equation}
\end{Proposition}
 \noindent We stress again that implied constants may depend on the implicit dimensions (so the $\Omega(1)$ term in (\ref{equation abstract Gen von Neu}) may depend on $s$). 
\begin{proof}
This theorem is very similar to the usual Generalised von Neumann Theorem (see \cite[Exercise 1.3.23]{Ta12}), and the proof is very similar too. A few extra technicalities arise from our dealing with the reals rather than with a finite group, but these are easily surmountable. 

We begin with some simple reductions. First, we assume that $C$ is large enough in terms of all other $O(1)$ parameters. For notational convenience, we will also allow $C$ to vary form line to line. Next, since $\psi_1(\mathbf{w}) = w_1+w_2+\cdots+w_{s+1}$, by shifting $w_1$ we can assume that $a_1 = 0$ in (\ref{equation abstract Gen von Neu}). Due to the restricted support of $F$, we may restrict the integral over $\mathbf{w}$ to $[-CN,CN]^{s+1}$. By Lemma \ref{Fourier transforms of Lipschitz functions}, for any $Y>2$ there is a function $\mathbf{c}_Y:\mathbb{R}^{s+1} \longrightarrow \mathbb{C}$ satisfying $\Vert \mathbf{c}\Vert_\infty \ll 1$ such that we may replace $F(\mathbf{w})$ by \[ \int\limits_{\substack{\mathbf{\boldsymbol{\theta}} \in \mathbb{R}^{s+1}\\ \Vert \boldsymbol{\theta} \Vert_\infty \leqslant Y}} c_Y(\boldsymbol{\theta})e(\frac{\boldsymbol{\theta}\cdot \mathbf{w}}{N}) \, d\boldsymbol{\theta} + O_C\left(\frac{\log Y}{\sigma_F Y}\right).\] We will determine a particularly suitable $Y$ later (which will depend on $\rho$). 

This means that \begin{align}
\label{after approximating Lipschitz}
&\Big\vert\frac{1}{N^{s+1}} \int\limits_{\mathbf{w}\in \mathbb{R}^{s+1}} \prod\limits_{j=1}^d g_j (\psi_j(\mathbf{w}) + a_j)  F(\mathbf{w})\, d\mathbf{w}\Big\vert \nonumber \\
& \ll \int\limits_{\substack{\boldsymbol{\theta} \in \mathbb{R}^{s+1}\\ \Vert \boldsymbol{\theta} \Vert_\infty \leqslant Y}} \Big\vert\frac{1}{N^{s+1}} \int\limits_{\mathbf{w} \in \mathbb{R}^{s+1}}^*e(\frac{\boldsymbol{\theta}}{N}\cdot \mathbf{w}) \Big( \prod\limits_{j=1}^d g_j (\psi_j(\mathbf{w}) + a_j)\Big) \, d\mathbf{w} \Big\vert \, d\boldsymbol{\theta} + O_C\left(\frac{\log Y}{\sigma_F Y}\right),
\end{align}
\noindent where $\int^*$ indicates the limits $\mathbf{w} \in [-CN,CN]^{s+1}$. Fix $\boldsymbol{\theta}$. The inner integral of (\ref{after approximating Lipschitz}) will be our primary focus. \\

Firstly, we wish to `absorb' the exponential phases $ e(\frac{\boldsymbol{\theta}}{N}\cdot \mathbf{w})$. To do this, we write $ e(\frac{\boldsymbol{\theta}}{N}\cdot \mathbf{w})$ as a product of functions $\prod_{k=1}^{s+1}b_k(\mathbf{w})$, where, for each $k$, the function $b_k:\mathbb{R}^{s+1}\longrightarrow \mathbb{C}$ is bounded in absolute value by $1$ and does not depend on the variable $w_k$. Since $s+1\geqslant 2$, this is possible. Now write \[ \prod\limits_{j=2}^d g_j(\psi_j(\mathbf{w}) + a_j) = \prod\limits_{k=1}^{s+1} b_k^\prime(\mathbf{w}),\] where each $b_k^\prime: \mathbb{R}^{s+1} \longrightarrow \mathbb{C}$ is bounded in absolute value by $1$ and does not depend on the variable $w_k$. This is possible since $\psi_1$ is the only function $\psi_j$ that includes all the variables $w_1,\dots,w_{s+1}$. 

Therefore we may rewrite the inner integral of (\ref{after approximating Lipschitz}) as 
\begin{equation}
\label{phases absorbed}
\frac{1}{N^{s+1}}\int\limits_{\mathbf{w} \in\mathbb{R}^{s+1}} ^*  g_1(\psi_1(\mathbf{w})) \prod\limits_{k=1}^{s+1} b_k^\prime(\mathbf{w}) b_k(\mathbf{w})\, d\mathbf{w}.
\end{equation}

A brief aside: readers familiar with the arguments of \cite[Appendix C]{GT10} (which motivate the present proof) may note that a different device is used in that paper to absorb the exponential phases. Those authors work in the setting of the finite group $\mathbb{Z}/N\mathbb{Z}$, and there the exponential phases can be absorbed simply by twisting the functions $g_j:\mathbb{Z}/N\mathbb{Z}\longrightarrow [-1,1]$ by a suitable linear phase function (witness the discussion surrounding expression (C.7) from \cite{GT10}). The key point there is that, if the linear form $\mathbf{w}\mapsto\boldsymbol{\theta}\cdot\mathbf{w}$ fails to be in the set $\spn(\psi_j:1\leqslant j\leqslant d)$, then a Fourier expansion of $g_j$ demonstrates that a certain expression, analogous to the inner integral of (\ref{after approximating Lipschitz}), is equal to zero. This clean argument isn't quite so easy to apply here, as the linear phases are not integrable over all of $\mathbb{R}$, which is why we choose a different approach. \\

Returning to (\ref{phases absorbed}), recall that $\psi_1(\mathbf{w}) = w_1+w_2+\cdots + w_{s+1}$. Therefore, applying the Cauchy-Schwarz inequality in each of the variables $w_1$ through $w_{s+1}$ in turn, one establishes that the absolute value of expression (\ref{phases absorbed}) is at most 
\begin{equation}
\label{The box norm style Gowers norm expression}
\ll_C \Big(\frac{1}{N^{2s+2}}\int\limits_{\mathbf{w} \in\mathbb{R}^{s+1}}^* \int\limits_{\mathbf{z} \in\mathbb{R}^{s+1}}^* \prod\limits_{\boldsymbol{\alpha}\in \{0,1\}^{s+1}}g_1\Big(\sum\limits_{\substack{k\leqslant s+1\\ \boldsymbol{\alpha}_k = 0}}w_k + \sum\limits_{\substack{k\leqslant s+1\\ \boldsymbol{\alpha}_k = 1}}z_k \Big) \, d\mathbf{w}\, d\mathbf{z}\Big)^{\frac{1}{2^{s+1}}}.
\end{equation}

This expression may be immediately related to the real Gowers norm as given in Definition \ref{Definition of Gowers norms over R}, by the change of variables  $m_k:=z_k - w_k$, for all $k$ at most $s+1$, and $u: = w_1+\cdots+w_{s+1}$. Performing this change of variables shows that(\ref{The box norm style Gowers norm expression}) is 
\begin{equation}
\label{reparametrise to make look like Gowers norms}
\ll  \Big(\frac{1}{N^{2s+2}}\int\limits_{(u,\mathbf{m},\mathbf{z_2^{s+1}})\in D}\prod\limits_{\boldsymbol{\alpha}\in \{0,1\}^{s+1}} g_1(u+\boldsymbol{\alpha}\cdot\mathbf{m}) \, du \, d\mathbf{m} \, d \mathbf{z_2^{s+1}} \Big)^{\frac{1}{2^{s+1}}},
\end{equation}
\noindent where $D$ is convex domain contained within $[-CN,CN]^{2s+2}$. It remains to replace $D$ by a Cartesian box.  \\

By Lemma \ref{Lipschitz approximation of convex cutoffs} we may write \[ 1_D = F_{\sigma} + O(G_{\sigma}),\] for any $\sigma$ in the range $0<\sigma<1/2$, where $F_{\sigma},G_{\sigma}:\mathbb{R}^{2s+2} \longrightarrow [0,1]$ are Lipschitz functions supported on $[-CN,CN]^{2s+2}$, with Lipschitz constant $O_C(1/\sigma N)$, such that $\int_{\mathbf{x}} G_\sigma (\mathbf{x}) \, d\mathbf{x} = O_C(\sigma N^{2s+2})$. Then, since $\Vert g_1\Vert_\infty \leqslant 1$, we may bound (\ref{reparametrise to make look like Gowers norms}) above by 
\begin{equation}
\label{final gowers norm tweak}
 \Big(\frac{1}{N^{2s+2}}\int\limits_{u,\mathbf{m},\mathbf{z_2^{s+1}}}^{\ast} F_{\sigma}(u,\mathbf{m},\mathbf{z_2^{s+1}})\prod\limits_{\boldsymbol{\alpha}\in \{0,1\}^{s+1}} g_1(u+\boldsymbol{\alpha}\cdot\mathbf{m})  du \, d\mathbf{m} \, d \mathbf{z_2^{s+1}} + O_C(\sigma)\Big)^{\frac{1}{2^{s+1}}} ,
 \end{equation} where $\int^*$ now refers to the domain of integration $[-CN,CN]^{2s+2}$. 

By applying Lemma \ref{Fourier transforms of Lipschitz functions} to $F_\sigma$, for any $X>2$ the absolute value of expression (\ref{final gowers norm tweak}) is 
\begin{align}
\label{final removing lipschitz}
\ll_C\Big(\Big(&\frac{1}{N^{2s+2}} \int\limits_{\substack{\boldsymbol{\xi} \in \mathbb{R}^{2s+2}\\ \Vert\boldsymbol{\xi}\Vert_\infty \leqslant X}} \Big\vert \int\limits_{u,\mathbf{m},\mathbf{z_2^{s+1}}}^{\ast} e(\frac{\boldsymbol{\xi}}{N} \cdot  (u,\mathbf{m},\mathbf{z_2^{s+1}})) \nonumber \\ &\prod\limits_{\boldsymbol{\alpha}\in \{0,1\}^{s+1}} g_1(u+\boldsymbol{\alpha}\cdot\mathbf{m})  du \, d\mathbf{m} \, d \mathbf{z_2^{s+1}} \Big\vert \, d\boldsymbol{\xi} \Big)  + O(\sigma)+ O\Big(\frac{\log X}{\sigma X}\Big) \Big)^{\frac{1}{2^{s+1}}}.
\end{align}

Integrating over the variables $z_2,\dots,z_{s+1}$, and splitting the exponential phase amongst the different functions, expression (\ref{final removing lipschitz}) is 
\begin{align}
\label{nearly there!}
\ll_C \Big(\Big(\frac{1}{N^{s+2}}\int\limits_{\substack{\boldsymbol{\xi} \in \mathbb{R}^{2s+2}\\ \Vert\boldsymbol{\xi}\Vert_\infty \leqslant X}}\Big\vert \int\limits_{(u,\mathbf{m}) \in [-CN,CN]^{s+2}} \prod\limits_{\ba \in \{0,1\}^{s+1}} g_{\ba}(u+\ba \cdot \mathbf{m}) \, du \, d\mathbf{m} \Big\vert \, d\boldsymbol{\xi}\Big) \nonumber \\ + O_C(\sigma)+ O_C\Big(\frac{\log X}{\sigma X}\Big) \Big)^{\frac{1}{2^{s+1}}},
\end{align} where each function $g_{\ba}$ is of the form \[g_{\ba}(u) := g_1(u) e( k_{\ba} u)\] for some real $k_{\ba}$. Note that $\Vert g_{\ba} \Vert_{U^{s+1}(\mathbb{R},N)} = \Vert g_1 \Vert_{U^{s+1}(\mathbb{R},N)}$.

Recall that $g_1$ is supported on $[-2N,2N]$. Therefore, if $\prod_{\ba \in \{0,1\}^{s+1}} g_{\ba}(u+\ba \cdot \mathbf{m})\neq 0$ then $(u,\mathbf{m}) \in [-O(N),O(N)]^{s+2}$. So, if $C$ is large enough in terms of $s$, we may replace the restriction $(u,\mathbf{m}) \in [-CN,CN]^{s+2}$ in (\ref{nearly there!}) with the condition $(u,\mathbf{m}) \in \mathbb{R}^{s+2}$, without changing the value of (\ref{nearly there!}). 

Then, by the Gowers-Cauchy-Schwarz inequality (Proposition \ref{Proposition Gowers Cauchy Schwarz}) and the triangle inequality, (\ref{nearly there!}) is 
\begin{align}
\label{nearly nearly there!}
&\ll_C (X^{O(1)} \Vert g_1 \Vert_{U^{s+1}(\mathbb{R})}^{2^{s+1}} +  \sigma + \frac{\log X}{\sigma X} )^{\frac{1}{2^{s+1}}} \nonumber \\
&\ll_C (X^{O(1)} \rho ^{2^{s+1}} +  \sigma +  \frac{\log X}{\sigma X} )^{\frac{1}{2^{s+1}}} 
\end{align} Choosing $X = \rho^{-c_1}$, with $c_1$ suitably small in terms of $s$, and $\sigma = \rho^{c_1/2}$, expression (\ref{nearly nearly there!}) is $O_C(\rho^{\Omega(1)})$. \\

Putting this estimate into (\ref{after approximating Lipschitz}), we get a bound on (\ref{after approximating Lipschitz}) of 
\begin{equation}
\label{nearly nearly nearly there!!!}
\ll_C Y^{O(1)} \rho^{\Omega(1)}  + O(\frac{\log Y}{\sigma_F Y}).
\end{equation}
\noindent Picking $Y = \rho^{-c_1}$, with $c_1$ suitably small in terms of $s$, we may ensure that (\ref{nearly nearly nearly there!!!}) is $O_C(\rho^{\Omega(1)} \sigma_F^{-1})$, thus proving the proposition. 
\end{proof}
With these propositions in hand, Theorem \ref{Theorem Generalised von Neumann Theorem over reals} follows quickly. 
\begin{proof}[Proof of Theorem \ref{Theorem Generalised von Neumann Theorem over reals}]
Assuming all the hypotheses of Theorem \ref{Theorem Generalised von Neumann Theorem over reals}, apply the result of Proposition \ref{Proposition separating out the kernel} to $\widetilde{T}_{F,G,N}^{L,\Xi,\widetilde{\mathbf{r}}}(g_1,\dots,g_d)$. Thus 
\begin{equation}
\label{applying the first proposition chapter 3}
\vert\widetilde{T}_{F,G,N}^{L,\Xi,\widetilde{\mathbf{r}}}(g_1,\dots,g_d)\vert\ll_{c,C,\varepsilon} \Big\vert\frac{1}{N^{h-m}} \int\limits_{\mathbf{x} \in\mathbb{R}^{h-m}}  F_1(\mathbf{x})\prod\limits_{j=1}^d g_j(\psi_j(\mathbf{x}) +a_j) \, d\mathbf{x}\Big\vert,
\end{equation} where $\Psi: \mathbb{R}^{h-m} \longrightarrow \mathbb{R}^d$ has $\Omega_{c,C}(1)$-Cauchy-Schwarz complexity at most $s$, for some $s$ at most $d-2$, $F_1:\mathbb{R}^{h-m} \longrightarrow [0,1]$ is a Lipschitz function supported on $[-O_{c,C,\varepsilon}(N),O_{c,C,\varepsilon}(N)]^{h-m}$ with Lipschitz constant $O(1/\sigma_F N)$, and $a_j = O_{c,C,\varepsilon}(1)$. Furthermore $\Vert \Psi\Vert_\infty = O_C(1)$.

We apply Proposition \ref{normal form algorithm} to $\Psi$. Therefore, for \emph{any} real numbers $w_1,\dots,w_{s+1}$, 
\begin{equation}
\label{Fixed dummey variables}
\vert\widetilde{T}_{F,G,N}^{L,\Xi,\widetilde{\mathbf{r}}}(g_1,\dots,g_d)\vert \ll \Big\vert \frac{1}{N^{h-m}}\int\limits_{\mathbf{x} \in \mathbb{R}^{h-m}}F_1(\mathbf{x} + \sum\limits_{k=1}^{s+1} w_k \mathbf{f_k})\prod\limits_{j=1}^d g_j(\psi_j^\prime(\mathbf{x},\mathbf{w}) + a_j) \, d\mathbf{x}\Big\vert,
\end{equation}
\noindent where 
\begin{itemize}
\item for each $j$ at most $d$, $\psi^\prime_j:\mathbb{R}^{h-m} \times \mathbb{R}^{s+1}\longrightarrow \mathbb{R}$ is a linear form;
\item $\psi_1^\prime(\mathbf{0},\mathbf{w}) = w_1+\cdots+w_{s+1}$;
\item $\mathbf{f_1},\dots,\mathbf{f_{s+1}} \in \mathbb{R}^{h-m}$ are some vectors that satisfy $\Vert\mathbf{f_k}\Vert_\infty = O_{c,C}(1)$ for each $k$ at most $s+1$;
\item the system of forms $(\psi^\prime_1,\dots,\psi^\prime_d)$ is in normal form with respect to $\psi_1^\prime$. 
\end{itemize}
\noindent We remark that the right-hand side of expression (\ref{Fixed dummey variables}) is independent of $\mathbf{w}$, as it was obtained by applying the change of variables $\mathbf{x} \mapsto \mathbf{x} + \sum_{k=1}^{s+1} w_k\mathbf{f_k} $ to expression (\ref{applying the first proposition chapter 3}).\\

Now, let $P:\mathbb{R}^{s+1}\longrightarrow [0,1]$ be some Lipschitz function, supported on $[-N,N]^{s+1}$, with Lipschitz constant $O(1/N)$. Also suppose that $ P(\mathbf{x}) \equiv 1$ if $\Vert \mathbf{x}\Vert_\infty \leqslant N/2$. Integrating over $\mathbf{w}$, we have that $\vert\widetilde{T}_{F,G,N}^{L,\Xi,\widetilde{\mathbf{r}}}(g_1,\dots,g_d)\vert$ is
\begin{align}
 &\ll_{c,C,\varepsilon}
\frac{1}{N^{h-m+s+1}} \int\limits_{\mathbf{w} \in \mathbb{R}^{s+1}} P(\mathbf{w}) \Big\vert\int\limits_{\mathbf{x}\in\mathbb{R}^{h-m}} F_1(\mathbf{x} + \sum\limits_{k=1}^{s+1} w_k \mathbf{f_k})\prod\limits_{j=1}^{d} g_j(\psi_j^\prime(\mathbf{x},\mathbf{w}) + a_j) \, d\mathbf{x}\Big\vert \, d\mathbf{w} \nonumber \\
\label{equation after itnegrating over w}
 &\ll_{c,C,\varepsilon}\Big\vert\frac{1}{N^{h-m+s+1}} \int\limits_{\substack{\mathbf{x} \in \mathbb{R}^{h-m}\\ \mathbf{w} \in \mathbb{R}^{s+1}}} H(\mathbf{x},\mathbf{w}) \prod\limits_{j=1}^{d} g_j(\psi_j^\prime(\mathbf{x},\mathbf{w}) + a_j) \, d\mathbf{x} \, d\mathbf{w}\Big\vert,
\end{align}
where the function $H:\mathbb{R}^{h-m+s+1} \longrightarrow [0,1]$ is defined by \[H(\mathbf{x},\mathbf{w}):= F_1(\mathbf{x} + \sum\limits_{k=1}^{s+1} w_k \mathbf{f_k}) P(\mathbf{w}). \] Since the vectors $\mathbf{f_k}$ satisfy $\Vert \mathbf{f_k}\Vert_\infty = O_{c,C}(1)$, $H$ is a Lipschitz function supported on $[-O_{c,C,\varepsilon} (N),O_{c,C,\varepsilon}(N)]^{h-m+s+1}$, with Lipschitz constant $O_{c,C}(1/\sigma_F N).$ Notice in (\ref{equation after itnegrating over w}) that we were able to move the absolute value signs outside the integral, as $P$ is positive and the integral over $\mathbf{x}$ is independent of $\mathbf{w}$ (so in particular has constant sign).

Fix $\mathbf{x}$. Then the integral over $\mathbf{w}$ in (\ref{equation after itnegrating over w}) satisfies the hypotheses of Proposition \ref{Proposition Cauchy}. Applying Proposition \ref{Proposition Cauchy} to this integral, and then integrating over $\mathbf{x}$, one derives \[\vert \widetilde{T}_{F,G,N}^{L,\Xi,\widetilde{\mathbf{r}}}(g_1,\dots,g_d)\vert \ll_{c,C,\varepsilon} \rho^{\Omega(1)} \sigma_F^{-1}. \] Theorem \ref{Theorem Generalised von Neumann Theorem over reals} is proved. 
\end{proof}

By our long series of reductions, this means that both Theorem \ref{Theorem rational set out version} and Theorem \ref{Main Theorem chapter 3} are proved. \qed \\

\section{Constructions}
\label{Constructions}
In this section we prove Theorem \ref{Converse to main theorem}, which, we remind the reader, is the partial converse of Theorem \ref{Main Theorem chapter 3}. In other words, we show that $L$ being bounded away from $V_{\degen}^*(m,d)$ is a necessary hypothesis for Theorem \ref{Main Theorem chapter 3} to be true.

\begin{proof}[Proof of Theorem \ref{Converse to main theorem}]
Recall the hypotheses of Theorem \ref{Converse to main theorem}. In particular, we suppose that $$\liminf\limits_{N\rightarrow\infty}\operatorname{dist}(L,V^*_{\degen}(m,d)) = 0,$$ i.e. we assume that $\operatorname{dist}(L,V_{\degen}^*(m,d)) = \omega(N)^{-1}$, for some function $\omega(N)$ such that $$\limsup_{N\rightarrow\infty} \omega(N) = \infty.$$  Let $\eta$ be a small positive quantity, picked small enough in terms of $c$ and $C$, and let $N$ be a natural number that is large enough so that $\omega(N)\geqslant \eta^{-1}$ and $\eta N\geqslant \operatorname{max}(1,\varepsilon)$. All implied constants to follow will be independent of $\eta$. 

Since $F$ is the indicator function of $[1,N]^d$ and $G$ is the indicator function of $[-\varepsilon,\varepsilon]^m$, one has \[T_{F,G,N}^L(f_1,\dots,f_d) = \frac{1}{N^{d-m}} \sum\limits_{\substack{\mathbf{n} \in [N]^d \\ \Vert L\mathbf{n} \Vert_\infty \leqslant \varepsilon}} \prod\limits_{j=1}^d f_j(n_j).\] Our aim is to construct functions $f_1,\dots,f_d:[N] \longrightarrow [-1,1]$ such that \[ \min_j\Vert f_j\Vert_{U^{s+1}[N]} \leqslant \rho\] for some $\rho$ at most $1$ and that \begin{equation}
\label{conclusion repeated}
T_{F,G,N}^L(f_1,\dots,f_d) > H(\rho) + E_\rho(N).
\end{equation} 

We begin by observing that the condition $\Vert L\mathbf{n}\Vert_\infty\leqslant \varepsilon$ implies certain constraints on two of the variables $n_i$. Indeed, let $L^\prime\in V_{\degen}^*(m,d)$ be such that $\Vert L - L^\prime \Vert_\infty = \operatorname{dist}(L,V_{\degen}^*(m,d))$. Write $\lambda_{ij}$ for the coefficients of $L$ and $\lambda_{ij}^\prime$ for the coefficients of $L^\prime$. By reordering columns, without loss of generality we may assume that there exist real numbers $\{a_i\}_{i=1}^m$ not all 0 s.t. for all $j$ in the range $3\leqslant j\leqslant d$ we have  
\begin{equation}
\label{all but 2 columns degenerate}
\sum\limits_{i=1}^m a_i \lambda^{\prime}_{ij} = 0,
\end{equation} and further we may assume that for all $i$ we have $\lambda^\prime_{i1} = \lambda_{i1}$ and $\lambda^\prime_{i2} = \lambda_{i2}$ (else $L^\prime\in V_{\degen}^*(m,d)$ is not one of the closest matrices to $L$). By reordering rows and rescaling, we may assume that $a_1$ has maximal absolute value amongst all the $a_i$, and that $\vert a_1\vert = 1$. 

Define $$b_1:=\sum\limits_{i=1}^m a_i \lambda_{i1}, \quad b_2:= \sum\limits_{i=1}^m a_i \lambda_{i2},$$ and let $\mathbf{n}\in [N]^d$ be some solution to $\Vert L\mathbf{n}\Vert_\infty \leqslant \varepsilon$. The critical observation is that (\ref{all but 2 columns degenerate}), combined with the assumptions on the $a_i$, implies that 
\begin{equation}
\label{the critical observation for construction}
\vert b_1 n_1 + b_2n_2\vert \ll \eta N.
\end{equation}  

Indeed, for $j$ in the range $3\leqslant j\leqslant d$ we have 
\begin{align*}
\left\vert\sum\limits_{i=1}^m a_i\lambda_{ij}\right\vert & = \left\vert\sum\limits_{i=1}^m a_i(\lambda_{ij} - \lambda^\prime_{ij})\right\vert \\
&\ll \eta.
\end{align*} 

\noindent Since $\Vert L\mathbf{n}\Vert_\infty \leqslant \varepsilon$, we certainly have that $$\left\vert b_1 n_1 + b_2n_2 + \sum\limits_{j=3}^d n_j\sum\limits_{i=1}^m a_i \lambda_{ij}\right\vert \ll \varepsilon,$$ and then (\ref{the critical observation for construction}) follows by the triangle inequality and the fact that $\eta N\geqslant \varepsilon$.\\

The constraint (\ref{the critical observation for construction}) will turn out to be enough for the proof. We consider various cases, constructing different counterexample functions $f_1$ and $f_2$ based on the size and sign of $b_1$ and $b_2$. To facilitate this, we let $c_1$ be a suitably small positive constant, depending on $c$ and $C$, but independent of $\eta$. All constants $C_1$ and $C_2$ to follow will be assumed to satisfy $C_1,C_2 = O_{c,C}(1)$. \\\\

\noindent \textbf{Case 1: $\vert b_1\vert,\vert b_2\vert \leqslant c_1$.}

Under the assumptions of Theorem \ref{Converse to main theorem}, this case is actually precluded. Indeed, consider the matrix $L^{\prime\prime}$, defined by taking $$\lambda_{ij}^{\prime\prime} = \lambda_{ij}^\prime$$ for all pairs $(i.j)\in [m]\times[d]$, except for $(1,1)$ and $(1,2)$. In these cases we let 
\begin{align*}
\lambda_{11}^{\prime\prime} &= \lambda_{11}^\prime - \frac{b_1}{a_1}\\
\lambda_{12}^{\prime\prime} & = \lambda_{12}^\prime - \frac{b_2}{a_1}.
\end{align*}
Then $$\sum\limits_{i=1}^m a_i\lambda_{ij}^{\prime\prime} = 0$$ for all $j$ in the range $1\leqslant j\leqslant d$. In other words we have shown that $\Vert L-L^{\prime\prime}\Vert_{\infty} \leqslant \eta + c_1$ for some matrix $L^{\prime\prime}$ with rank less than $m$. Since $\eta + c_1< c$ (if $c_1$ is small enough), this implies that $\operatorname{dist}(L,V_{\rank}(m,d))<c$, which contradicts the assumptions of Theorem \ref{Converse to main theorem}. Therefore this case is indeed precluded. \\

\noindent \textbf{Case 2: $b_1,b_2$ both of the same sign, and $b_1,b_2\geqslant c_1$.}

In this case, (\ref{the critical observation for construction}) implies\footnote{The same conclusion is true for $n_2$, but this will not be needed.} that $n_1\leqslant C_1\eta N$ for some constant $C_1$. Now, define $f_1:[N]\longrightarrow [-1,1]$ to be the indicator function of the interval $[\lceil C_1\eta N\rceil, N]\cap \mathbb{N}$. We then have 
\begin{align*}
\Vert f_1 - 1\Vert_{U^{s+1}[N]} &\ll \Big(\frac{1}{N^{s+2}} \sum\limits_{x,h_1,\dots,h_{s+1}\ll C_1\eta N} 1\Big)^{\frac{1}{2^{s+1}}}\\
&\leqslant C_2(C_1\eta)^{\frac{s+2}{2^{s+1}}}
\end{align*}
\noindent for some constant $C_2$. However, observe that 
\begin{align*}
\vert T_{F,G,N}^L(f_1-1,1,\dots,1)\vert &= \vert T_{F,G,N}^L(f_1,1,\dots,1) - T_{F,G,N}^L(1,1,\dots,1)\vert\\
& = \vert 0 - T_{F,G,N}^L(1,1,\dots,1)\vert
\gg_{c,C,\varepsilon} 1
\end{align*}
\noindent by the hypotheses of Theorem \ref{Converse to main theorem}. If $T_{F,G,N}^L(f_1 - 1,1,\dots,1)$ did not satisfy (\ref{conclusion repeated}), then $$1 \ll_{c,C,\varepsilon} H(\rho) + E_{\rho}(1),$$ where $\rho := C_2(C_1\eta)^{\frac{s+2}{2^{s+1}}}$. Picking $\eta$ small enough, then $N$ large enough, this inequality cannot possibly hold, and we have a contradiction. So $T_{F,G,N}^L(f_1 - 1,1,\dots,1)$ satisfies (\ref{conclusion repeated}). \\

\noindent \textbf{Case 3: $b_1,b_2$ of opposite signs, and $b_1,b_2\geqslant c_1$.}

This is the most involved case, although the central idea is very simple. The condition (\ref{the critical observation for construction}) confines $n_2$ to lie within a certain distance of a fixed multiple of $n_1$. By constructing functions $f_1$ and $f_2$ using random choices of blocks of this length, but coupled in such a way that condition (\ref{the critical observation for construction}) is very likely to hold, we can guarantee that $T_{F,G,N}^L(f_1-p,f_2-p,1,\dots,1)$ is bounded away from zero, where $p$ is the probability used to choose the random blocks. However, despite the block construction and the coupling, the functions $f_1$ and $f_2$ still individually exhibit enough randomness to conclude that $\Vert f_1 - p\Vert _{U^{s+1}[N]} = o(1)$ as $N\rightarrow \infty$, and the same for $f_2$.  \\

We now fill in the technical details. Relation (\ref{the critical observation for construction}) implies that 
\begin{equation}
\label{the critical equation for construction with an explicit constant}
 \vert b_1n_1 + b_2n_2\vert \leqslant C_1\eta N,
\end{equation}
\noindent for some $C_1$ satisfying $C_1 = O(1)$, and without loss of generality assume that $b_1$ is positive, $b_2$ is negative, and $\vert b_1\vert$ is at least $\vert b_2\vert$. Let $C_2$ be some parameter, chosen so that $(C_1C_2\eta)^{-1}$ is an integer. Such a $C_2$ will of course depend on $\eta$, but in magnitude we may pick $C_2\asymp 1$. We consider the real interval $[0,N]$ modulo $N$, and for $x\in [0,N]$ and $i$ in the range $0\leqslant i\leqslant (C_1C_2\eta)^{-1} - 1$ we define the half-open interval modulo $N$ $$I_{x,i}: = [x+ i C_1C_2\eta N,x+(i+1)C_1C_2\eta N).$$  This choice guarantees that 
\begin{equation}
\label{the defining partition}
[0,N] = \bigcup\limits_{i=0}^{(C_1C_2\eta)^{-1} - 1} I_{x,i},
\end{equation}
 and the union is disjoint. Now, for $\delta$ a small constant to be chosen later\footnote{This $\delta$ is unrelated to the notation $\delta = T_{F,G,N}^L(f_1,\dots,f_d)$ used in previous sections.}, we define $$I^\delta_{x,i}: = [x+ (i+\frac{1}{2}-\delta) C_1C_2\eta N,x+(i+\frac{1}{2}+\delta)C_1C_2\eta N).$$ 
 
We will use the partition (\ref{the defining partition}) to construct a function $f_1$, using an averaging argument to choose an $x$ so that the $I_{x,i}^\delta$ intervals capture a positive proportion of the solution density of the linear inequality system. Indeed, for $n_1\in [N]$ let the weight $u(n_1)$ denote the number of $d-1$-tuples $n_2,\dots,n_d\leqslant N$ that together with $n_1$ satisfy the inequality $\Vert L\mathbf{n}\Vert_\infty < \varepsilon$. The weight $u(n_1)$ could be zero, of course. Let \[E_{x,\delta} := \cup_i I_{x,i}^\delta.\] Then
\begin{align*}
\frac{1}{N}\int\limits_{0}^N \sum\limits_{n\in [N]} u(n) 1_{E_{x,\delta}}(n) \, dx & = \frac{1}{N}\sum\limits_{n\in [N]} u(n)\int\limits_{0}^N 1_{E_{x,\delta}}(n)\, dx  \\
& = \sum\limits_{n\in [N]} u(n) 2\delta \\
& = 2\delta N^{d-m}T_{F,G,N}^L(1,\dots,1)
\end{align*}
\noindent Therefore, by the assumptions of Theorem \ref{Converse to main theorem}, we may fix an $x$ such that 
\begin{equation}
\label{capturing solutions in intervals}
\sum\limits_{n\in [N]} u(n) 1_{E_{x,\delta}}(n) \gg_{c,C} \delta N^{d-m}T_{F,G,N}^L(1,\dots,1).
\end{equation} 

Let us finally define the function $f_1$. Let $p$ be a small positive constant (to be decided later). Fix a value of $x$ such that (\ref{capturing solutions in intervals}) holds. Then we define a random subset $A\subseteq[N]$ by picking all of $I_{x,i} \cap \mathbb{N}$ to be members of $A$, with probability $p$, or none of $I_{x,i} \cap \mathbb{N}$ to be members of $A$, with probability $1-p$. We then make this same choice for each $i$ in the range $0\leqslant i \leqslant (C_1C_2\eta)^{-1} - 1$, independently. Observe immediately that for each $n\in [N]$ the probability that $n\in A$ is always $p$ (though these events are not always independent). We let $f_1(n)$ be the indicator function $1_A(n)$. 

The function $f_2$ is defined in terms of $f_1$. Indeed, let $$J_{x,i} = \frac{b_1}{\vert b_2\vert} I_{x,i}\cap (0,N],$$ where the dilation of the interval $I_{x,i}$ is not considered modulo $N$ but rather just as an operator on subsets of $\mathbb{R}$ (see Section \ref{Section Notation} for this notation). Since $b_1\geqslant \vert b_2\vert$ we have that these $J_{x,i}$ also form a disjoint partition of $[0,N]$. [NB: If $b_1>\vert b_2\vert$ it may be that certain $J_{x,i}$ are empty, since the dilate of the corresponding $I_{x,i}$ may land entirely outside $[0,N]$.] Then let $B$ be the subset of $[N]$ defined so that for each $i$ with $J_{x,i}$ non-empty we have $J_{x,i}\cap \mathbb{N}\subseteq B$ if and only if $I_{x,i}\cap \mathbb{N}\subseteq A$. Note again that for each individual $n\in [N]$ the probability that $n\in B$ is always $p$. We let $f_2(n)$ be the indicator function $1_B(n)$. \\

Our first claim is that, if $p$ is small enough in terms of $\delta$, 
\begin{equation}
\label{lower bound on solution count of coupled random sets}
\vert \mathbb{E} T_{F,G,N}^L(f_1,f_2,1,\dots,1)-T_{F,G,N}^L(p,p,1,\dots,1)\vert \gg_{c,C,\varepsilon} \delta^2.
\end{equation} Indeed, suppose that $I_{x,i}$ is included in the set $A$, and suppose that $n_1\in I_{x,i}^\delta$. If $n_2\in [N]$ satisfies $\vert \frac{b_1}{\vert b_2\vert} n_1 - n_2\vert \leqslant \frac{1}{b_2}C_1\eta N$ and if $\delta$ is small enough in terms of $b_1$ and $b_2$, then\footnote{This fact is the reason why we introduced the parameter $\delta$.} $n_2 \in J_{x,i}$. Thus, by the observation (\ref{the critical equation for construction with an explicit constant}), $n_2 \in B$, for every integer $n_2$ that is the second coordinate of a solution vector\footnote{i.e a vector $\mathbf{n}$ such that $\Vert L\mathbf{n}\Vert_\infty \leqslant \varepsilon$.} $\mathbf{n}$ for which the first coordinate is $n_1$. Therefore
\begin{align*}
\mathbb{E}T_{F,G,N}^L(f_1,f_2,1,\dots,1)& = \frac{1}{N^{d-m}}\sum\limits_{\substack{\mathbf{n}\in [N]^d \\ \Vert L\mathbf{n}\Vert_{\infty} \leqslant\varepsilon}} \mathbb{P}(n_1\in A\wedge n_2\in B)\\
&\geqslant \frac{1}{N^{d-m}}\sum\limits_{\substack{\mathbf{n}\in [N]^d \\ \Vert L\mathbf{n}\Vert_{\infty} \leqslant \varepsilon}} \mathbb{P}(n_1\in A \wedge n_1 \in I_{x,i}^\delta \text{ for some } i \wedge n_2\in B)\\
&\geqslant \frac{1}{N^{d-m}}\sum\limits_{\substack{\mathbf{n}\in [N]^d \\ \Vert L\mathbf{n}\Vert_{\infty} \leqslant \varepsilon}} \mathbb{P}(n_1\in A \wedge n_1 \in I_{x,i}^\delta \text{ for some } i)\\
& = \frac{1}{N^{d-m}}\sum\limits_{n_1\in [N] } u(n_1)p1_{E_{x,\delta}}
(n_1)\\
&\geqslant 2\delta p  T_{F,G,N}^L(1,\dots,1),
\end{align*}
\noindent where the final line follows from (\ref{capturing solutions in intervals}). On the other hand $T_{F,G,N}^L(p,p,1,\dots,1) = p^2T_{F,G,N}^L(1,\dots,1)$, and hence 
\begin{equation}
\mathbb{E} T_{F,G,N}^L(f_1,f_2,1,\dots,1)-T_{F,G,N}^L(p,p,1,\dots,1)\geqslant (2\delta p - p^2)T_{F,G,N}^L(1,\dots,1).
\end{equation}
Picking $p$ small enough in terms of $\delta$, and using the assumption that $T_{F,G,N}^L(1,\dots,1)  = \Omega_{c,C,\varepsilon}(1)$, this proves the relation (\ref{lower bound on solution count of coupled random sets}). \\

Our second claim is that 
\begin{equation}
\label{equatoin bound on expectation of gowers norms}
\mathbb{E}\Vert f_1 - p\Vert_{U^{s+1}[N]},\mathbb{E}\Vert f_2 - p\Vert_{U^{s+1}[N]} \ll \eta^{\frac{1}{2^{s+1}}}.
\end{equation} We first consider $f_1$. Then $$\mathbb{E}\Vert f_1 - p\Vert_{U^{s+1}[N]}^{2^{s+1}} \ll \frac{1}{N^{s+2}}\sum\limits_{\substack(x,\mathbf{h})\in \mathbb{Z}^{s+2}}\\ \mathbb{E}\Big(\prod\limits_{\boldsymbol{\omega} \in \{0,1\}^{s+1}} (f_1 - p1_{[N]})(x+\mathbf{h}\cdot\boldsymbol{\omega})\Big).$$ Observe that for fixed $(x,\mathbf{h})$ the random variables $(f_1 - p1_{[N]})(x+\mathbf{h}\cdot\boldsymbol{\omega})$ each have mean zero and, unless some two of the expressions $x+\mathbf{h}\cdot\boldsymbol{\omega}$ lie in the same block $I_i$, these random variables are independent. Hence, apart from those exceptional cases, we may factor the expectation and conclude that $$\mathbb{E}\Big(\prod\limits_{\boldsymbol{\omega} \in \{0,1\}^{s+1}} (f_1 - p1_{[N]})(x+\mathbf{h}\cdot\boldsymbol{\omega})\Big) = \prod\limits_{\boldsymbol{\omega} \in \{0,1\}^{s+1}} \mathbb{E}((f_1 - p1_{[N]})(x+\mathbf{h}\cdot\boldsymbol{\omega})) = 0.$$ Therefore,
\begin{align*}
\mathbb{E}\Vert f_1 - p\Vert_{U^{s+1}[N]}^{2^{s+1}}&\ll\frac{1}{N^{s+2}} \sum\limits_{(x,\mathbf{h})\in [-N,N]^{s+2}} 1_R(\mathbf{h})\\
&\ll \eta,
\end{align*}
\noindent where \[R = \{\mathbf{h}:\vert\mathbf{h}\cdot(\boldsymbol{\omega_1} - \boldsymbol{\omega_2})\vert \leqslant C_1C_2\eta N \text{ for some }\boldsymbol{\omega_1},\boldsymbol{\omega_2}\in \{0,1\}^{s+1},\, \boldsymbol{\omega_1}\neq\boldsymbol{\omega_2}\}.\]
\noindent Thus by Jensen's inequality we have 
\begin{equation}
\mathbb{E}\Vert f_1 - p\Vert_{U^{s+1}[N]} \ll \eta^{\frac{1}{2^{s+1}}},
\end{equation}
\noindent as claimed in (\ref{equatoin bound on expectation of gowers norms}).

The calculation for $f_2$ is essentially identical, noting that the length of the blocks $J_{x,i}$ is also $O(\eta N)$. \\

It is possible that one could finish the argument here by considering a second moment, and choosing some explicit $f_1$ and $f_2$. To avoid calculating a second moment, we argue as follows. Suppose for contradiction that there were no functions $f_1,\dots,f_d$ that satisfied (\ref{conclusion repeated}). Then, by (\ref{lower bound on solution count of coupled random sets}), if we pick $p$ to be small enough in terms of $\delta$ we have
\begin{align}
\label{block construction contradiction}
\delta^2 &\ll_{c,C,\varepsilon} \vert \mathbb{E}T_{F,G,N}^L(f_1,f_2,1,\dots,1) - T_{F,G,N}^L(p,p,1,\dots,1)\vert \nonumber\\
&\ll \vert \mathbb{E}T_{F,G,N}^L(f_1-p,f_2,1,\dots,1)\vert + \vert \mathbb{E}T_{F,G,N}^L(p,f_2-p,1,\dots,1)\vert \nonumber\\
&\ll \mathbb{E}(H(\rho_1) + E_{\rho_1}(N)) +\mathbb{E}(H(\rho_2) + E_{\rho_2}(N)),
\end{align}
\noindent where $\rho_1$ (resp. $\rho_2$) is any chosen upper-bound on $\Vert f_1 - p\Vert_{U^{s+1}[N]}$ (resp. $\Vert f_2 - p\Vert_{U^{s+1}[N]}$). Note that the values $\rho_i$ may be random variables themselves. 

We claim that the random variables $\rho_1$ and $\rho_2$ may be chosen so that the right-hand side of (\ref{block construction contradiction}) is $\kappa(\eta)+o_{\eta}(1)$ as $N\rightarrow \infty$. To prove this, we make two observations. Note first that by Markov's inequality $$ \mathbb{P}(\Vert f_1 - p\Vert_{U^{s+1}[N]}\geqslant \eta^{\frac{1}{2^{s+2}}}) \ll \eta^{\frac{1}{2^{s+2}}}$$ We choose the (random) upper-bound $\rho_1$ satisfying $$\rho_1 =\begin{cases}
1 & \text{ if } \Vert f_1 - p\Vert_{U^{s+1}[N]}\geqslant \eta^{\frac{1}{2^{s+2}}} \\
\eta^{\frac{1}{2^{s+2}}}  & \text{ otherwise }.
\end{cases}$$ Secondly, we may upper-bound $H$ by a concave envelope, so without loss of generality we may assume that $H$ is concave. 

Then by Jensen's inequality, 
\begin{align}
\label{equation applying Jensen to F}
\mathbb{E}(H(\rho_1) + E_{\rho_1}(N)) &\ll H(\mathbb{E}\rho_1) + \mathbb{E}(E_{\rho_1}(1))\nonumber\\
&\ll \kappa(\eta^{\frac{1}{2^{s+2}}} )  + o_{\eta}(1)\nonumber\\
&\ll \kappa(\eta)+o_{\eta}(1).
\end{align} 
\noindent We do the same manipulation for $f_2$. Combining (\ref{equation applying Jensen to F}) with (\ref{block construction contradiction}) we conclude that 
\begin{equation}
\label{equation final final final contradiction}
\delta^2 \ll_{c,C,\varepsilon} \kappa(\eta) + o_{\eta}(1). 
\end{equation}

The only condition on $\delta$ occurred in the proof of (\ref{lower bound on solution count of coupled random sets}), in which we assumed that $\delta$ was small enough in terms of $b_1$ and $b_2$. Therefore there exists a suitable $\delta$ that satisfies $\delta = \Omega_{c,C}(1)$. Picking such a $\delta$, and then picking $\eta$ small enough and $N$ large enough, (\ref{equation final final final contradiction}) is a contradiction. So there must be some functions $f_1,\dots,f_d$ that satisfy (\ref{conclusion repeated}). \\

\noindent \textbf{Case 4: Exactly one of $b_1,b_2$ satisfies $b_i\geqslant c_1$.}

Without loss of generality we may assume that $b_1\geqslant c_1$. But then, as in Case 2, (\ref{the critical observation for construction}) implies that $n_1\leqslant C_1\eta N$ for some constant $C_1$. The same construction as in Case 2 then applies.\\

We have covered all cases, and thus have concluded the proof of Theorem \ref{Converse to main theorem}. 
\end{proof}
\newpage
\appendix

\section{Gowers norms}
\label{sec.Gowers norms}

There are several existing accounts of the basic theory of Gowers norms -- for example in \cite{Gr07} and \cite{Ta12} -- and the reader looking for an introduction to the theory in its full generality should certainly consult these references, as well as Appendices B and C of \cite{GT10}. However, in the interests of making this paper as self-contained as possible, we use this section to pick out the central definitions and notions that are used in the main text. 

\begin{Definition}
\label{Definition of Gowers norms over cyclic groups}
Let $N$ be a natural number. For a function $f:\mathbb{Z}/N\mathbb{Z}\longrightarrow \mathbb{C}$, and a natural number $d$, we define the Gowers $U^{d}$ norm $\Vert f\Vert_{U^{d}(N)}$ to be the unique non-negative solution to 
\begin{equation}
\label{Definition of Gowers norms}
\Vert f\Vert_{U^{d}(N)}^{2^d} = \frac{1}{N^{d+1}}\sum\limits_{x,h_1,\dots,h_d}\prod\limits_{\boldsymbol{\omega}\in \{0,1\}^d}\mathcal{C}^{\vert \boldsymbol{\omega} \vert} f(x+\mathbf{h}\cdot\boldsymbol{\omega}),
\end{equation}
\noindent where $\vert \boldsymbol{\omega}\vert = \sum\limits_i \omega_i$, $\mathbf{h} = (h_1,\dots,h_d)$, $\mathcal{C}$ is the complex-conjugation operator, and the summation is over $x,h_1,\dots,h_d\in \mathbb{Z}/N\mathbb{Z}$.
\end{Definition}
\noindent For example, $$\Vert f\Vert_{U^1(N)} = \Big\vert\frac{1}{N}\sum\limits_{x} f(x)\Big\vert,$$ and $$\Vert f\Vert_{U^2(N)} = \left(\frac{1}{N^3}\sum\limits_{x,h_1,h_2}  f(x)\overline{f(x+h_1)f(x+h_2)}f(x+h_1+h_2)\right)^{\frac{1}{4}}.$$ It is not immediately obvious that the right-hand side of  (\ref{Definition of Gowers norms}) is always a non-negative real, nor why the $U^d$ norms are genuine norms if $d\geqslant 2$: proofs of both these facts may be found in \cite{TaVu06}. An immediate Cauchy-Schwarz argument, which may also be found in \cite{TaVu06}, gives the so-called `nesting property' of Gowers norms, namely the fact that \[ \Vert f\Vert_{U^2(N)} \leqslant \Vert f\Vert_{U^3(N)} \leqslant \Vert f\Vert_{U^4(N)} \leqslant \cdots. \]

The functions in the main text do not have a cyclic group as a domain but rather the interval $[N]$, but the theory may easily be adapted to this case. 

\begin{Definition}
\label{Definition of gowers norms over integers}
Let $N,N^\prime$ be natural numbers, with $N^\prime\geqslant N$. Identify $[N]$ with a subset of $\mathbb{Z}/N^\prime\mathbb{Z}$ in the natural way, i.e. $[N] = \{1,\dots,N\} \subseteq \{1,\dots,N^\prime\}$, which we then view as $\mathbb{Z}/N^\prime\mathbb{Z}$. For a function $f:[N]\longrightarrow \mathbb{C}$, and a natural number $d$, we define the Gowers norm $\Vert f\Vert_{U^d[N]}$ to be the unique non-negative real solution to the equation
\begin{equation}
\label{Definintion of gowers norms over integers the actual equation}
\Vert f\Vert_{U^d[N]} ^{2^d} = \frac{1}{\vert R\vert}\sum\limits_{x,h_1,\dots,h_d}\prod\limits_{\boldsymbol{\omega}\in \{0,1\}^d}\mathcal{C}^{\vert \boldsymbol{\omega} \vert} f1_{[N]}(x+\mathbf{h}\cdot\boldsymbol{\omega}),
\end{equation}
\noindent where $f1_{[N]}$ is the extension by zero of $f$ to $\mathbb{Z}/N^\prime\mathbb{Z}$, the summation is over\\ $x,h_1,\dots,h_d\in \mathbb{Z}/N^\prime\mathbb{Z}$, and the set $R$ is the set $$R:=\{x,h_1,\dots,h_d\in \mathbb{Z}/N^\prime\mathbb{Z}: \text{ for every } \boldsymbol{\omega}\in \{0,1\}^d, x+\mathbf{h}\cdot\boldsymbol{\omega}\in [N]\}.$$
\end{Definition}

One can immediately see that this definition is equivalent to $$\Vert f\Vert_{U^d[N]} = \Vert f1_{[N]}\Vert_{U^d(N^\prime)}/\Vert 1_{[N]}\Vert_{U^d(N^\prime)},$$ and is also independent of the choice of $N^\prime$ as long as $N^\prime/N$ is large enough (in terms of $d$). Taking $N^\prime = O(N)$ we have $\Vert 1_{[N]}\Vert_{U^d(N^\prime)}\asymp 1$, and thus $\Vert f\Vert_{U^d[N]}\asymp \Vert f1_{[N]}\Vert_{U^d(N^\prime)}$. (See \cite[Lemma B.5]{GT10} for more detail on this).

We observe that there is only a contribution to the summand in equation (\ref{Definintion of gowers norms over integers the actual equation}) when $x \in [N]$ and for every $i$ we have $ h_i \in \{-N,-N+1,\dots, N-1,N\} \text{ modulo } N^\prime$. Further, it may be easily seen that $\vert R\vert \asymp N^{d+1}$. Therefore, choosing $N^\prime/N$ sufficiently large, we conclude that 
\begin{equation}
\label{comparing GN over ints with obvious sum}
\Vert f\Vert_{U^d[N]} \asymp \left( \frac{1}{N^{d+1}}\sum\limits_{x,h_1,\dots,h_d\in \mathbb{Z}}\prod\limits_{\boldsymbol{\omega}\in \{0,1\}^d}\mathcal{C}^{\vert \boldsymbol{\omega} \vert} f(x+\mathbf{h}\cdot\boldsymbol{\omega})\right)^{\frac{1}{2^d}}.
\end{equation}
\noindent The relation (\ref{comparing GN over ints with obvious sum}) is implicitly assumed throughout the main text.\\

In order to succinctly state Theorem \ref{Theorem Generalised von Neumann Theorem over reals}, we had to refer to a Gowers norm $U^d(\mathbb{R})$, which has been used in some recent work on linear patterns in subsets of Euclidean space (see \cite[Lemma 4.2]{CoMa17}, \cite[Proposition 3.3]{DVR17}). This Gowers norm is a less well-studied object, as the theory was originally developed over finite groups. Nevertheless it may be perfectly well defined, and even deep aspects of its inverse theory may be deduced from the corresponding theory of the discrete Gowers norm (see \cite{Ta15}). 
\begin{Definition}
\label{Definition of Gowers norms over R}
Let $f:[0,1]\longrightarrow {\mathbb{C}}$ be a bounded measurable function, and let $d$ be a natural number. Then we define the Gowers norm $\Vert f \Vert_{U^{d}(\mathbb{R})}$ to be the unique non-negative real satisfying 

\begin{equation}
\Vert f \Vert_{U^{d}(\mathbb{R})}^{2^{d}} = \int\limits_{(x,\mathbf{h}) \in\mathbb{R}^{d+1}} \prod\limits_{\boldsymbol{\omega}\in \{0,1\}^d}\mathcal{C}^{\vert \boldsymbol{\omega} \vert} f(x+\sum\limits_{i=1}^d h_i\omega_i) \, dx\,dh_1\cdots dh_d
\end{equation} 

\noindent where $\vert \boldsymbol{\omega}\vert = \sum\limits_i \omega_i$, and $\mathcal{C}$ is the complex-conjugation operator.
\end{Definition}

Let $N$ be a positive real, and let $g:[-N,N]\longrightarrow \mathbb{C}$ be a measurable function. Define the function $f:[0,1]\longrightarrow \mathbb{C}$ by $f(x): = g(2Nx - N)$, and then set $$\Vert g\Vert_{U^d(\mathbb{R},N)}: = \Vert f\Vert_{U^d(\mathbb{R})}.$$ Explicitly, a change of variables shows that \begin{equation}
\label{explicit real gowers norms}
\Vert g \Vert_{U^{d}(\mathbb{R},N)}^{2^{d}} \asymp \frac{1}{N^{d+1}}\int\limits_{(x,\mathbf{h}) \in\mathbb{R}^{d+1}} \prod\limits_{\boldsymbol{\omega}\in \{0,1\}^d}\mathcal{C}^{\vert \boldsymbol{\omega} \vert} g(x+\sum\limits_{i=1}^d h_i\omega_i) \, dx\,dh_1\cdots dh_d.
\end{equation}

We require one further fact about Gowers norms. 

\begin{Proposition}[Gowers-Cauchy-Schwarz inequality]
\label{Proposition Gowers Cauchy Schwarz}
Let $d$ be a natural number, and, for each $\bo \in \{0,1\}^d$, let $f_{\bo}:[0,1]\longrightarrow \mathbb{C}$ be a bounded measurable function. Define the Gowers inner-product \[\langle (f_{\bo})_{\bo\in \{0,1\}^d}\rangle :=  \int\limits_{(x,\mathbf{h}) \in\mathbb{R}^{d+1}} \prod\limits_{\bo \in \{0,1\}^d}\mathcal{C}^{\vert \boldsymbol{\omega} \vert} f_{\bo}(x+\sum\limits_{i=1}^d h_i\omega_i) \, dx\,dh_1\cdots dh_d.\] Then \[\vert\langle (f_{\bo})_{\bo\in \{0,1\}^d}\rangle \vert \leqslant \prod\limits_{\bo\in \{0,1\}^d} \Vert f_{\bo}\Vert_{U^d(\mathbb{R})}.\]
\end{Proposition}
\begin{proof}
See \cite[Chapter 11]{TaVu06} for the proof in the finite group setting. The modification to the setting of the reals is trivial. 
\end{proof}

\section{Lipschitz functions}
\label{Lipschitz functions}

In the body of the paper we made extensive use of properties of Lipschitz functions. 

\begin{Definition}[Lipschitz functions]
We say that a function $F:\mathbb{R}^m \longrightarrow \mathbb{C}$ is Lipschitz, with Lipschitz constant at most $M$, if \[M \geqslant \sup\limits_{\substack{\mathbf{x},\mathbf{y}\in \mathbb{R}^m \\ \mathbf{x}\neq \mathbf{y}}} \frac{ \vert F(\mathbf{x}) - F(\mathbf{y}) \vert}{\Vert \mathbf{x} - \mathbf{y} \Vert_\infty }.\]

We say that a function $G:\mathbb{R}^m/\mathbb{Z}^m \longrightarrow \mathbb{C}$ is Lipschitz, with Lipschitz constant at most $M$, if \[ M \geqslant \sup\limits_{\substack{\mathbf{x},\mathbf{y}\in \mathbb{R}^m/\mathbb{Z}^m \\ \mathbf{x}\neq \mathbf{y}}} \frac{ \vert G(\mathbf{x}) - G(\mathbf{y}) \vert}{\Vert \mathbf{x} - \mathbf{y} \Vert_{\mathbb{R}^m/\mathbb{Z}^m}}.\]
\end{Definition}

We record the three properties of Lipschitz functions that we will require.

\begin{Lemma}
\label{Lipschitz approximation of convex cutoffs}
Let $N$ be a positive real, let $m$ be a natural number, let $K$ be a convex subset of $[-N,N]^m$, and let $\sigma$ be some parameter in the range $0<\sigma<1/2$. Then there exist Lipschitz functions $F_\sigma, G_\sigma:\mathbb{R}^m\longrightarrow [0,1]$ supported on $[-2N,2N]^m$, both with Lipschitz constant at most $O(\frac{1}{\sigma N})$, such that \[1_K = F_\sigma + O(G_\sigma)\] and $\int_{\mathbf{x}} G_\sigma(\mathbf{x}) \, d\mathbf{x} = O(\sigma N^m)$. Furthermore, $F_\sigma(\mathbf{x})\geqslant 1_K(\mathbf{x})$ for all $\mathbf{x}\in \mathbb{R}^m$, and $G$ is supported on $\{ \mathbf{x} \in \mathbb{R}^m: \dist(\mathbf{x},\partial(K)) \leqslant \sigma N\}$.
\end{Lemma}
\noindent This is \cite[Corollary A.3]{GT10}. It was be used in Lemmas \ref{Lemma replacing F cut-off} and \ref{Lemma making G lipschitz} to replace sums with sharp cut-offs by sums with Lipschitz cut-offs. 

\begin{Lemma}
\label{Fourier transforms of Lipschitz functions on tori}
Let $X$ be a positive real, with $X>2$. Let $F:\mathbb{R}^m/\mathbb{Z}^m\longrightarrow \mathbb{C}$ be a Lipschitz function such that $\Vert F\Vert_\infty \leqslant 1$ and the Lipschitz constant of $F$ is at most $M$. Then
\begin{equation}
F(\mathbf{x}) = \sum\limits_{\substack{\mathbf{k} \in \mathbb{Z}^m \\ \Vert \mathbf{k} \Vert_\infty \leqslant X}} c_X(\mathbf{k}) e( \mathbf{k} \cdot \mathbf{x}) + O\left(M \frac{\log X}{X}\right)
\end{equation}
for every $\mathbf{x} \in\mathbb{R}^m/\mathbb{Z}^m$, for some function $c_X(\mathbf{k})$ satisfying $\Vert c_X(\mathbf{k})\Vert_\infty \ll 1$. (The implied constant in the error term above may depend on the underlying dimensions, as always in this paper).
\end{Lemma}
\noindent This is \cite[Lemma A.9]{GrTa08}, and was used in Lemma \ref{Lemma upper bound involving integral} as a way of bounding the number of solutions to a certain inequality.

\begin{Lemma}
\label{Fourier transforms of Lipschitz functions}
Let $X,N,C$ be positive reals, with $X>2$ and $N>1$. Let $F:\mathbb{R}^m\longrightarrow \mathbb{C}$ be a Lipschitz function, supported on $[-CN,CN]^m$, such that $\Vert F\Vert_\infty \leqslant 1$ and the Lipschitz constant of $F$ is at most $M$. Then
\begin{equation}
F(\mathbf{x}) = \int\limits_{ \substack{\boldsymbol{\xi} \in \mathbb{R}^m \\ \Vert \boldsymbol{\xi} \Vert_\infty \leqslant X} } c_X(\boldsymbol{\xi})e(\frac{\boldsymbol{\xi}\cdot \mathbf{x}}{N}) \, d\boldsymbol{\xi} + O_{C}\left(MN\frac{\log X}{X}\right) 
\end{equation}
\noindent for every $\mathbf{x} \in\mathbb{R}^m$, for some function $c_X(\boldsymbol{\xi})$ satisfying $\Vert c_X(\boldsymbol{\xi})\Vert_\infty \ll_{C} 1$. 
\end{Lemma}

Lemma \ref{Fourier transforms of Lipschitz functions} is very similar to Lemma \ref{Fourier transforms of Lipschitz functions on tori}, and may be easily proved by adapting that standard harmonic analysis argument found in \cite[Lemma A.9]{GrTa08} from  $\mathbb{R}^m/\mathbb{Z}^m$ to $\mathbb{R}^m$. For completeness, we sketch the proof. 

\begin{proof}[Sketch proof]
By rescaling the variable $\mathbf{x}$ by a factor of $N$, we reduce to the case where $F$ is supported on $[-C,C]^m$ and has Lipschitz constant at most $MN$. 

Let \[K_X(\mathbf{x}) : = \prod\limits_{i=1}^m \frac{1}{X} \left( \frac{\sin (\pi X x_i)}{\pi x_i} \right)^2.\] Then \[ \widehat{ K_X}(\boldsymbol{\xi}) = \prod\limits_{i=1}^m \max(1- \frac{\vert \xi_i\vert}{X},0).\] We have \[(F \ast K_X)(\mathbf{x}) = \int\limits_{ \substack{\boldsymbol{\xi} \in \mathbb{R}^m \\ \Vert \boldsymbol{\xi} \Vert_\infty \leqslant X} } \widehat{F}(\boldsymbol{\xi}) \widehat{ K_X}(\boldsymbol{\xi}) e(\boldsymbol{\xi} \cdot \mathbf{x}) \, d\mathbf{x},\] and, since $\vert \widehat{F}(\boldsymbol{\xi})\vert \leqslant \Vert F\Vert_1 \ll_C 1$, letting $c_X(\boldsymbol{\xi}) : = \widehat{F}(\boldsymbol{\xi}) \widehat{ K_X}(\boldsymbol{\xi})$ gives a main term of the desired form. 

It remains to show that \[\Vert F - F\ast K_X\Vert_\infty \ll_{C} MN \frac{\log X}{X}.\] By writing \[ \vert F(\mathbf{x}) - (F\ast K_X)(\mathbf{x})\vert = \Big\vert \int\limits_{\mathbf{y} \in \mathbb{R}^m} (F(\mathbf{x}) - F(\mathbf{y})) K_X(\mathbf{x} - \mathbf{y}) \, d\mathbf{y} \Big\vert,\] one sees that it suffices to show that \[ \int\limits_{\Vert \mathbf{z}\Vert_\infty \leqslant 2C} \Vert \mathbf{z} \Vert_\infty K_X(\mathbf{z}) \, d\mathbf{z} \ll_{C} \frac{\log X}{X}.\] But this bound follows immediately from a dyadic decomposition. 
\end{proof}

We used Lemma \ref{Fourier transforms of Lipschitz functions} extensively in the Generalised von Neumann Theorem argument in Section \ref{section General proof of the real variable von Neumann Theorem}.\\

\section{Rank matrix and normal form: proofs}
\label{section Rank matrix and normal form: Proofs}

In this appendix we prove the two quantitative statements from earlier in the paper, namely Propositions \ref{rank matrix} and \ref{normal form algorithm}.

\begin{Proposition}
\label{sec.Analysis algebra quant}
Let $n$ be a natural number, and let $S = \{f_1,\dots,f_k\}$ be a finite set of continuous functions $f_1,\dots,f_k:\mathbb{R}^n\longrightarrow \mathbb{R}$. Let \[V(S) = \{\mathbf{x} \in \mathbb{R}^n: f_i(\mathbf{x}) = 0 \text{ for all } i\leqslant k\}.\] Suppose that $\mathbf{x}\in \mathbb{R}^n$ is a point with $\Vert \mathbf{x}\Vert_\infty \leqslant C$ and with $\operatorname{dist}(\mathbf{x},V(S))\geqslant c$, for some absolute positive constants $c$ and $C$. Then, there is some $f_j$ such that $\vert f_j(\mathbf{x})\vert = \Omega_{c,C,S}(1)$.
\end{Proposition} 

\begin{proof}
This is nothing more than the fact that every continuous function on a compact set is bounded, applied to the continuous function $\min (1/\vert f_1\vert,\dots,1/\vert f_k\vert )$ and the compact set $\{\mathbf{x} \in \mathbb{R}^n: \Vert \mathbf{x}\Vert_\infty \leqslant C, \,\operatorname{dist}(\mathbf{x},V(S))\geqslant c\}$. 
\end{proof}

From Proposition \ref{sec.Analysis algebra quant} it is easy to deduce the existence of rank matrices, namely Proposition \ref{rank matrix}.

\begin{proof}[\textbf{Proof of \emph{Proposition \ref{rank matrix}}}]
Let $k$ be equal to $\left(\begin{smallmatrix} d\\ m\end{smallmatrix}\right)$, and identify $\mathbb{R}^{md}$ with the space of $m$-by-$d$ real matrices. Then let $f_1,\dots,f_k$ be the $k$ polynomials on $\mathbb{R}^{md}$ that are given by the $k$ determinants of $m$-by-$m$ submatrices of $L$. One then sees that $V_{\rank}(m,d)$ is exactly the set of common zeros of the functions $f_i$. This is since row rank equals column rank, and linear independence of columns in a square matrix can be detected by the determinant. 

Since we assume that $\Vert L\Vert_\infty\leqslant C$ and $\operatorname{dist}(L,V_{\rank}(m,d))\geqslant c$ we can fruitfully apply Proposition \ref{sec.Analysis algebra quant} to deduce that there is some $j$ for which $\vert f_j(L)\vert= \Omega_{c,C}(1)$. The matrix $M$ whose determinant corresponds to the polynomial $f_j$ is exactly the claimed rank matrix. \\

This settles the first part of Proposition \ref{rank matrix}. The second part then follows immediate by the construction of $M^{-1}$ as the adjugate matrix of $M$ divided by $\det M$. \\

The third part, namely the statement about linear combinations of rows, follows quickly from the others. Indeed, without loss of generality, assume that the rank matrix $M$ is realised by columns $1$ through $m$. Then, the fact that the rows of $L$ are linearly independent means that there are unique real numbers $a_i$ such that $\sum\limits_{i=1}^m a_i \lambda_{ij} = v_{j}$ for all $j$ in the range $1\leqslant j\leqslant d$. (Recall that $(\lambda_{ij})_{i\leqslant m,j\leqslant d}$ denotes the coefficients of $L$). Restricting to $j$ in the range $1\leqslant j \leqslant m$, we observe that the $a_i$ are forced to satisfy $$\left(\begin{matrix} 
a_1 \\ 
\vdots \\
a_m
\end{matrix}\right)
 = (M^{T})^{-1}\left(\begin{matrix} 
v_1 \\ 
\vdots \\
v_m
\end{matrix}\right).$$ Since $\Vert (M^{-1})^T\Vert_{\infty} = \Vert M^{-1}\Vert_{\infty} = O_{c,C}(1)$, we conclude that $a_i = O_{c,C,C_1}(1)$ for all $i$. \\

The final part of the proposition is to show that if $\dist(L,V_{\rank}^{\unif}(m,d))\geqslant c$ then, for each $j$, there exists a rank matrix of $L$ that doesn't include the $j^{th}$ column. But this statement follows immediately from the above, after having deleted the $j^{th}$ column.  
\end{proof}

We now prove Proposition \ref{normal form algorithm} on the existence of quantitative normal form parametrisations. We remind the reader that, in the proof, the implied constants may depend on the dimensions of the underlying spaces, namely $m$ and $n$. For the definition of the variety $V_{\mathcal{P}_i}$, which consists of all systems of linear forms for which the partition $\mathcal{P}_i$ is not `suitable', the reader may consult Definition \ref{Definition degeneracy varieties}. The reader may also find the example that follows the proof to be informative. 

\begin{proof}[\textbf{Proof of \emph{Proposition \ref{normal form algorithm}}}]
Fix $i$, and let $\mathcal{P}_i$ be a partition of $[m]\setminus \{i\}$ such that $\operatorname{dist}(\Psi, V_{\mathcal{P}_i})\geqslant c_1$ (such a $\mathcal{P}_i$ exists by the definition of $c_1$-Cauchy-Schwarz complexity, i.e. by Definition \ref{Definition Cauchy Schwarz complexity}). The partition $\mathcal{P}_i$ has $s_i+1$ parts, for some $s_i$ at most $s$. 

It is clear from Definition \ref{Definition Cauchy Schwarz complexity} that we may, without loss of generality, further subdivide the partition and assume that the partition $\mathcal{P}_i$ has exactly $s+1$ parts. Call the parts $\mathcal{C}_1$ through $\mathcal{C}_{s+1}$. 

Following Section 4 of \cite{GT10}, for each $k\in [s+1]$ there exists a vector $\mathbf{f_k}\in \mathbb{R}^n$ that witnesses the fact that $ \operatorname{dist}(\Psi, V_{\mathcal{P}_i})>0$, i.e. for which $\psi_i(\mathbf{f_k}) = 1$ but $\psi_j(\mathbf{f_k})=0$ for all $j\in \mathcal{C}_k$. Such a vector can be found using Gaussian elimination, say. Consider the extension $$\Psi^\prime(\mathbf{u}, w_1,\dots,w_{s+1}): = \Psi(\mathbf{u} + w_1\mathbf{f_1}+ \cdots + w_{s+1}\mathbf{f_{s+1}}).$$ Then, if $\Psi^\prime = (\psi_1^\prime, \dots, \psi_m^\prime)$, the form $\psi^\prime_i(\mathbf{u},w_1,\dots,w_{s+1})$ is the only one that uses all of the $w_k$ variables. Furthermore, $\psi_i^\prime (\mathbf{0},\mathbf{w})= w_1 + \cdots + w_{s+1}$. Also, $n^\prime = n+ s+1$, which is at most $n+m-1$. So Proposition \ref{normal form algorithm} is proved if for each $k$ we can find such a vector $\mathbf{f_k}$ that additionally satisfies $\Vert\mathbf{f_k}\Vert_\infty = O_{c_1,C_1}(1)$. 

Consider a fixed $k$, and let $\Gamma$ be the set of possible implementations of Gaussian elimination on the set of forms $\psi_i \cup \{\psi_j: j\in \mathcal{C}_k\}$ to find a solution vector $\mathbf{f_k}$. If in the course of implementing these algorithms we are given a free choice for a co-ordinate of $\mathbf{f_k}$, we set it to be equal to zero. Note that $\vert \Gamma\vert = O(1)$.

Now, for each $\gamma\in \Gamma$, let the rational functions $$\frac{p_{\gamma,1}(\Psi)}{q_{\gamma,1}(\Psi)},\dots,\frac{p_{\gamma,n}(\Psi)}{q_{\gamma,n}(\Psi)}$$ be the $n$ rational functions defining the claimed coefficients of $\mathbf{f_k}$. One may assume without loss of generality that, for all $j$, we have $p_{\gamma,j},q_{\gamma,j}\in \mathbb{Z}[X_1,\dots,X_n]$ with coefficients of size $O(1)$. Now let $$Q_\gamma := \prod\limits_{j\leqslant n}q_{\gamma,j}.$$ We claim that $V(I)\subseteq V_{\mathcal{P}_i}$, where $I$ is the ideal generated by the set of polynomials $\{Q_\gamma:\gamma\in \Gamma\}$ and $V(I)$ is the affine algebraic variety generated by $I$. Indeed, if $Q_{\gamma}(\Psi)=0$ for all $\gamma\in \Gamma$ then there is no Gaussian elimination implementation that finds a solution $\mathbf{f_k}$, and this in turn implies that $\mathcal{P}_i$ is not suitable for $\Psi$ and hence that $\Psi \in V_{\mathcal{P}_i}$.

 Since $V(I)\subseteq V_{\mathcal{P}_i}$, the assumptions of Proposition \ref{normal form algorithm} imply that $\operatorname{dist}(\Psi,V(I))\geqslant c_1$. Applying Proposition \ref{sec.Analysis algebra quant} to the polynomials $\{Q_\gamma:\gamma \in\Gamma\}$, we conclude that there is some $\gamma\in\Gamma$ such that $\vert Q_\gamma(\Psi)\vert = \Omega_{c_1,C_1}(1)$. In particular, we conclude that the solution vector $\mathbf{f_k}$ obtained by the implementation $\gamma$ has coefficients that are $O_{c_1,C_1}(1)$. This concludes the proof of Proposition \ref{normal form algorithm}.
\end{proof}
 
Let us illustrate the above proof with an example which we hope will be instructive. Consider $n=3$, $m=2$, $i=1$, and denote $$\Psi = \left(\begin{matrix}
a_{11} & a_{12} & a_{13}\\
a_{21} & a_{22} & a_{23}
\end{matrix}\right).$$ Then the partition $\mathcal{P}_i$ consists of the singleton $\{2\}$, and suppose one wished to construct a suitable $\mathbf{f_1}$ simply by applying Gaussian elimination. Implementing the algorithm a certain way we have $$\mathbf{f_1} = \left(\begin{matrix}
a_{22}/(a_{11}a_{22} - a_{12}a_{21})\\
-a_{21}/(a_{11}a_{22} - a_{12}a_{21})\\
 0
 \end{matrix}\right)
 $$ as a solution, in the case where $a_{11}a_{22} - a_{12}a_{21}$ is non-zero. Of course if $a_{11}a_{23}-a_{13}a_{21}$ is non-zero too, we have another solution 
 $$\mathbf{f_1} = \left(\begin{matrix}
 a_{23}/(a_{11}a_{23} - a_{13}a_{21})\\
   0\\
   -a_{21}/(a_{11}a_{23} - a_{13}a_{21})
  \end{matrix}\right).
  $$ So, if one applied Gaussian elimination idly, one might end up with either of these two solutions. Unfortunately it could be the case that $\operatorname{dist}(\Psi,V_{\mathcal{P}_i})\geqslant c_1$ whilst one of these determinants, $a_{11}a_{22} - a_{12}a_{21}$ say, was non-zero yet $o(1)$ (as the unseen variable $N$, on which $\Psi$ will ultimately depend, tends to infinity). In this instance, applying the first implementation of the algorithm would not give a desirable solution vector $\mathbf{f_1}$. For this reason we had to apply somewhat indirect arguments in order to find the appropriate vector $\mathbf{f_1}$.\\

It is worth including a brief discussion on why these quantitative subtleties do not arise in the setting of \cite{GT10}. Indeed, assume that $\Psi$ has rational coefficients of naive height at most $C_1$ and that $\Psi \notin V_{\mathcal{P}_i}$. Since there are only $O_{C_1}(1)$ many possible choices of $\Psi$ we immediately conclude that $\dist(\Psi,V_{\mathcal{P}_i}) \gg_{C_1}1$, without needing to assume this as an extra hypothesis. Then \emph{any} implementation of Gaussian elimination succeeds in finding a suitably bounded $\mathbf{f_k}$, since one is only ever working with rationals of bounded height.

\section{Additional linear algebra}
\label{section additional linear algebra}
In this appendix, we collect together the assortment of standard linear algebra lemmas that we used at various points throughout the paper. We also give the linear algebra argument that we used to construct the matrix $P$ during the proof of Lemma \ref{Lemma generating a purely irrational map}. \\

This first lemma demonstrates the intuitive fact, that if $L:\mathbb{R}^d \longrightarrow \mathbb{R}^m$ is a linear map then $L: (\ker L)^\perp\longrightarrow \mathbb{R}^m$ has bounded inverse. 
\begin{Lemma}
\label{nonsingularity of map on orthogonal complement of Kernel}
Let $m,d$ be natural numbers, with $d\geqslant m+1$, and let $c,C,l$ be positive constants. Let $L:\mathbb{R}^d\longrightarrow \mathbb{R}^m$ be a surjective linear map, and suppose $\Vert L\Vert_\infty\leqslant C$ and $\operatorname{dist}(L,V_{\rank}(m,d))\geqslant c$. Let $K$ denote $\ker L$. Let $R$ be a convex set contained in $[-l,l]^m$. Then, if $\mathbf{v}\in K^\perp$, $L\mathbf{v}\in R$ only when $\mathbf{v}\in R^\prime$, where $R^\prime$ is some convex region that satisfies $R^\prime\subseteq [-O_{c,C}(l),O_{c,C}(l)]^d$.
\end{Lemma}
\begin{proof}
We choose to prove this statement using the concept of the 'rank matrix' introduced earlier. Writing $L$ as a $m$-by-$d$ matrix with respect to the standard bases, let $\boldsymbol{\lambda_i} \in \mathbb{R}^d$ denote the column vector such that $\boldsymbol{\lambda_i}^T$ is the $i^{th}$ row of $L$. Since $\operatorname{dist}(L,V_{\rank}(m,d))\geqslant c$, the vectors $\boldsymbol{\lambda_i}$ are linearly independent. Moreover, we may extend the set $\{\boldsymbol{\lambda_i}:i\leqslant m\}$ by orthogonal vectors of unit length to form a basis $\{\boldsymbol{\lambda_i}:i\leqslant d\}$ for $\mathbb{R}^d$. 

We claim that for all $k\in [d]$ we have $$\sum\limits_{i=1}^{d}a_{ki}\boldsymbol{\lambda_i} = \mathbf{e_k},$$ for some coefficients $a_{ki}$ satisfying $\vert a_{ki}\vert = O_{c,C}(1)$, where $\mathbf{e_k}\in\mathbb{R}^d$ is the $k^{th}$ standard basis vector. Indeed, fix $k$, and note that $\mathbf{e_k} = \mathbf{x_k} + \mathbf{y_k}$, where $\mathbf{x_k} \in \operatorname{span}(\boldsymbol{\lambda_i}:i\leqslant m)$ and $\mathbf{y_k}\in \operatorname{span}(\boldsymbol{\lambda_i}:m+1\leqslant i\leqslant d)$. The vectors $\mathbf{x_k}$ and $\mathbf{y_k}$ are orthogonal by construction, so in particular $\Vert \mathbf{x_k}\Vert_2^2+ \Vert \mathbf{y_k}\Vert_2^2 = 1$, and hence $\Vert \mathbf{x_k}\Vert_{\infty}, \Vert \mathbf{y_k}\Vert_{\infty}\ll 1$. By the third part of Proposition \ref{rank matrix} applied to $\mathbf{x_k}$ we get $\vert a_{ki}\vert = O_{c,C}(1)$ when $i\leqslant m$, and the orthonormality of $\{\boldsymbol{\lambda_i}:m+1\leqslant i\leqslant d\}$ implies that $\vert a_{ki}\vert = O(1)$ when $i$ is in the range $m+1\leqslant i\leqslant d$. 

Now notice that $\operatorname{span}(\boldsymbol{\lambda_i}:m+1\leqslant i\leqslant d)$ is exactly equal to $K$. Let $\mathbf{v}\in K^\perp$, and suppose $L\mathbf{v}\in R$. Letting $L^\prime$ be the $d$-by-$d$ matrix whose rows are $\boldsymbol{\lambda_i}^T$, we have that $L^\prime \mathbf{v} = \mathbf{w}$ for some vector $\mathbf{w}$ satisfying $\Vert \mathbf{w}\Vert_{\infty}\ll l $. Pre-multiplying by the matrix $A = (a_{ki})$, we immediately get $\mathbf{v} = A\mathbf{w}$, and hence $\Vert \mathbf{v}\Vert_{\infty} = O_{c,C}(l)$. The region $R^\prime: = (L^{-1} R)\cap K^{\perp}$ is therefore bounded. $R^\prime$ is clearly convex, and so the lemma is proved. \qedhere 
\end{proof}

The second lemma concerns vectors, with integer coordinates, that lie near to a subspace. 
\begin{Lemma}
\label{Lemma integer distance}
Let $h,d$ be natural numbers, with $h\leqslant d$, and let $C,\eta$ be positive reals. Let $\Xi:\mathbb{R}^h \longrightarrow \mathbb{R}^d$ be an injective linear map, with $\Vert \Xi\Vert_\infty \leqslant C$.  Suppose further that $\Xi(\mathbb{Z}^h) = \mathbb{Z}^d \cap \Xi(\mathbb{R}^h)$. Let $\mathbf{n},\widetilde{\mathbf{r}}\in \mathbb{Z}^d$. Suppose that 
\begin{equation}
\label{equation contradict integer distance}
 \dist(\mathbf{n} , \Xi(\mathbb{R}^h) + \widetilde{\mathbf{r}}) \leqslant \eta.
 \end{equation}
 Then, if $\eta$ is small enough in terms of $C$, $h$ and $d$, $\mathbf{n} = \Xi(\mathbf{m})+ \widetilde{\mathbf{r}}$, for some unique $\mathbf{m} \in\mathbb{Z}^h$.    
\end{Lemma}
\begin{proof}
By replacing $\mathbf{n}$ with $\mathbf{n} - \widetilde{\mathbf{r}}$, we can assume without loss of generality that $\widetilde{\mathbf{r}} = \mathbf{0}$. It will also be enough to show that $\mathbf{n} \in \Xi(\mathbb{R}^h)$, as the injectivity of $\Xi$ and the assumption that $\Xi(\mathbb{Z}^h) = \mathbb{Z}^d \cap \Xi(\mathbb{R}^h)$ immediately go on to imply the existence of a unique $\mathbf{m}$.

Suppose for contradiction then that $\mathbf{n} \notin \Xi(\mathbb{R}^h)$. In matrix form, $\Xi$ is a $d$-by-$h$ matrix with linearly independent columns, all of whose coefficients are integers with absolute value at most $C$. We can extend this matrix to a $d$-by-$d$ matrix $\widetilde{\Xi}$, with linearly independent columns, all of whose coefficients are integers with absolute value at most $C$. Then $(\widetilde{\Xi})^{-1}$ is a $d$-by-$d$ matrix with rational coefficients of naive height at most $C^{O(1)}$, and $(\widetilde{\Xi})^{-1}(\Xi(\mathbb{R}^{h})) = \mathbb{R}^h \times \{0\}^{d-h}$. 

Since $\mathbf{n} \notin \Xi(\mathbb{R}^h)$, we have $(\widetilde{\Xi})^{-1}(\mathbf{n})\notin \mathbb{R}^h \times \{0\}^{d-h}$.  But $(\widetilde{\Xi})^{-1}(\mathbf{n}) \in \frac{1}{K} \mathbb{Z}^d$, for some natural number $K$ satisfying $K = O(C^{O(1)})$. Therefore \[ \dist ( (\widetilde{\Xi})^{-1}(\mathbf{n}),(\widetilde{\Xi})^{-1}(\Xi(\mathbb{R}^{h}))) \gg C^{-O(1)}.\] Applying $\widetilde{\Xi}$, we conclude that \[ \dist(\mathbf{n} , \Xi(\mathbb{R}^h)) \gg C^{-O(1)},\] which is a contradiction to (\ref{equation contradict integer distance}) if $\eta$ is small enough. 
\end{proof}

The construction of the matrix $\widetilde{\Xi}$ in the above proof also has an even more basic consequence, namely that $\Xi^{-1}:\im \Xi \longrightarrow \mathbb{R}^h$ is bounded. 
\begin{Lemma}
\label{Lemma bounded inverse}
Let $h,d$ be natural numbers, with $h\leqslant d$, and let $C,\eta$ be positive reals. Suppose that $\Xi:\mathbb{R}^h \longrightarrow \mathbb{R}^d$ is an injective linear map, with $\Vert \Xi\Vert_\infty \leqslant C$. Suppose further that $\Xi(\mathbb{Z}^h) \subseteq \mathbb{Z}^d \cap \Xi(\mathbb{R}^h)$. Then if $\Vert \Xi(\mathbf{y})\Vert_\infty \leqslant \eta$, we have $\Vert \mathbf{y} \Vert_\infty \ll C^{-O(1)} \eta$. 
\end{Lemma}
\begin{proof}
Construct the matrix $\widetilde{\Xi}$ as in the previous proof. Then $\Vert(\widetilde{\Xi})^{-1} (\Xi(\mathbf{y}))\Vert_\infty \ll C^{O(1)} \eta$, by the bound on the size of the coefficients of $\widetilde{\Xi}$. But $(\widetilde{\Xi})^{-1} (\Xi(\mathbf{y})) \in\mathbb{R}^{d}$ is nothing more than the vector $\mathbf{y} \in\mathbb{R}^h$ extended by zeros. So $\Vert \mathbf{y} \Vert_\infty \ll C^{O(1)} \eta$ as claimed. 
\end{proof}

Finally, we give the linear algebra argument used to construct the matrix $P$ during the proof of Lemma \ref{Lemma generating a purely irrational map}. 
\begin{Lemma}
\label{Lemma construction of P}
Let $m,d$ be natural numbers, with $d\geqslant m+1$. Let $L:\mathbb{R}^{d} \longrightarrow \mathbb{R}^m$ be a surjective linear map with rational dimension $u$, and let $\Theta:\mathbb{R}^{m} \longrightarrow \mathbb{R}^u$ be a rational map for $L$. Suppose that $\Vert L\Vert_\infty \leqslant C$ and $\Vert \Theta \Vert_\infty \leqslant C$. Equating $L$ with its matrix, suppose that the first $m$ columns of $L$ form the identity matrix. Let $\{\mathbf{a_1},\dots,\mathbf{a_u}\}$ be a basis for the lattice $\Theta L(\mathbb{Z}^d)$ that satisfies $\Vert \mathbf{a_i}\Vert_\infty = O_{C}(1)$ for every $i$. Let $\mathbf{x_1},\dots,\mathbf{x_u}\in\mathbb{Z}^d$ be vectors such that, for every $i$, $\Theta L(\mathbf{x_i}) = \mathbf{a_i}$ and $\Vert \mathbf{x_i}\Vert_\infty = O_{C}(1)$. Then \begin{equation}
\label{direct sum chapter 3}
\mathbb{R}^m = \spn(L\mathbf{x_i}:i\leqslant u) \oplus \ker \Theta
\end{equation} and there is an invertible linear map $P:\mathbb{R}^m \longrightarrow \mathbb{R}^m$ such that
\begin{align}
P((\spn(L\mathbf{x_i}:i\leqslant u))) &= \mathbb{R}^u \times \{0\}^{m-u},\nonumber \\
P(\ker \Theta) &= \{0\}^{u} \times \mathbb{R}^{m-u}, \nonumber
\end{align}
\noindent and both $\Vert P\Vert_\infty = O_{C}(1)$ and $\Vert P^{-1} \Vert_\infty = O_{C}(1)$. 
\end{Lemma}
\noindent Note that both $\{\mathbf{a_1},\dots,\mathbf{a_u}\}$ and $\mathbf{x_1},\dots,\mathbf{x_u}\in\mathbb{Z}^d$ exist by applying Lemma \ref{Lemma parametrising the image lattice} to the map $S: = \Theta L$.

\begin{proof}
The expression (\ref{direct sum chapter 3}) is immediate from the definitions, so it remains to construct $P$. We may assume, since the first $m$ columns of $L$ form the identity matrix, that $\Theta$ has integer coefficients. 

As $\Vert \Theta\Vert_\infty = O_{C}(1)$, we may pick a basis $\{\mathbf{y_1},\dots,\mathbf{y_{m-u}}\}$ for $\ker \Theta$ in which $\mathbf{y_j} \in \mathbb{Z}^m$ and $\Vert \mathbf{y_j}\Vert_\infty = O_{C}(1)$ for all $j$. Let $\mathbf{b_1},\dots,\mathbf{b_m}$ denote the standard basis of $\mathbb{R}^m$, and define $P$ by letting
\begin{align}
\label{definition of P}
P(L\mathbf{x_i}) &:= \mathbf{b_i}, \qquad \quad 1\leqslant i\leqslant u \nonumber \\
P(\mathbf{y_j}) &:= \mathbf{b_{j+u}}, \qquad 1\leqslant j\leqslant m-u,
\end{align}
\noindent and then extending linearly to all of $\mathbb{R}^m$. Clearly $P((\spn(L\mathbf{x_i}:i\leqslant u))) = \mathbb{R}^u \times \{0\}^{m-u}$ and $P(\ker \Theta) = \{0\}^{u} \times \mathbb{R}^{m-u}$. It is also immediate that $\Vert P^{-1} \Vert_\infty = O_{C}(1)$, since $\Vert L\mathbf{x_i}\Vert_\infty = O_{C}(1)$ and $\Vert \mathbf{y_j}\Vert_\infty = O_{C}(1)$ for all $i$ and $j$. It remains to bound $\Vert P\Vert_\infty$. If $L\mathbf{x_i}$ were all vectors with integer coordinates then this bound would be immediate as well, as then $P^{-1}$ would have integer coordinates and hence $\vert \det P^{-1}\vert \geqslant 1$. As it is, we have to proceed more slowly. 

To this end, for a standard basis vector $\mathbf{b_k}$ write \[ \mathbf{b_k} = \sum\limits_{i=1}^u \lambda_i L\mathbf{x_i} + \sum\limits_{j=1}^{d-u}  \mu_j \mathbf{y_j}.\] It will be enough to show that $\vert \lambda_i\vert,\vert \mu_j\vert = O_{C}(1)$ for all $i$ and $j$. First note that, since the first $m$ columns of $L$ form the identity, $\mathbf{b_k} \in L(\mathbb{Z}^d)$. Also $\Theta (\mathbf{b_k}) = \sum_{i=1}^u \lambda_i\mathbf{a_i}$. So $\mathbf{a} : = \sum_{i=1}^u \lambda_i\mathbf{a_i}$ is an element of $\Theta L(\mathbb{Z}^d)$ that satisfies $\Vert \mathbf{a}\Vert_\infty  = O_{C}(1)$. Since $\Vert \mathbf{a_i}\Vert_\infty = O_{C}(1)$ for every $i$, and $\{\mathbf{a_1},\dots,\mathbf{a_u}\}$ is a basis for the lattice $\Theta L(\mathbb{Z}^d)$, this implies that $\vert \lambda_i\vert = O_{C}(1)$ for every $i$. 

So then $\sum_{j=1}^{d-u}  \mu_j \mathbf{y_j}$ is a vector in $\ker \Theta$ satisfying $\Vert \sum_{j=1}^{d-u}  \mu_j \mathbf{y_j}\Vert_\infty = O_{C}(1)$. Since $\{\mathbf{y_1},\dots,\mathbf{y_{m-u}}\}$ is a set of linearly independent vectors, each of which has integer coordinates with absolute value $O_{C}(1)$, this implies that $\vert \mu_j\vert = O_{C}(1)$ for every $j$. 

Therefore $P$ satisfies the conclusions of the lemma.
\end{proof}

\begin{Remark}
\label{Remark about P}
\emph{We note the effects of the above construction in the case when $L$ has algebraic coefficients. We use a rudimentary version of height: if $Q\in \mathbb{Z}[X]$ we define \[H(Q):= \max(\vert q_i\vert: q_i\text{ a coefficient of }  Q)\] to be the \emph{height} of $Q$, and we say that the height of an algebraic number is the height of its minimal polynomial. (So there are $O_{k,H}(1)$ algebraic numbers of degree at most $k$ and height at most $H$). Then, if in the statement of Lemma \ref{Lemma construction of P} all the coefficients of $L$ are algebraic numbers with degree at most $k$ and height at most $H$, all the coefficients of $P$ are algebraic numbers of degree $O_{k}(1)$ and height $O_{C,k,H}(1)$.}
\end{Remark}

\section{The approximation function in the algebraic case}
\label{section algebraic approximation}
We use this final appendix to give the proof of relation (\ref{approximation function in the algebraic case}). The following lemma makes this relation quantitatively precise. 

\begin{Lemma}
\label{Lemma algebraic coeffs implies gen appr}
Let $m,d$ be natural numbers, with $d\geqslant m+1$, and let $c,C$ be positive constants. Let $L:\mathbb{R}^d \longrightarrow \mathbb{R}^m$ be a surjective linear map, and suppose that the matrix of $L$ has algebraic coefficients of algebraic degree at most $k$ and algebraic height at most $H$ (see Remark \ref{Remark about P} for definitions). Suppose that $\Vert L\Vert_\infty \leqslant C$, that $\dist(L,V_{\rank}(m,d)) \geqslant c$, and that $L$ has rational complexity at most $C$. Let $\tau_1,\tau_2$ be two parameters in the range $0<\tau_1,\tau_2\leqslant 1$. Then \[ A_L(\tau_1,\tau_2) \gg_{k,H,c,C} \min(\tau_1, \tau_2^{O_{k}(1)}).\] 
\end{Lemma}

\begin{proof}
We begin by reducing to the case when $L$ is purely irrational. Indeed, consider Lemma \ref{Lemma generating a purely irrational map} and replace $L$ by the map $L^\prime$ (expression (\ref{equation definition of L prime})). By part (9) of Lemma \ref{Lemma generating a purely irrational map}, $A_{L^\prime}(\tau_1,\tau_2)\ll_{c,C} A_L(\Omega_{c,C}(\tau_1),\Omega_{c,C}(\tau_2))$. Also, using Remark \ref{Remark about P}, it follows that $L^\prime$ has algebraic coefficients of algebraic degree at most $O_{k}(1)$ and algebraic height at most $O_{c,C,k,H}(1)$.  So, replacing $L$ with $L^\prime$, without loss of generality we may assume that $L$ is purely irrational. \\

Suppose for contradiction that for all choices of constants $c_1$ and $C_2$, there exist parameters $\tau_1$ and $\tau_2$ such that $A_L(\tau_1,\tau_2) < c_1 \min(\tau_1, \tau_2^{C_2})$, i.e. there exists a map $\alpha \in (\mathbb{R}^m)^*$ and a map $\varphi \in (\mathbb{Z}^d)^T$ such that  $\tau_1\leqslant \Vert \alpha\Vert_\infty \leqslant \tau_2^{-1}$ and 
\begin{equation}
\label{first of lemma alg coeffs}
\Vert L^*\alpha - \varphi\Vert_\infty <c_1 \min(\tau_1,\tau_2^{C_2}).
\end{equation}
\noindent Fix $\alpha$ and $\varphi$ so that they satisfy (\ref{first of lemma alg coeffs}). We will obtain a contradiction if $c_1$ is small enough in terms of $c,C,k,H$, and if and $C_2$ is large enough in terms of $k$. \\

In the first part of the proof, we apply various reductions to enable us to replace $\alpha$ with a map that has integer coordinates with respect to the standard dual basis of $(\mathbb{R}^m)^*$.

Let $M$ be a rank matrix of $L$ (Proposition \ref{rank matrix}), and assume without loss of generality that $M$ consists of the first $m$ columns of $L$. Then there exists a map $\beta \in (\mathbb{R}^m)^*$, namely $\beta: = M^* \alpha$, such that $\tau_1 \ll_{c,C} \Vert \beta\Vert_\infty \ll_{c,C} \tau_2^{-1}$ and
\begin{equation}
\label{equation before switching to integers}
 \Vert L^* (M^{-1})^* \beta - \varphi\Vert_\infty < c_1 \min(\tau_1,\tau_2^{C_2}).
 \end{equation}
 Since the first $m$ columns of $M^{-1} L$ form the identity matrix, (\ref{equation before switching to integers}) implies that
\begin{equation}
\label{equation beta integer norm bound}
\dist(\beta, (\mathbb{Z}^m)^T) < c_1 \min(\tau_1,\tau_2^{C_2}).
\end{equation} 

We know that $\Vert \beta \Vert_\infty = \Omega_{c,C}( \tau_1)$. Also, considering (\ref{equation beta integer norm bound}), by perturbing $\beta$ by a suitable element $\gamma \in (\mathbb{R}^m)^*$ with $\Vert \gamma\Vert_\infty < c_1 \min(\tau_1,\tau_2^{C_2})$ we may obtain a map $\rho\in(\mathbb{Z}^{m})^T$. Combining these facts, note how  
\begin{align*}
\Vert \rho\Vert_\infty &\geqslant \Vert \beta \Vert_\infty -  c_1 \min(\tau_1,\tau_2^{C_2})\\
&\gg_{c,C} \tau_1
\end{align*} if $c_1$ is small enough, and so certainly $\rho \neq 0$.

From (\ref{equation before switching to integers}), we therefore conclude that there exists some $\rho \in (\mathbb{Z}^m)^T\setminus \{0\}$, satisfying $\Vert \rho\Vert_\infty = O_{c,C}(\tau_2^{-1})$, such that \begin{equation}
\label{expressions which will contradict algebraicity}
 \Vert L^* (M^{-1})^*\rho - \varphi \Vert_\infty < c_1 C_3\tau_2^{C_2}
\end{equation} where $C_3$ is some constant that depends on $c$ and $C$. Referring back to (\ref{first of lemma alg coeffs}), we see that we have achieved our goal of replacing $\alpha$ with a map that has integer coefficients.\\

Expression (\ref{expressions which will contradict algebraicity}) leads to a contradiction. Morally this follows from Liouville's theorem on the diophantine approximation of algebraic numbers, but we couldn't find exactly the statement we needed in the literature, so we include a short argument here.

Indeed, let $\varphi = (\begin{matrix} \varphi_1 & \cdots & \varphi_d\end{matrix})$ be the representation of $\varphi$ with respect to the standard dual basis of $(\mathbb{R}^d)^*$ (with analogous notation for $L^* (M^{-1})^* \rho$). Since $L$ is assumed to be purely irrational, so is $M^{-1}L$. Therefore, since $\rho:\mathbb{R}^m \longrightarrow \mathbb{R}$ is surjective (since it is non-zero), we may pick some co-ordinate $i$ at most $d$ for which $(L^* (M^{-1})^* \rho)_i - \varphi_i \neq 0$. So there are algebraic numbers $\lambda_1,\dots,\lambda_{m}$ with algebraic degree $O_{k}(1)$ and algebraic height $O_{c,C,k,H}(1)$ for which 
\begin{equation}
\label{single coordinate bound}
0<\vert \sum\limits_{j=1}^{m}\lambda_j\rho_j - \varphi_i\vert< c_1C_3 \tau_2^{C_2},
\end{equation} where $(\begin{matrix} \rho_1 & \cdots & \rho_m\end{matrix})$ is the representation of $\rho$ with respect to the standard dual basis. Note that if $c_1$ is small enough, by (\ref{single coordinate bound}) and the fact that $\Vert\rho\Vert_\infty = O_{c,C} (\tau_2^{-1})$ one has $\vert \varphi_i\vert  = O_{c,C} (\tau_2^{-1})$. \\

Our aim will be to find a suitable polynomial $Q$ for which $Q(\sum_{j \leqslant m} \lambda_j\rho_j) = 0$, and then to apply Liouville's original argument.\\

Assume without loss of generality that each $\lambda_j\rho_j$ is non-zero. For each $j$ at most $m$, let $Q_j\in\mathbb{Z}[X]$ denote the minimal polynomial of $\lambda_j\rho_j$. Note that the degree of $Q_j$ is $O_{k}(1)$ (since $\rho_j\in\mathbb{Z}$). By the bounds on the degree and height of $\lambda_j$, and since $\Vert \rho\Vert_\infty  = O_{c,C} (\tau_2^{-1})$, we have $H(Q_j) = O_{c,C,k,H}(\tau_2^{-O_k(1)})$. 

By using the standard construction based on resultants (see \cite[Section 4.2.1]{Co93}), this implies that there is a polynomial $Q\in \mathbb{Z}[X]$ with degree $O_k(1)$ such that\\ $Q(\sum_{j \leqslant m} \lambda_j\rho_j) = 0$ and $H(Q) = O_{c,C,k,H}( \tau_2^{-O_k(1)})$. 

Now, it could be that $\varphi_i$ is a root of $Q$. If this is the case, we use the factor theorem and Gauss' Lemma to replace $Q$ by the integer-coefficient polynomial $Q\cdot (X-\varphi_i)^{-1}$. In this case, $H( Q\cdot(X-\varphi_i)^{-1})\ll_{c,C,k,H} (\varphi_i + 1)^{O_k(1)}\tau_2^{-O_k(1)}$. By repeating this process as necessary, since $\vert \varphi_i\vert  = O_{c,C} (\tau_2^{-1})$ we may assume therefore that $\varphi_i$ is not a root of $Q$.\\

This immediately implies a bound on the derivative of $Q$, namely that, for any $\theta$, \[\vert Q^\prime (\theta)\vert \ll_{c,C,k,H} \tau_2^{-O_k(1)}\sum\limits_{0\leqslant a \leqslant O_k(1)} \theta^a.\] But then the mean value theorem implies that for some $\theta$ in the interval \\$[\sum_j\lambda_j\alpha_j, \varphi_i]$ one has \[1\leqslant \vert Q(\varphi_i)\vert = \vert Q(\sum\limits_{j=1}^{m}\lambda_j\rho_j) - Q(\varphi_i)\vert \leqslant  \vert Q^\prime (\theta)\vert \vert \sum\limits_{j=1}^{m}\lambda_j\rho_j - \varphi_i\vert\ll_{c,C,k,H} c_1C_3 \tau_2^{-O_k(1)} \tau_2^{C_2}. \] If $C_2$ is large enough in terms of $k$, this implies that $c_1 = \Omega_{c,C,k,H} (1)$ , which is a contradiction if $c_1$ is small enough. Therefore the lemma holds. 
\end{proof}

\bibliographystyle{plain}
\bibliography{Gowers}
\end{document}